\documentclass[oneside]{amsart}

\usepackage[english]{babel}
\usepackage[utf8]{inputenc}
\usepackage[T1]{fontenc}
\usepackage{lmodern}

\DeclareSymbolFont{CMlargesymbols}{OMX}{cmex}{m}{n}
\DeclareMathAccent{\widehat}{\mathord}{CMlargesymbols}{"62}

\usepackage{tensor}

\usepackage[colorlinks=true, allcolors=green]{hyperref}

\usepackage[hyperref=true,backref=true,backend=biber,style=alphabetic,sorting=nyt,citestyle=alphabetic]{biblatex}
\usepackage{amsfonts}               
\usepackage{amsmath}
\usepackage{amssymb}                 
\usepackage{amsthm} 
\usepackage[labelfont=bf]{caption} 
\usepackage{csquotes}
\usepackage{enumitem} 
\usepackage{mathtools} 
\usepackage{nicefrac}       
\usepackage{textcomp} 
\usepackage{url}            
\usepackage{verbatim} 
\usepackage{xcolor} 
\usepackage{euscript}
\usepackage{wrapfig}
\usepackage{accents}
\usepackage{bbm}
\usepackage{todonotes}

\makeatletter
\g@addto@macro\bfseries{\boldmath}
\makeatother

\makeatletter
\let\orig@maketitle\maketitle
\renewcommand{\maketitle}{
  \begingroup
    \def\@addpunct##1{}
    \orig@maketitle
  \endgroup
}
\makeatother

\hypersetup{
    colorlinks,
    linkcolor={red!50!black},
    citecolor={blue!50!black},
    urlcolor={blue!80!black}
}

\DeclarePairedDelimiter\abs{\lvert}{\rvert}%
\DeclarePairedDelimiter\norm{\lVert}{\rVert}%

\makeatletter
\let\oldabs\abs
\def\abs{\@ifstar{\oldabs}{\oldabs*}}
\let\oldnorm\norm
\def\norm{\@ifstar{\oldnorm}{\oldnorm*}}
\makeatother

\def\A{\EuScript A}
\def\B{\EuScript B}

\def\F{\mathcal F}
\def\U{\mathcal U}
\newcommand{\C}{\mathcal{C}}

\def\IFF{\Leftrightarrow}

\renewcommand{\a}{\alpha}
\renewcommand{\b}{\beta}
\newcommand{\g}{\gamma}

\renewcommand{\d}{\delta}
\renewcommand{\k}{\kappa}
\renewcommand{\l}{\lambda}

\renewcommand{\phi}{\varphi}
\renewcommand{\o}{\omega}

\newcommand{\RR}{\mathbb{R}}

\newcommand{\GG}{\mathbb{G}}
\newcommand{\pre}[2]{\tensor[^{#1}]{{#2}}{}}
\newcommand{\seq}[2]{\langle #1 \mid #2 \rangle}

\newcommand{\Ord}{\mathrm{Ord}}
\newcommand{\Card}{\mathrm{Card}}

\newcommand{\dom}{\operatorname{dom}}
\newcommand{\ran}{\operatorname{ran}}
\newcommand{\cof}{\operatorname{cof}}

\newcommand{\id}{\operatorname{id}}

\newcommand{\leng}[1]{\operatorname{length}(#1)}

\newcommand{\weight}{\operatorname{w}}

\DeclareMathOperator{\conc}{\mathbin{ {}^\smallfrown {} }}

\newcommand{\markdef}[1]{\textbf{#1}}

\newcommand{\ZFC}{{\sf ZFC}}

\newtheorem{theorem}{Theorem}[section]
\newtheorem*{theorem*}{Theorem}
\newtheorem{lemma}[theorem]{Lemma}
\newtheorem{cor}[theorem]{Corollary}
\newtheorem{corollary}[theorem]{Corollary}
\newtheorem*{corollary*}{Corollary}
\newtheorem*{lemma*}{Lemma}
\newtheorem{prop}[theorem]{Proposition}
\newtheorem{proposition}[theorem]{Proposition}

\newtheorem{problem}{Problem}

\newtheorem{fact}[theorem]{Fact}
\newtheorem{claim}{Claim}[theorem]
\theoremstyle{definition}
\newtheorem{defin}[theorem]{Definition}
\newtheorem{definition}[theorem]{Definition}

\theoremstyle{remark}
\newtheorem{remark}[theorem]{Remark}

\newcommand{\Siii}[3]{{ #3 }\text{-}\boldsymbol{\Sigma}^{ #1 }_{ #2}}
\newcommand{\Piii}[3]{{ #3 }\text{-}\boldsymbol{\Pi}^{ #1 }_{ #2 }}
\newcommand{\Deee}[3]{{ #3 }\text{-}\boldsymbol{\Delta}^{ #1 }_{ #2}}
\newcommand{\Bor}[1]{{ #1 }\text{-}\mathbf{Bor}}

\numberwithin{equation}{section}

\DeclareMathOperator{\ord}{ord}
\DeclareMathOperator{\rank}{rk}
\DeclareMathOperator{\supp}{supp}
\DeclareMathOperator{\powset}{\EuScript{P}}
\let \succ \relax
\DeclareMathOperator{\succ}{succ}
\DeclareMathOperator{\cl}{cl}

\newcommand{\minus}{\setminus}
\newcommand{\bP}{\mathbb{P}}
\newcommand{\bQ}{\mathbb{Q}}

\newcommand{\name}[1]{\dot{#1}}
\newcommand{\crank}[2]{|#1|_{#2}}
\newcommand{\aforc}{{\mathbb{B}\mathbb{M}}}
\newcommand{\clopen}[1]{[#1]}

\DeclareMathSymbol{\mlq}{\mathord}{operators}{'134}
\DeclareMathSymbol{\mrq}{\mathord}{operators}{'42}
\let\oldenquote\enquote
\renewcommand{\enquote}[1]{\relax\ifmmode\mlq#1\mrq\else\oldenquote{#1}\fi}

\newenvironment{enumerate-(1)}{\begin{enumerate}[label={\upshape (\arabic*)}, leftmargin=2pc]}{\end{enumerate}}

\newenvironment{itemizenew}{\begin{itemize}[leftmargin=2pc]}{\end{itemize}}

\renewcommand{\subset}{\subseteq}

\begin{document}
\title{Generalized Borel sets}
\date{\today}

\author[C.\ Agostini]{Claudio Agostini}
\address[Claudio Agostini]
{HUN-REN Alfréd Rényi Institute of Mathematics, Reáltanoda utca 13-15, H-1053, Budapest}
\email{agostini.claudio@renyi.hu}

\author[N.\ Chapman]{Nick Chapman}
\address[Nick Chapman]
{Institut f\"ur Diskrete Mathematik und Geometrie, Technische Universit\"at Wien, Wiedner Hauptstra{\ss}e 8-10/104, 1040 Vienna, Austria}
\email{nick.chapman@tuwien.ac.at}

\author[L.\ Motto Ros]{Luca Motto Ros}
\address[Luca Motto Ros]{Dipartimento di matematica \guillemotleft{Giuseppe Peano}\guillemotright, Universit\`a di Torino, Via Carlo Alberto 10, 10121 Torino, Italy}
\email{luca.mottoros@unito.it}

\author[B.\ Pitton]{Beatrice Pitton}
\address[Beatrice Pitton]{
Université de Lausanne, Quartier UNIL-Chamberonne, Bâtiment Anthropole,
1015 Lausanne, Switzerland and Dipartimento di matematica \guillemotleft{Giuseppe Peano}\guillemotright, Universit\`a di Torino, Via Carlo Alberto 10, 10121 Torino, Italy}
\email{beatrice.pitton@unil.ch}

\begin{abstract}
Generalizing classical descriptive set theory opens foundational questions about the Borel hierarchy. 
In this paper we systematically study those questions, working in the general framework of Polish-like spaces relative to an uncountable cardinal $\kappa$, possibly singular, satisfying $2^{<\kappa}=\kappa$. 
We provide fundamental properties of the $\kappa^+$-Borel hierarchy of any regular Hausdorff space of weight at most $\kappa$, and establish sufficient conditions for its non-collapse. 
We highlight a unique phenomenon that arises in the case of singular cardinals, namely, the existence of a second, distinct Borel hierarchy, the $\kappa$-Borel hierarchy: we prove that it is strictly finer than the $\kappa^+$-Borel hierarchy, and then characterize the precise relationship between the two. Finally, for regular cardinals, we resolve three questions about the behavior of the $\kappa^+$-Borel hierarchy on subspaces of the generalized Baire space $\pre{\kappa}{\kappa}$, constructing various models via forcing where several nontrivial constellations for the length of the $\kappa^+$-Borel hierarchy on the space are realized.
\end{abstract}

\subjclass[2020]{Primary: 03E15; Secondary: 03E35, 03E47, 54H05, 54E99}
\keywords{Generalized descriptive set theory, higher Baire spaces, generalized Borel sets and functions, hierarchies of definable sets}

\thanks{
This research was funded by the Austrian Science Fund (FWF) P35655 and P35588, by the Italian PRIN 2022 ``Models, sets and classifications'' (prot.\ 2022TECZJA), by the Gruppo Nazionale per le Strutture Algebriche, Geometriche e le loro Applicazioni (GNSAGA) of the Istituto Nazionale di Alta Matematica (INdAM) of Italy, and by the European Union  Horizon 2020 research and innovation programme, under the Marie Sklodowska-Curie action 101210902 (TopAspOfGDST). \includegraphics[height=1em]{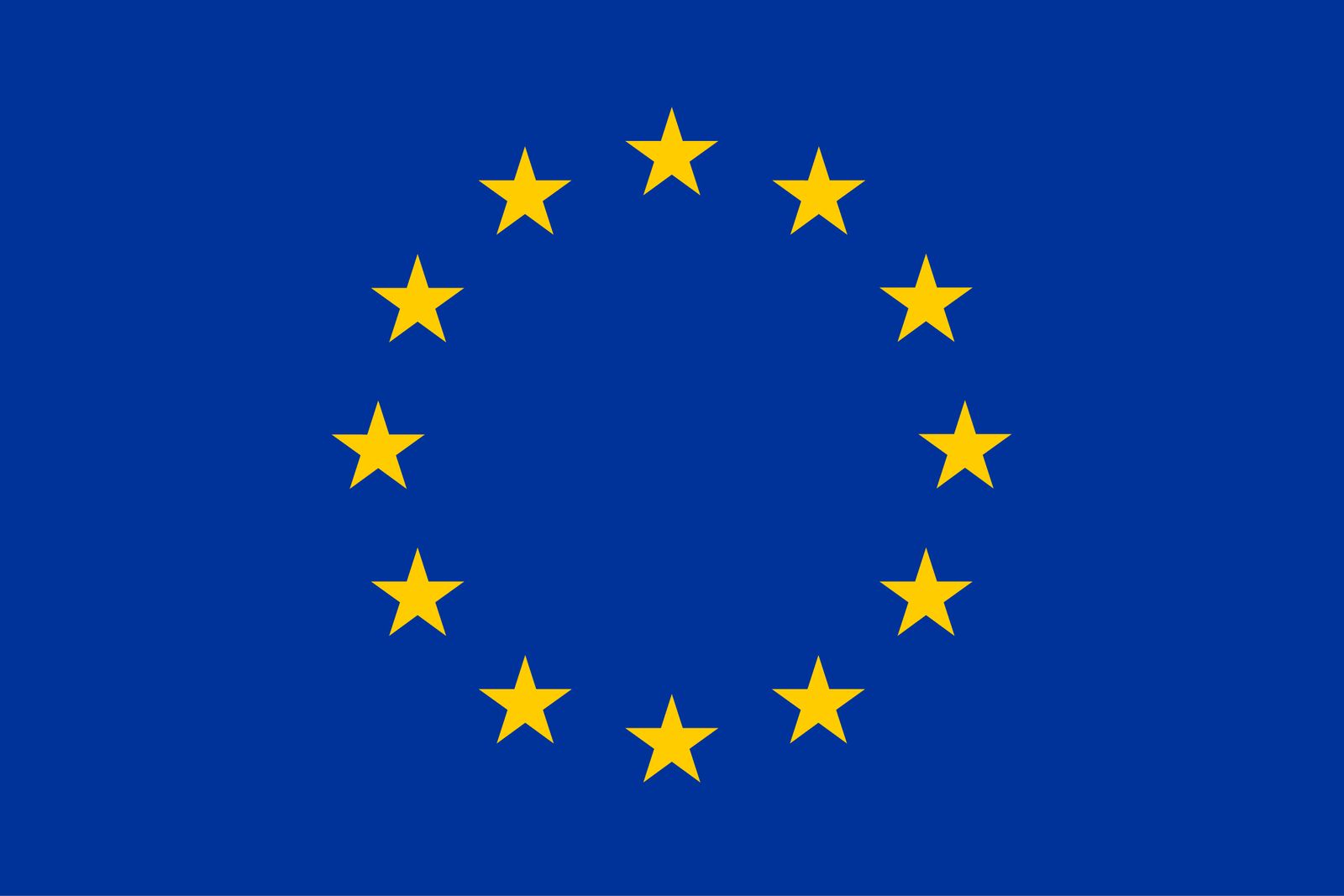} Views and opinions expressed are however those of the authors only and do not necessarily reflect those of the European Union or Horizon 2020 programme. Neither the European Union nor the granting authority can be held responsible for them. For open access purposes, the authors have applied a CC BY public copyright license to any author accepted manuscript version arising from this submission. Chapman would like to thank his advisor Martin Goldstern for continued support.
The authors would also warmly thank the anonymous referee for their careful reading of the paper and the useful suggestions.
}

\maketitle

\tableofcontents
\section{Introduction}

Descriptive set theory is a branch of set theory that investigates definable (e.g., Borel) subsets of the Cantor space \(\pre{\omega}{2}\), the Baire space \(\pre{\omega}{\omega}\), and other Polish spaces. In generalized descriptive set theory, one extends this framework by replacing $\omega$ with an uncountable cardinal $\kappa$ or its cofinality $\cof(\kappa)$. 
In this context, the correct generalizations of the Cantor and Baire space are the generalized Cantor space $\pre{\kappa}{2}$ and the generalized Baire space $\pre{\cof(\kappa)}{\kappa}$, both equipped with the bounded topology (see Section~\ref{sec:setup}).
Under the assumption \(2^{< \kappa}=\kappa\), these generalized spaces exhibit a key property analogous to the classical framework, that is, they both have weight $\kappa$, mirroring how $\pre{\omega}{\omega}$ and $\pre{\omega}{2}$ are second countable in the classical setting.
Also, in generalized descriptive set theory one is naturally led to the study of $\kappa^+$-Borel sets as analogues of classical Borel sets.

The generalized Cantor and Baire spaces have been extensively studied for $\kappa$ a (necessarily regular) cardinal satisfying \( \kappa^{< \kappa} = \kappa\), yielding interesting results on their $\kappa^+$-Borel subsets
(see e.g.\ \cite{MV,friedman_generalized_2014}, among many others). However, two significant aspects remain largely unexplored: first, the study of the $\kappa^+$-Borel hierarchy for \textit{subspaces} of the generalized Baire space, or for more general Polish-like spaces; and second, the entire framework of generalized descriptive set theory at singular cardinals, 
a new direction that has only recently begun to attract attention (see~\cite{AMR19} for some initial results, and~\cite{DMR} for an extensive treatment of the countable cofinality case). 
In particular, the study of hierarchies of definable sets for subspaces of $\pre{\cof(\kappa)}{\kappa}$ when $\kappa$ is a singular cardinal of uncountable cofinality constitutes a completely novel area that needed systematic study. 
Our work addresses all these issues.

The first part of this paper, i.e.\ Sections~\ref{sec:preliminaries}--\ref{sec_alt}, develops the foundational aspects of the theory of generalized Borel sets,
while the second part of the paper, corresponding to Sections~\ref{sec: psp section}--\ref{sec: alpha-forcing}, is concerned with somewhat unexpected behaviors of the \( \kappa^+ \)-Borel hierarchy on the arguably nicest kind of Polish-like spaces, namely, closed subspaces of \( \pre{\cof(\kappa)}{\kappa} \).
We pursue maximal generality in two directions. First, we work with cardinals $\kappa$ satisfying only $2^{< \kappa}=\kappa$, dropping the commonly assumed regularity condition, and thus including singular cardinals. Second, we extend beyond subspaces of $\pre{\cof(\kappa)}{\kappa}$ to the broader setting of regular Hausdorff spaces of weight at most \( \kappa \).
This aligns with recent advances in generalized descriptive set theory, where many classical results have been successfully extended to various classes of Polish-like spaces, rather than restricting the attention solely to $\pre{\cof(\kappa)}{\kappa}$ or $\pre{\kappa}{2}$ (see \cite{CoskSchlMR3453772, Gal19, PhDThesisAgostini, AgosMottoSchlichtRegular}).
Actually, our analysis indicates that even imposing such weaker Polish-like properties onto the space is not necessary for the study of its \( \kappa^+ \)-Borel hierarchy, and many results hold under remarkably mild conditions.
In particular, one can safely work with arbitrary subspaces of \( \pre{\k}{2}\) and \( \pre{\cof(\k)}{\k} \) and forget about all the restrictions stated in the various results. Furthermore, one can even work with arbitrary \( T_0 \) spaces of weight at most \( \k \), if only the global behavior of the hierarchy and its infinite levels are under question (see Proposition~\ref{prop:embed_T_0_spaces_into_Cantor}).

Let us now describe the content of the paper in more detail.

In Section~\ref{sec:preliminaries} we introduce the \( \g \)-Borel hierarchy (see Section~\ref{subsec:gamma-Borelhierarchy}), a natural stratification of the \( \g \)-Borel sets that is a foundational tool in descriptive set theory 
and plays a central role in both the classical case ($\gamma = \omega_1$) and its generalization to uncountable cardinals ($\gamma = \kappa^+$).
Despite its widespread use, a systematic presentation of the 
\(\g\)-Borel hierarchy in full generality has been lacking.
We address this gap by providing a formal treatment of all relevant definitions, and establishing the key preliminary results necessary for the development of the theory.
This includes natural features of the \( \g \)-Borel hierarchy, like the property of being increasing (Section~\ref{subsec:inreasing}), the length of the hierarchy on a given space (Section~\ref{subsec:order}) and on its subspaces (Section~\ref{subsec:decomposition}), or the existence of universal sets (Section~\ref{subsec:universalsets}).
In particular, we introduce the crucial concept of order \(\ord_\gamma(X)\) of the $\g$-Borel hierarchy on a space $X$ (Definition~\ref{def:order}), which measures its length.

In Section~\ref{sec_borel_hierarchy} (and, partially, in Section~\ref{sec_alt}), we provide a comprehensive description of the \(\k^+\)-Borel hierarchy on an arbitrary regular Hausdorff space \( X \) of weight at most $\kappa$,  and we establish its fundamental properties.  In particular: 
\begin{itemizenew}
\item 
We determine the optimal closure properties of the pointclasses \( \Siii{0}{\a}{\k^+}(X)\), \( \Piii{0}{\a}{\k^+}(X)\), and \( \Deee{0}{\a}{\k^+}(X) \) appearing in the \( \k^+ \)-Borel hierarchy (Proposition~\ref{k+_hierarchy_closure}, Proposition~\ref{prop:optimalityregular}, and Corollary~\ref{cor_hier_clos}).
\item 
We prove that the \( \k^+ \)-Borel hierarchy is proper, in the sense that new \( \k^+ \)-Borel subsets of \( X \) appear at each level of the hierarchy up to \( \ord_{\k^+}(X) \) (Corollary~\ref{cor:proper_for_k+} and Corollary~\ref{cor:proper_for_k+_final_level}).
\item 
We study the existence of universal and complete sets for the various levels of interest (Proposition~\ref{universal_sigma_pi}).
\end{itemizenew}
Interestingly enough, while for a regular \( \k \) all these results can be proved without auxiliary tools, when \( \k \) is singular we need to pass through the study of an alternative hierarchy, discussed below. This explains why all the results for the regular case already appear in Section~\ref{sec_borel_hierarchy}, while in a few cases their counterparts for the singular case have to be postponed to Section~\ref{sec_alt}.

We also provide a sufficient condition for the \(\k^+\)-Borel hierarchy to be non-collapsing, i.e.\ to ensure that \( \ord_{\kappa^+}(X) \) attains the maximal value \( \kappa^+ \):

\begin{theorem*}[Theorem~\ref{thm:non-collpase_general_space}]
Let $X$ be a regular Hausdorff topological space of weight at most $\kappa$.
  If there is a \( \kappa^+\)-Borel embedding of \( \pre{\kappa}{2} \) into \( X \), then
the $\kappa^+$-Borel hierarchy on $X$ does not collapse.
\end{theorem*}

When \(\k\) is regular, it is consistent that the converse of Theorem~\ref{thm:non-collpase_general_space} holds for all $\kappa^+$-Borel subspaces of $\pre{\kappa}{\kappa}$, and thus for all topological spaces of weight at most \( \kappa \)
that are $\kappa^+$-Borel isomorphic to a $\kappa^+$-Borel subset of $\pre{\kappa}{\kappa}$;\footnote{Such spaces have been referred to as standard Borel \( \kappa \)-spaces, see e.g.\ \cite[Section 3]{MR13}.} 
this is due to the fact that there are models of set theory in which the \( \kappa \)-Perfect Set Property holds for all $\kappa^+$-Borel subsets of $\pre{\kappa}{\kappa}$ (see~\cite{schlicht_perfect_2017}). This will be complemented in Section~\ref{sec: psp section}, where we show that, consistently, 
there can be even closed subsets of \( \pre{\kappa}{\kappa} \) which do not contain a \( \kappa^+ \)-Borel isomorphic copy of \( \pre{\kappa}{2} \), yet their \( \kappa^+ \)-Borel hierarchy does not collapse.

As anticipated, in Section~\ref{sec_alt} we also introduce and analyze an alternative hierarchy for \(\k^+\)-Borel sets that naturally arises when \(\k\) is a singular cardinal. Indeed, for such cardinals the notions of \(\k^+\)-algebra and \(\k\)-algebra coincide, allowing  the collection of \(\k^+\)-Borel subsets of a space to be equivalently described as the smallest \(\k\)-algebra containing all its open sets. This leads to a new hierarchy, called \(\k\)-Borel hierarchy, which is the focus of this section, along with its relationship to the \(\k^+\)-Borel hierarchy. 
One of our main results, Theorem~\ref{hierarchy_theorem}, precisely characterizes the connection between these two hierarchies for spaces whose open sets can be written as unions of \( \cof(\kappa) \)-many closed sets ---  
the latter is not a strong restriction, since such condition is naturally satisfied by most spaces studied in the literature on generalized descriptive set theory, including all subspaces of $\pre{\cof(\kappa)}{\kappa}$ and, more generally, all $(\omega,\cof(\kappa))$-Nagata--Smirnov spaces (see Fact~\ref{lem:NS_implies_increasing_hierarchy}).
It is remarkable that, even when the $\kappa^+$-Borel and $\k$-Borel sets coincide, the two families stratify in two different hierarchies, with the \(\k\)-Borel hierarchy being strictly finer than the \(\k^+\)-Borel hierarchy; thus, the \(\k\)-Borel hierarchy provides a more informative classification of $\k^+$-Borel sets when \( \k \) is singular.
Theorem~\ref{hierarchy_theorem} also allows us to extend various results from the $\k^+$-Borel hierarchy to the $\k$-Borel hierarchy. For example, it yields the bounds 
\[ 
\ord_{\kappa^+}(X) \leq \ord_{\kappa}(X) \leq 2 \cdot \ord_{\kappa^+}(X) 
\] 
relating the lengths of the two hierarchies (see Corollary~\ref{cor:ordersingular}), hence the $\k^+$-Borel hierarchy collapses if and only if the $\k$-Borel hierarchy does (Corollary~\ref{collapse_k}).
We also study additional fundamental properties of the $\k$-Borel hierarchy analogous to those listed above for the \( \k^+ \)-Borel hierarchy, including 
its optimal closure properties (Propositions~\ref{k_hierarchy_closure1} and~\ref{k_hierarchy_closure2}), its properness (Propositions~\ref{prop:proper_for_k} and~\ref{prop:proper_for_k_alpha=ord}), the conditions for its collapse (Proposition~\ref{prop:conditions_collapse_singular}), and the existence of universal and complete sets (Proposition~\ref{universal_sigma_pi_for_k-Borel}).
From these, we derive the counterparts for a singular cardinal \( \k \) of the missing results about the \( \k^+\)-Borel hierarchy obtained in Section~\ref{sec_borel_hierarchy} for the regular case (Corollary~\ref{cor_hier_clos}).

With Section~\ref{sec: psp section} and Section~\ref{sec: alpha-forcing}, we dive into the second part of the paper, where we return to the case of regular cardinals \( \kappa \), and we construct via forcing various models of set theory in which the $\kappa^+$-Borel hierarchy exhibits specific behaviors. 
Our aim is to address three central problems, each of which we resolve affirmatively.

The first problem is motivated by the fact that if \( X \subseteq \pre{\kappa}{\kappa} \) contains a \( \kappa \)-perfect subset, then \( \ord_{\kappa^+}(X) = \kappa^+ \); this is well-known in the classical setting \( \kappa = \omega \), and follows from our Theorem~\ref{thm:non-collpase_general_space} if \( \kappa \) is uncountable.

\begin{problem}  \label{q: no perfect but full order}
Can there be a $\kappa^+$-Borel (or even closed) set $X \subseteq \pre{\kappa}{\kappa}$ which contains no \( \kappa \)-perfect subset, yet $\ord_{\k^+}(X) = \kappa^+$?
\end{problem}

The reader might wonder which notion of \( \kappa \)-perfect set is involved in Problem~\ref{q: no perfect but full order}.
In fact, several reasonable notions of a \( \kappa \)-perfect set have been considered in generalized descriptive set theory, each with its own specific nuances, depending on the context. For the purposes of this paper, we adopt the following definitions. A set $X \subseteq \pre{\kappa}{\kappa}$ is \( \kappa \)-perfect if it is closed and homeomorphic to the generalized Cantor space $\pre{\kappa}{2}$, and it is \( \kappa \)-thin if it has no \( \kappa \)-perfect subset (Definition~\ref{def:perfect}).
This choice does not rule out the other possibilities that have been considered elsewhere:
in fact,
 Corollary~\ref{cor: thin borel set eq}
yields that for $\kappa^+$-Borel sets \( X \subseteq \pre{\k}{\k} \), our definition 
gives rise to the same notion of \( \kappa \)-Perfect Set Property given by 
most of the other definitions of \( \kappa \)-perfect set proposed in literature: indeed, $X$ has a \( \kappa \)-perfect subset if and only if there is $\kappa^+$-Borel injection from $\pre{\kappa}{2}$ into $X$.
     In particular, this aligns our formulation of Problem~\ref{q: no perfect but full order} with Theorem~\ref{thm:non-collpase_general_space}.

Problem~\ref{q: large and ord 2} is instead motivated by the easy observation that if \( |X| \leq \kappa \), then \( \ord_{\kappa^+}(X) \leq 2 \) because every subset of \( X \) can be written as a union of length at most \( \kappa \) of singletons.

\begin{problem} \label{q: large and ord 2}
Can there be a $\kappa^+$-Borel (or even closed) set $X \subseteq \pre{\kappa}{\kappa}$ such that $|X| > \kappa$, yet $\ord_{\kappa^+}(X) = 2$?
\end{problem}

In the classical setting \( \kappa = \omega \), the Perfect Set Property makes questions about the order of the Borel hierarchy on Borel sets \( X \) moot: either \( X \) is countable, and hence $\ord_{\omega_1}(X) \leq 2$, or it contains a copy of the whole Cantor space \( \pre{\omega}{2} \), and thus $\ord_{\omega_1}(X) = \omega_1$. In the uncountable setting, however, the fact that the \( \kappa \)-Perfect Set Property may consistently fail for \( \kappa^+ \)-Borel (and even closed) sets makes the next problem worth investigating. 

\begin{problem}\label{q: order in between}
Can there be a $\kappa^+$-Borel (or even closed) set $X \subseteq \pre{\kappa}{\kappa}$ with $2 < \ord_{\kappa^+}(X) < \kappa^+$?
\end{problem}

In Section~\ref{sec: psp section}, we investigate the interplay between the \( \kappa \)-Perfect Set Property and the order of the $\kappa^+$-Borel hierarchy on definable subspaces of the generalized Baire space.
We first deal with matters of definability, and lay the groundwork for connecting the \( \kappa \)-Perfect Set Property to the interpretation of a definable set in generic extensions. This is done via Theorem \ref{th: imperfect equiv}, due to L\"ucke \cite{luckeSigma11Definability2012}. We observe that we may pass to generic extensions by ${<}\kappa$-closed forcings without adding new elements to \( \kappa \)-thin $\kappa^+$-Borel sets; in other words, every code for a \( \kappa \)-thin $\kappa^+$-Borel set yields the same subset of $\pre{\kappa}{\kappa}$ when interpreted in a ${<}\kappa$-closed forcing extension. 
We also state Corollary~\ref{cor:closed_uncountable_order_2}, anticipating a positive answer to Problem~\ref{q: large and ord 2}.
We then turn our attention to combining the work of Hamkins, Kunen and Miller to prove Lemma~\ref{lem: small forcing preserves order}, in analogy with what has been shown about the Cohen real model in \cite[Theorem~14.3]{miller_descriptive_1995}. Together with Theorem~\ref{th: hamkins}, this shows (Corollary~\ref{cor: closed set with full order}) that the ground model's generalized Baire space \( \pre{\k}{\k} \cap V \) is a \( \kappa \)-thin closed set with order $\kappa^+$ in every forcing extension \( V[G] \) by a small forcing, giving a positive answer to Problem~\ref{q: no perfect but full order} and showing that the sufficient conditions provided in Theorem~\ref{thm:non-collpase_general_space} are not necessary in general, as already anticipated.

Constructing a space $X \subseteq \pre{\kappa}{\kappa}$ with $2 < \ord_{\kappa^+}(X) < \kappa^+$ to solve Problem~\ref{q: order in between} is significantly more difficult. In Section~\ref{sec: alpha-forcing}, we partially generalize work of Miller \cite{miller_length_1979} surrounding the modification of the order of the Borel hierarchy on a given space $X$. The primary tool here is Miller's \textit{$\alpha$-forcing}, the use of which allows for fine control over specific levels of the hierarchy. Following the proof of \cite[Theorem~34]{miller_length_1979}, we use of an iteration of $\alpha$-forcing to obtain:

\begin{theorem*}[Theorem~\ref{th: set order to n}]  Let $X \subseteq \pre{\kappa}{\kappa}$ be such that $|X| > \kappa$, and let $1 < n < \omega$. Then there exists a ${<}\kappa$-closed, $\kappa^+$-c.c.\ forcing extension $V[G]$ of \( V \) such that $V[G] \models \ord_{\kappa^+}(X) = n$. 
\end{theorem*}

Finally, combining this theorem with the results in Section \ref{sec: psp section}, we derive the following:

\begin{theorem*}[Corollary~\ref{cor:final_7}]
Let $X \subseteq \pre{\kappa}{\kappa}$ be a \( \kappa \)-thin $\kappa^+$-Borel set with $|X| > \kappa$, and let $1 < n < \omega$. Then there is a ${<}\kappa$-closed, $\kappa^+$-c.c.\ forcing extension $V[G]$ of \( V \) such that $V[G] \models {X^{V[G]} = X} \wedge {\ord_{\kappa^+}(X) = n}$.
\end{theorem*}

In the above theorem, \enquote{\( X^{V} = X^{V[G]} \)} means that given a \( \kappa^+ \)-Borel code for \( X \) (in \( V \)), its re-interpretation in \( V[G] \) gives rise to exactly the same set \( X \). In particular, this means that if \( X \) was a closed subset of \( \pre{\kappa}{\kappa}\) in \( V \), then it stays closed also when moving to the generic extension \( V[G] \), and similarly for the other classes in the \( \kappa^+ \)-Borel hierarchy of \( \pre{\kappa}{\kappa}\). 
Thus we obtain that, consistently, there are closed (hence nice Polish-like) subspaces of \( \pre{\kappa}{\kappa}\) whose \( \kappa^+ \)-Borel hierarchy has length strictly between \( 2 \) and \( \kappa^+ \), in stark contrast with the situation in classical descriptive set theory.

The last two results can be extended to infinite ordinals as well, but the proof requires a more complicated treatment and will appear in a separate paper by the second author.
Other natural continuations of the present work include the study of hierarchies refining the Borel one, such as the difference hierarchy and the Wadge hierarchy, and the search for analogues of the results from Section~\ref{sec: psp section} and~\ref{sec: alpha-forcing} for singular cardinals of uncountable cofinality. The former line of research has already been pursued by the third and fourth authors (and their collaborators) and will be published in a sequel to this paper, while the latter is left open for future research.

\section{Setup and preliminaries} \label{sec:setup}

A basic understanding of set theory and general topology is assumed, and standard notation and terminology from these fields is employed. For any undefined notion, readers are directed to \cite{jech_set_2003,ER}. We also refer to \cite{AMR19} for the generalized descriptive set theory notation adopted in this paper.

Throughout the paper, we work in \( \ZFC \) and we assume that 
\begin{center}
\emph{\(\k\) is an uncountable cardinal satisfying the condition \(2^{<\k}=\k\).}
\end{center}
We let \( \Ord \) be the class of all ordinals, and \( \Card \) be the subclass of cardinals.
Given any set \( X \), we denote by \( \id_X \) the identity function on \( X \). 
A function \( f \colon X \to Y \) between the topological spaces \( X \) and \( Y \) is an embedding if it is a homeomorphism onto its image \( f[X] \), where the latter is equipped with the subspace topology inherited from \( Y \).
The product of any two functions \( f \colon X \to Y \) and \( g \colon V \to W \) is the map \( f \times g \colon X \times V \to Y \times W \) sending \((x,v) \in X \times V \) to \( (f(x),g(v)) \). If all of \( X \), \( Y \), \( V \), and \( W \) are topological spaces, then \( f \times g \) is continuous if and only if so are \( f \) and \( g \).

The \markdef{weight} of a topological space \( X \) is the smallest cardinality of a basis for its topology, and is denoted by \( \weight(X) \), or \( \weight(X,\tau) \) when we want to specify the topology \( \tau \) on \( X \). Given \( \nu \in \Card \), we say that the space \( X \)
is \markdef{$\nu$-additive} if the intersection of any collection of fewer than \(\nu\)-many open subsets of \( X \) is still an open set. In particular, every topological space is \( \omega \)-additive, and it is easy to show that if a regular space \( X \) is \( \nu \)-additive for some \( \nu > \omega \), then \( X \) is also zero-dimensional, that is, it admits a basis consisting of clopen sets.
It is easy to check that if a Hausdorff space \( X \) is \( \nu \)-additive for \( \nu > \weight( X ) \), then \( X \) is discrete; if \( X \) is not discrete, instead, the set of infinite cardinals \( \nu \) such that \( X \) is \( \nu \)-additive has a maximum, which we call the \markdef{additivity of \( X \)}. 
 
As mentioned in the introduction, generalized descriptive set theory is primarily concerned with the study of the generalized Cantor and Baire spaces, which are constructed as follows.
Given a set $A$ and \(\gamma \in \Ord\),  the set $\pre{\gamma}{A}=\{ x \mid x \colon \gamma \to A \}$ is the Cartesian product of \(\gamma\)-many copies of \(A\), that is, \(\pre{\gamma}{A}=\prod_{\alpha<\gamma}A_i\) where \(A_\alpha=A\) for every \(\alpha<\gamma\). We also set \(\pre{<\gamma}{A}= \bigcup_{\b<\gamma}\pre{\b}{A}\). 
Let \( \mu \in \Card \).
Any set of the form $\pre{\mu}{A}$ is naturally equipped with the \markdef{bounded topology} $\tau_b$, i.e.\ the  smallest topology generated by the family $\{\clopen{s} \mid s \in \pre{<\mu}{A}\}$, where $\clopen{s} = \{ x \in \pre{\mu}{A} \mid s \subseteq x \}$. It is straightforward to see that each $\clopen{s}$ is clopen, making the topology \(\tau_b\) zero-dimensional. If \(\mu\) is infinite and \(A\) has at least two elements, then the space \((\pre{\mu}{A}, \tau_b)\) is regular Hausdorff and has weight \(|A|^{<\mu}\) and additivity \(\cof(\mu)\). 
The \markdef{generalized Baire space} and the \markdef{generalized Cantor space} are, respectively, \((\pre{\cof(\k)}{\k}, \tau_b)\) and  \((\pre{\k}{2}, \tau_b)\). The reference to \( \tau_b \) will be dropped in most cases.

Recent advances in generalized descriptive set theory have pushed its scope beyond the two spaces mentioned above, and several classes of Polish-like spaces that can be considered as higher analogues of classical Polish spaces have been isolated --- we refer the reader to \cite{PhDThesisAgostini,AgosMottoSchlichtRegular} for a thorough discussion on this. 
What matters for our purposes is that all such spaces are regular Hausdorff and have weight at most $\kappa$: this justifies our restriction to spaces with these features from Section~\ref{sec:preliminaries} onward. 
When \( \kappa \) is singular, there is another condition that is shared by most Polish-like spaces considered so far and that is relevant to us, namely, the property of being a \( (\omega,\cof(\kappa)) \)-Nagata--Smirnov space.\footnote{Let $X$ be a topological space.
A (regular Hausdorff) topological space is called a \( (\omega,\cof(\k)) \)-Nagata--Smirnov if it has a basis for the topology which is the union of $\cof(\kappa)$-many locally finite families, where a family $\F$ of subsets of $X$ is said to be locally finite if every point $x\in X$ has a neighborhood $U$ that intersects finitely many elements of $\F$.}
The class of $(\omega,\cof(\kappa))$-Nagata--Smirnov spaces is very wide and includes most spaces considered in previous literature in generalized descriptive set theory, like all subspaces of \(\ \pre{\k}{2}\) and \( \pre{\cof(\kappa)}{\kappa} \), all $\GG$-metrizable spaces for a totally ordered group $\GG$ of degree $\cof(\k)$, and all spaces of weight at most $\cof(\kappa)$. Thus, assuming that the space is \( (\omega,\cof(\kappa)) \)-Nagata--Smirnov is not particularly restrictive either.
In any case, the only thing that we will use about such spaces is that, as shown in~\cite[Lemma 2.2.48]{PhDThesisAgostini}: 

\begin{fact}\label{lem:NS_implies_increasing_hierarchy}
Let \(X\) be a regular Hausdorff $(\omega,\cof(\kappa))$-Nagata--Smirnov space. Then every open subset of \( X \) can be written as a union of \( \cof(\kappa) \)-many closed sets.
\end{fact}

The latter property is the one that will be explicitly stated in our results, when needed.
In contrast, it is interesting to notice that any further ``Polish-like requirement'' that one can add to the spaces under consideration, like being (completely) $\cof(\kappa)$-metrizable or being spherically complete, does not have a direct impact on the theory of generalized Borel sets that we are going to develop.

A \markdef{pointclass} \(\boldsymbol\Gamma\) is a class-function assigning to every nonempty topological space $X$ a nonempty family \(\boldsymbol\Gamma(X) \subseteq \powset(X)\). 
The dual of \( \boldsymbol{\Gamma}\) is the pointclass \( \check{\boldsymbol\Gamma} \) defined by \(\check{\boldsymbol\Gamma}(X)=\left\{ X \setminus A \mid A \in \boldsymbol\Gamma(X)\right\} \), while the ambiguous 
pointclass \( \boldsymbol{\Delta_{\Gamma}} \) associated to \(\boldsymbol\Gamma\) is obtained by setting \(\boldsymbol{\Delta_{\Gamma}}(X)= \boldsymbol\Gamma(X) \cap \check{\boldsymbol\Gamma}(X)\).
A pointclass \(\boldsymbol{\Gamma}\) is said to be \markdef{boldface} if it is closed under continuous preimages,
that is: \(f^{-1}(B)\in \boldsymbol{\Gamma}(X)\) whenever \(B \in \boldsymbol{\Gamma}(Y)\) and \(f \colon X \to Y \) is continuous. Obviously, if \( \boldsymbol{\Gamma} \) is boldface, then so are \( \check{\boldsymbol{\Gamma}}\) and \( \boldsymbol{\Delta_{\Gamma}} \).
 A pointclass \( \boldsymbol{\Gamma} \) is \markdef{hereditary} if \( \boldsymbol{\Gamma}(Y) = \{ A \cap Y \mid A \in \boldsymbol{\Gamma}(X) \} \) for every \( Y \subseteq X \). Notice that if \( \boldsymbol{\Gamma}\) is a boldface pointclass, then it is hereditary if and only if for every \( Y \subseteq X \) and \( B \in \boldsymbol{\Gamma}(Y) \) there is \( A \in \boldsymbol{\Gamma}(X) \) such that \( B = A \cap Y \). Finally, 
a boldface pointclass \(\boldsymbol{\Gamma}\) is called \markdef{nonselfdual on $X$} if \(\boldsymbol{\Gamma}(X)\neq \check{\boldsymbol{\Gamma}}(X)\), and it is called \markdef{selfdual on $X$} otherwise. 
 A function \( f \colon X \to Y \) between two topological spaces is \markdef{\( \boldsymbol{\Gamma}\)-measurable} if \( f^{-1}(U) \in \boldsymbol{\Gamma}(X) \) for every open \( U \subseteq Y \).

Let $X,Y$ be topological spaces, and let \(\boldsymbol{\Gamma}\) be a boldface pointclass.
 A set \(A \subseteq X\) is \markdef{\( \kappa \)-hard for \(\boldsymbol{\Gamma}\)} if for all \( B \in \boldsymbol{\Gamma}(\pre{\k}{2})\) there is a continuous function \( f \colon \pre{\kappa}{2} \to X \) such that \(B = f^{-1} (A) \); if moreover \( A \in \boldsymbol{\Gamma}(X)\), then we say that \( A \) is \markdef{\( \kappa \)-complete for \(\boldsymbol{\Gamma}\)}. 
A set \(\U \subseteq Y \times X \) is \markdef{$Y$-universal for \(\boldsymbol{\Gamma}(X)\)} if \( \U \in \boldsymbol{\Gamma}(Y \times X) \) and \(\ \boldsymbol{\Gamma}(X)=\left\{ \U_y \mid y \in Y\right \} \), where \(\U_y= \left\{ x \in X \mid (y,x) \in \U \right\} \) is the vertical section of \(\U\) at \( y \). It is clear that if \(\U\) is $Y$-universal for \(\boldsymbol{\Gamma}(X)\), then its complement \(\U^c = (Y \times X) \setminus \U\) is $Y$-universal for \(\check{\boldsymbol{\Gamma}}(X)\).
Note that in order to show that  \(\U \in \boldsymbol{\Gamma}(Y \times X) \) is Y-universal for \(\boldsymbol{\Gamma}(X)\), it is enough to check that \(\U \in \boldsymbol{\Gamma}(Y \times X) \) and that for every \(A \in \boldsymbol{\Gamma}(X) \) there is \( y \in Y \) such that \(A=\U_y\): indeed, the remaining condition \enquote{\(\U_y \in \boldsymbol{\Gamma}(X)\) for all \(y \in Y\)} already follows from the fact that each \(\U_y\) is the preimage of \(\U\) via the continuous map \(x \mapsto (y,x)\).

We report three technical lemmas concerning the existence of universal and complete sets. They are basically folklore, but since we could not trace them back in the literature in the form we need, we give full proofs for the reader's convenience.
 
\begin{lemma}\label{lemma_universal}
Let $\boldsymbol\Gamma$ be a hereditary
boldface pointclass, and let \(X \), \( Y \) and \( Z\) be topological spaces such that there is a topological embedding
 $f \colon Z \to Y$. If there is a $Z$-universal set \( \U' \) for $\boldsymbol\Gamma(X)$, then there is a \(Y\)-universal set \( \U \) for \(\boldsymbol\Gamma(X)\). 
 \end{lemma}

\begin{proof}
The map \( g = f \times \id_X\) is an embedding of \( Z \times X \) into \( Y \times X \), therefore \( g[\U'] \in \boldsymbol{\Gamma}(f[Z] \times X )\) because \( \boldsymbol{\Gamma} \) is boldface. Since it is also hereditary, there is \( \U \in \boldsymbol{\Gamma}(Y \times X) \) such that \( \U \cap (f[Z] \times X) = g[\U'] \). Such a \( \U \) is as desired.
   \end{proof}

\begin{lemma}\label{univ_compl} 
Let $\boldsymbol\Gamma$ be a hereditary boldface pointclass, and let \(X , Y \) be topological spaces. Suppose that there is an embedding $f \colon \pre{\k}{2} \to X$. If \( \U \subseteq Y \times X \) is \( Y \)-universal for \( \boldsymbol{\Gamma}(X) \), then \( \U \) is \(\k\)-complete for \( \boldsymbol{\Gamma} \). If in addition \( Y \times X \) is homeomorphic to \( X \), then there is a subset of \( X \) which is \(\k\)-complete for \( \boldsymbol{\Gamma} \). 
\end{lemma}

\begin{proof}
Fix any \( A \in \boldsymbol{\Gamma}(\pre{\kappa}{2}) \). Then \( f[A] \in \boldsymbol{\Gamma}(f[\pre{\kappa}{2}]) \), hence there is \( B \in \boldsymbol{\Gamma}(X)\) such that \( B \cap f[\pre{\kappa}{2}] = f[A] \). Let \( \bar y \in Y \) be such that \( B = \U_{\bar y} \): then the continuous map \( g \colon \pre{\kappa}{2} \to Y \times X\) defined by \( g(x) = (\bar y,f(x))\) is  such that \( g^{-1}(\U) = A \). This shows that \( \U \) is \(\k\)-complete for \( \boldsymbol{\Gamma}\). If moreover, \( h \colon Y \times X \to X \) is a homeomorphism, then \( h \circ g \colon \pre{\kappa}{2} \to X \) is continuous and such that \( (h \circ g)^{-1}(h(\U)) = A \). It easily follows that \( h(\U) \subseteq X \) is \(\k\)-complete for \( \boldsymbol{\Gamma} \) as well.
\end{proof}

\begin{lemma}\label{universal_selfdual}
Let \(X\) be a topological space, and let \(\boldsymbol{\Gamma} \) be a boldface pointclass that is selfdual on \( X \).  Then there is no \(X\)-universal set for \(\boldsymbol{\Gamma}(X) \). 	
\end{lemma}

\begin{proof}
Towards a contradiction, suppose that there exists an X-universal set \(\U\) for \(\boldsymbol{\Gamma}(X) \). Let \(f \colon X \rightarrow X \times X \) be such that \(f(x)=(x,x)\), and let \(D=\left\lbrace x \in X \mid f(x)\notin \U \right\rbrace \). Note that \(D= X \setminus f^{-1}(\U)\).  Since \(\boldsymbol{\Gamma} \) is boldface and \(\boldsymbol{\Gamma}(X)= \check{\boldsymbol{\Gamma}}(X) \), then \(D \in \boldsymbol{\Gamma}(X)\). Hence there is \(y_0 \in X\) such that \(\U_{y_0}=D\)  by universality of \(\U\). But then 
\[ 
(y_0,y_0) \in \U \IFF y_0 \in \U_{y_0} \IFF y_0 \in D \IFF f(y_0)\notin \U \IFF (y_0,y_0) \notin \U, \] 
a contradiction.
\end{proof}

For the results of Sections~\ref{sec: psp section} and \ref{sec: alpha-forcing}, we assume the reader to be familiar and comfortable with the theory of iterated forcing.
Our notation for forcing is standard. The statement $q \leq p$ means \enquote{$q$ is stronger than $p$}; in addition, we strive to follow Goldstern's alphabet convention, i.e., stronger conditions should come later in the alphabet. If $\varphi$ is a statement in the forcing language, we say that a condition $p$ decides $\varphi$ if either $p \Vdash \varphi$ or $p \Vdash \neg \varphi$ holds. Likewise, if $\name{\tau}$ is name for a ground model object, we say that $p$ decides $\name{\tau}$ if there exists an $x \in V$ with $p \Vdash \name{\tau} = \check{x}$, where $\check{x}$ is the standard (check) name for a ground model object. We write $\bP \Vdash \varphi$ if the statement $\varphi$ is forced by every condition in $\bP$. As usual, $\parallel$ and $\perp$ denote the compatibility and incompatibility relation, respectively.

\begin{definition} \label{def:forcinproperties}
We say a forcing notion $\bP$ is
\begin{itemizenew}
\item 
\markdef{${<}\kappa$-closed} 
if for every decreasing sequence $\seq{p_i}{i < \delta}$ with $\delta < \kappa$ we can find  $q \in \bP$ such that $q \leq p_i$ for all $i < \delta$; 
\item 
\markdef{$\kappa^+$-c.c.} 
if every antichain in $\bP$ has size at most $\kappa$;
\item 
\markdef{$\kappa$-linked} 
if there are $\seq{\bP_i}{i<\kappa}$ such that $\bP = \bigcup_{i < \kappa} \bP_i$ and for each $i < \kappa$, any $p,q \in \bP_i$ are compatible;
\item 
\markdef{well-met} 
if for any two compatible conditions $p, q \in \bP$ there exists a greatest lower bound $r = p \wedge q$, that is, $r \leq p,q$ and for every $r'$ with $r' \leq p,q$ we have $r' \leq r$.
\end{itemizenew}
\end{definition}

\begin{fact}[Folklore] \label{fact: iteration cc}
Let $\seq{\bP_\gamma, \name{\bQ}_\gamma}{\gamma < \gamma^*}$ be a ${<}\kappa$-supported iteration such that for every $\gamma < \gamma^*$, $\name{\bQ}_\gamma$ is forced to be ${<}\kappa$-closed, well-met, and $\kappa$-linked. Then $\bP_{\gamma^*}$ satisfies the $\kappa^+$-c.c.\ and is ${<}\kappa$-closed.
\end{fact}

\begin{proof}
It is well-known that such iterations are ${<}\kappa$-closed. For the $\kappa^+$-c.c., see for example \cite[Lemma~V.5.14]{kunen2013}; the proof is formulated for the case $\kappa = \omega_1$, but transfers seamlessly to all uncountable $\kappa = \kappa^{<\kappa}$.
\end{proof}

\section{\texorpdfstring{\( \gamma \)-Borel sets, and their hierarchy}{gamma-Borel sets, and their hierarchy}} \label{sec:preliminaries}

Although we will mostly be concerned with \( \gamma \)-Borel sets for \( \gamma = \kappa^+\) (Sections~\ref{sec_borel_hierarchy}, \ref{sec: psp section}, and~\ref{sec: alpha-forcing}) or, when \( \kappa \) is singular, for \( \gamma = \kappa \) (Section~\ref{sec_alt}), in this section we develop the theory in full generality for an arbitrary ordinal \( \g > \o \).

\subsection{\texorpdfstring{The \( \g \)-Borel hierarchy}{The gamma-Borel hierarchy}} \label{subsec:gamma-Borelhierarchy}
Let $\g \in \Ord$. A \markdef{\( \g \)-algebra} on a set $X$ is a family of subsets of $X$ which is closed under the operations of complementation and well-ordered unions of length less than \( \g\).
When \(X\) is a topological space, the \( \g \)-algebra generated by the topology of \( X \), denoted by \( \Bor{\gamma}(X) \), is the smallest \( \gamma \)-algebra on \( X \) containing all its open sets; its element are called \markdef{\( \g \)-Borel} sets. 
Equivalently, \( \Bor{\gamma}(X) \) is the smallest collection of subsets of \( X \) containing all open and closed sets and closed under intersections and unions of length less than \( \gamma \).
When needed, we might add a reference to the topology \( \tau \) of \( X \) in the notation, and write e.g.\ \( \Bor{\gamma}(X,\tau) \). Letting vary \( X \) over all nonempty topological spaces we get the pointclass \( \Bor{\gamma}\), which is easily seen to be boldface and hereditary.

A function \( f \colon X \to Y \) between two topological spaces \( X \) and \( Y \) is \markdef{\( \g \)-Borel} if it is \( \Bor{\gamma} \)-measurable; this is the same as requiring that \( f^{-1}(A) \in \Bor{\gamma}(X) \) for every \( A \in \Bor{\gamma}(Y)\). 
A \markdef{\(\g \)-Borel isomorphism} between $X$  and $Y$ is a bijection $f \colon X \to Y$ such that both $f$ and $f^{-1}$ are \(\g \)-Borel; $X$ and $Y$ are then \markdef{\(\g\)-Borel isomorphic} if there is a \(\g\)-Borel isomorphism between them. A \markdef{\(\g \)-Borel embedding} $f \colon X \to Y $ is an injective function which is a \(\g \)-Borel isomorphism as a function from $X$ to $f[X]$.

\begin{lemma}\label{graf_closed}
Let \(f \colon X \to Y \) be a function between two topological spaces \( X \) and \( Y \), with \( Y \) Hausdorff. If $f$ is \(\g\)-Borel, then its graph 
\(
\mathrm{Gr}(f)=\{ (x,y) \in X \times Y \mid f(x)=y\}
\)
is a \(\g\)-Borel subset of \( X \times Y \). If $f$ is continuous, then \( \mathrm{Gr}(f)\) is closed.
\end{lemma}

\begin{proof}
The diagonal $\Delta=\{(y,y) \in Y \times Y \}$ of \( Y \) is closed in $Y\times Y$ because \( Y \) is Hausdorff. 
The function $f \times \id_Y \colon X\times Y\to Y\times Y$ is $\g$-Borel (respectively, continuous) if and only if $f$ is $\g$-Borel (respectively, continuous). Since \( \mathrm{Gr}(f)=(f \times \id_Y)^{-1}(\Delta) \), the result follows. 
\end{proof}

Given a set \( X \), a family $\A \subseteq \powset(X)$, and \(\mu \in \Card\), we let
\begin{align*}
( \A )_{\sigma_{<\mu}} &= \left\{ \bigcup\nolimits_{\alpha < \beta} A_\alpha \mid \beta < \mu, A_\alpha \in \A \right\},
\quad \text{and} \\
( \A )_{\delta_{<\mu}} &= \left\{ \bigcap\nolimits_{\alpha < \beta} A_\alpha \mid \beta < \mu, A_\alpha \in \A \right\}. 
\end{align*}
To simplify the notation, we also set \( ( \A )_{\sigma_\mu}  = (\A)_{\sigma_{<\mu^+}}  \) and	\( ( \A )_{\delta_\mu}  = (\A)_{\delta_{<\mu^+}}  \).

In the classical case \(\k=\o\), the collection of Borel subsets of a topological space \( X \) is stratified in a hierarchy formed by the classes \(\boldsymbol{\Sigma}^0_\alpha(X)\), \(\boldsymbol{\Pi}^0_\alpha(X)\), and \( \boldsymbol{\Delta}^0_\alpha(X) \), where \( \alpha \) ranges over nonzero ordinals. 
Following~\cite[Section 2.4]{AMR19},
a similar construction can be carried out for the collection of $\gamma$-Borel sets, for any \(\gamma>\o\): we call it the \markdef{$\g$-Borel hierarchy}.    
\begin{defin}\label{def_hierarchy_k+}
For every topological space \( X \), the following classes are defined by recursion on the ordinal \(\a\geq 1\):
\begin{align*}
\Siii{0}{1}{\g}(X) &= \left\{  U \subseteq X \mid U \text{ is open} \right\}  &&&
\Piii{0}{1}{\g}(X)&= \left\{  C \subseteq X \mid C \text{ is closed} \right\}  \\
\Siii{0}{\a}{\g}(X)  &=  \left( \bigcup\nolimits_{1 \leq \beta < \alpha}  \Piii{0}{\b}{\g}(X) \right)_{\sigma_{<\gamma}} &&&
\Piii{0}{\a}{\g}(X)  &= \left\{ X \setminus A  \mid A \in  \Siii{0}{\a}{\g}(X) \right\}.
\end{align*}
We also set \(  \Deee{0}{\a}{\g}(X) = \Siii{0}{\a}{\g}(X)   \cap \Piii{0}{\a}{\g}(X)   \).
\end{defin}
Notice that
\( \Piii{0}{\a}{\g}(X)  = \left( \bigcup_{1 \leq \beta < \alpha}  \Siii{0}{\b}{\g}(X)  \right)_{\d_{<\gamma}}\).
When it is important to specify the topology \( \tau \) of \(X\), we might add a reference to it in the notation and write e.g.\ \( \Siii{0}{\alpha}{\g}(X,\tau) \).

Letting \( X \) vary over all nonempty topological spaces, we get the boldface pointclass \( \Siii{0}{\a}{\g} \), its dual \( \Piii{0}{\a}{\g} \), and the associated ambiguous pointclass \( \Deee{0}{\a}{\g} \). The pointclasses \( \Siii{0}{\a}{\g} \) and \( \Piii{0}{\a}{\g} \) are also hereditary, while in general \( \Deee{0}{\a}{\g} \) fails to have such property.

Let \( \g^* \) be the smallest regular cardinal such that \( \g^* \geq \g \), that is, \( \g^* = \g \) if \( \g \) is already a regular cardinal, and \( \g^* = |\g|^+ \) otherwise.
Then it is easy to check that
\begin{equation} \label{upperboundonlength1}
\Bor{\g}(X) = \bigcup_{1 \leq \alpha < \g^*} \Siii{0}{\a}{\g}(X)   = \bigcup_{1 \leq \alpha < \g^*} \Piii{0}{\a}{\g}(X)   = \bigcup_{1 \leq \alpha < \g^*} \Deee{0}{\a}{\g}(X).  
\end{equation}

\begin{remark}\label{rmk:refining}
Let $\tau \subseteq \tau'$ be two topologies on a set $X$,
and let $\a,\g, \g' \in \Ord$ be such that $\a \geq 1$ and \(\g \leq \g'\).
Then $\Siii{0}{\a}{\g}(X,\tau)\subseteq \Siii{0}{\a}{\g'}(X,\tau')$, and hence also $\Piii{0}{\a}{\g}(X,\tau)\subseteq \Piii{0}{\a}{\g'}(X,\tau')$, $\Deee{0}{\a}{\g}(X,\tau)\subseteq \Deee{0}{\a}{\g'}(X,\tau')$, and $\Bor{\g}(X,\tau)\subseteq\Bor{\g'}(X,\tau')$.

Moreover, when \( \gamma \) is not a cardinal, then \( \Bor{\g}(X,\tau)  = \Bor{|\g|^+}(X,\tau) \) and the \( \gamma \)-Borel hierarchy coincides, level by level, with the \( |\g|^+\)-Borel hierarchy. 
\end{remark}

Thus it would not be restrictive to assume that \( \gamma \) is always cardinal, and the only relevant distinction is whether such a cardinal is regular or singular. Moreover, we get the following easy fact.

\begin{remark} \label{rmk:easyclosure}
Let \( X \) be any topological space, and \( \alpha \geq 1\).
\begin{itemizenew}
 \item
If \( \gamma \) is a regular cardinal, then \( \Siii{0}{\a}{\g}(X)\) is closed under unions of length less than \( \gamma \).
\item 
If \( \gamma \) is a singular cardinal, then \( \Siii{0}{\a}{\g}(X)\) is closed under unions of length less than \( \cof(\gamma) \).
\item 
If \( \gamma \) is not a cardinal, then \( \Siii{0}{\a}{\g}(X)\) is closed under unions of length less than \( |\g|^+ \), and in particular under \( |\g| \)-sized unions.
\end{itemizenew}
Dually, we obtain closure properties of \( \Piii{0}{\a}{\g}(X)\) under intersections.
\end{remark}

\begin{lemma}\label{lem:change_topology_hierarchy}
Suppose that $f \colon X \to Y$ is a $\gamma$-Borel isomorphism between the topological spaces \( X \) and \( Y \), and let \( \beta \geq 1 \) be such that both $f$ and $f^{-1}$ are $\Siii{0}{\b}{\g}$-measurable. Then if%
\footnote{An ordinal \( \delta \) is additively closed if \( \delta_0+\delta_1 < \delta \) for all \( \delta_0,\delta_1 < \delta \).
Then \( \beta \cdot \omega \) is the smallest additively closed ordinal greater than \( \b \).}
$\alpha\geq \beta\cdot\omega$ we have 
\[
A\in \Siii{0}{\a}{\g}(X) \IFF f[A]\in\Siii{0}{\a}{\g}(Y),
\]
and similarly for \( \Siii{0}{\a}{\g} \) replaced by \( \Piii{0}{\a}{\g} \) and \( \Deee{0}{\a}{\g} \).
Moreover, 
\[
A \in \bigcup_{1 \leq \a < \beta \cdot \omega} \Siii{0}{\a}{\g}(X) \IFF f[A] \in \bigcup_{1 \leq \a < \beta \cdot \omega} \Siii{0}{\a}{\g}(Y).
\]
\end{lemma}

\begin{proof}
Fix any \( \a \geq 1 \), and 
let \( \alpha' \in \Ord \) be such that \( \a = 1+\a'\). Arguing by induction on \( \a' \), one can easily prove that $f[A] \in \Siii{0}{\beta + \alpha'}{\g}(Y) $ for each $A \in \Siii{0}{\a}{\g}(X) $ because \( f^{-1} \) is $\Siii{0}{\b}{\g}$-measurable. If $\alpha \geq \beta \cdot \omega \geq \omega$, then  $\beta + \alpha' = \beta+\alpha = \alpha$, and thus we get $A\in \Siii{0}{\a}{\g}(X) \Rightarrow f[A]\in\Siii{0}{\a}{\g}(Y)$. 
If instead \( \a < \beta \cdot \omega \), then  \( \b+ \a' < \beta \cdot \omega \) and we get 
\[
A \in \bigcup_{1 \leq \a < \beta \cdot \omega} \Siii{0}{\a}{\g}(X) \Rightarrow f[A] \in \bigcup_{1 \leq \a < \beta \cdot \omega} \Siii{0}{\a}{\g}(Y).
\]
 The reverse implications can be proved by repeating the same argument for $f^{-1}$.
\end{proof}

\subsection{When the hierarchy is increasing} \label{subsec:inreasing}

It is desirable that the  $\gamma$-Borel hierarchy is increasing (as it happens, for example, in the countable case for all Polish spaces), in the following sense.

\begin{defin}
We say that the $\gamma$-Borel hierarchy on a topological space on \( X \) is \markdef{increasing} (respectively, \markdef{increasing above level} \( \delta \in \Ord \)) if 
\( \Siii{0}{\a}{\g}(X) \subseteq \Siii{0}{\b}{\g}(X) \) holds for every \( \beta > \alpha \geq 1 \) (respectively, for every \( \beta > \alpha\geq \delta \)).
\end{defin}

In this respect, the only problematic case, as already noticed in~\cite[Lemma 2.2]{AMR19}, is when \( \alpha = 1 \) and \( \beta = 2 \).

\begin{lemma}\label{lem:hierarchy_increasing_above_2}
Let \( X \) be any topological space. Then the $\gamma$-Borel hierarchy on \( X \) is increasing above level $2$, and \( \Siii{0}{1}{\g}(X)\subseteq\Piii{0}{2}{\g}(X)\subseteq \Siii{0}{3}{\g}(X) \). In particular, the \( \gamma \)-Borel hierarchy on \( X \) is increasing if and only if $\Siii{0}{1}{\g}(X)\subseteq\Siii{0}{2}{\g}(X)$. 
\end{lemma}

\begin{remark}\label{rmk:always_true_inclusions}
Note however that certain inclusions hold regardless of whether the $\g$-Borel hierarchy is increasing  or not.
In particular, 
for all ordinals $\alpha\leq \beta$ we have \( \Siii{0}{\a}{\g}(X)\cup \Piii{0}{\a}{\g}(X)\subseteq \Siii{0}{\b}{\g}(X)\cup  \Piii{0}{\b}{\g}(X) \) and \( \Deee{0}{\a}{\g}(X)\subseteq \Deee{0}{\b}{\g}(X) \).  \end{remark}

Since the boldface pointclasses \( \Siii{0}{\a}{\g} \) are hereditary, we get that if the \( \g \)-Borel hierarchy on \( X \) is increasing and \( Y \subseteq X\), then the \( \g \)-Borel hierarchy on \( Y \) is increasing, too. 
The property of being increasing is responsible for the \( \gamma \)-Borel hierarchy behaving as expected with respect to inclusion. Indeed, by Lemma~\ref{lem:hierarchy_increasing_above_2} we have that for every \( \beta > \alpha \geq 2 \)
\begin{equation} \label{eq:inclusionswhenincreasing}
\Siii{0}{\a}{\g}(X)\cup \Piii{0}{\a}{\g}(X) \subseteq \Deee{0}{\b}{\g}(X) \subseteq \Siii{0}{\b}{\g}(X), \Piii{0}{\b}{\g}(X) 
\end{equation}
and
\begin{equation} \label{eq:simplification}
\Siii{0}{\b+1}{\g}(X) = \left( \Piii{0}{\b}{\g}(X) \right)_{\sigma_{<\g}} ,
\end{equation}
but when the \( \gamma \)-Borel hierarchy on \( X \) is increasing, then both~\eqref{eq:inclusionswhenincreasing} and~\eqref{eq:simplification} hold unconditionally at all levels of the \( \g \)-Borel hierarchy.
All these observations will often be used implicitly in proofs.

\begin{remark} \label{rmk:closureunderfinite operations}
It is easy to check that if the \( \g \)-Borel hierarchy on \( X \) is increasing, then for every $\alpha\geq 1$ we have that \( \Siii{0}{\a}{\g}(X) \) is closed under finite intersections and \( \Piii{0}{\a}{\g}(X) \) is closed under finite unions. Together with Remark~\ref{rmk:easyclosure}, this implies that \( \Deee{0}{\a}{\g}(X) \) is an \( \omega \)-algebra.
\end{remark}

In light of the discussion above, it is natural to ask for which topological spaces \( X \) we have $\Siii{0}{1}{\g}(X)\subseteq\Siii{0}{2}{\g}(X)$, so that the corresponding \( \gamma \)-Borel hierarchy is increasing. Notably, it turns out that when \( \gamma = \kappa^+\) such condition is satisfied in virtually all topological spaces that are relevant to generalized descriptive set theory, since they all have weight at most \( \kappa \).

\begin{lemma}\label{lem:increasing}
Let \(X\) be a regular space with \( \weight(X) < \gamma \). Then $\Siii{0}{1}{\g}(X)\subseteq\Siii{0}{2}{\g}(X)$.
\end{lemma} 

\begin{proof}
Let $\B$ be a basis for the topology of \( X \) such that $|\B| < \g$. 
 For any open set $U \subseteq X$, consider the family of closed sets 
$\C_U=\{\cl(B)\mid B\in \B, \cl(B)\subseteq U\}$, which has size at most \( |\B| \). 
Then $\bigcup \C_U\subseteq U$ by definition, while \( U \subseteq \bigcup \C_U \) by regularity of the topology of \( X \). Thus $U\in \Siii{0}{2}{\g}(X)$, and we are done.
\end{proof}

Since from Section~\ref{sec_borel_hierarchy} onward all spaces will be assumed to be regular Hausdorff and of weight at most \( \kappa \), this means that the \( \kappa^+\)-Borel hierarchy on each of them will always be increasing.%
\footnote{For the readers interested in non-regular spaces or spaces of weight greater than \(\k\), we mention that in any case most results of the paper about the \( \k^+\)-Borel hierarchy are still valid if we drop the topological requirements on the space and we explicitly assume $ \Siii{0}{1}{\k^+}(X)\subseteq\Siii{0}{2}{\k^+}(X)$ in order to guarantee that such hierarchy is increasing.}
As for the \( \gamma \)-Borel hierarchy when \( \gamma = \kappa \) is a singular cardinal, the condition $\Siii{0}{1}{\k}(X)\subseteq\Siii{0}{2}{\k}(X)$ is no longer automatically true if \( X \) has weight \( \kappa \), and that is one reason why (a slightly stronger form of) it will explicitly appear as an hypothesis in almost all statements of Section~\ref{sec_alt}. However, notice that since we consider this hierarchy only when \( \cof(\k) < \k \), then by Fact~\ref{lem:NS_implies_increasing_hierarchy} such conditions are verified in most spaces of interest, including all subspaces of \( \pre{\k}{2}\) or \(\pre{\cof(\k)}{\k}\). In the latter cases, one can additionally show that every open set can be written as a union of \( \cof(\k)\)-many \emph{clopen} sets (see the proof of~\cite[Lemma 3.5]{MRP25}).

\begin{fact}\label{fact:open=unionclopen}
If $X\subseteq \pre{\k}{2}$, then any open subset of \(X\) can be written as a union of \(\cof(\k)\)-many clopen sets. 
\end{fact}

\subsection{When the hierarchy collapses} \label{subsec:order}
 A basic parameter measuring the behavior of the \( \gamma \)-Borel hierarchy is its length.
A standard result in classical descriptive set theory is that if $X$ is an uncountable Polish space, then its (\( \omega_1\)-)Borel hierarchy has maximal length, that is, \(\boldsymbol{\Sigma}^0_\a(X) \subsetneq \boldsymbol{\Sigma}^0_\beta(X)\) for all \(1 \leq \a<\b<\o_1\), and thus the length of the hierarchy is \( \omega_1 \). 
When moving to the general case of $\gamma$-Borel hierarchies, we know from~\eqref{upperboundonlength1} that an upper bound for its length is \( \g^* \), the smallest regular cardinal above \( \gamma \). However, for some specific topological space \( X \) it might happen that the \( \gamma \)-Borel hierarchy on \( X \) is shorter than that, in which case the following notion becomes relevant. 

\begin{definition}\label{def:order}
Let $X$ be a topological space. The \markdef{order} of the $\g$-Borel hierarchy on $X$ is
\[
\ord_\gamma(X) = \min\left\{\alpha \in \Ord \mid \Siii{0}{\a}{\g}(X) =\Bor{\g}(X) \right\}.
\]
\end{definition}

As usual, a reference to the topology \( \tau \) of \( X \) will be added to the notation when needed.
Notice that in the definition of \( \ord_\gamma(X) \) we can equivalently replace \( \Siii{0}{\a}{\g}(X) \) with \( \Piii{0}{\a}{\g}(X) \) or \( \Deee{0}{\a}{\g}(X) \), and that if \( \ord_\g(X) = \a \), then \( \Siii{0}{\a}{\g}(X) = \Piii{0}{\a}{\g}(X) \). However, the latter equality is not sufficient to ensure \( \ord_\g(X) \leq \a \): if e.g.\ \( \g = \k \) is a singular cardinal, it is possible that the \( \k \)-Borel hierarchy on a space \( X \) does not collapse, yet \( \Siii{0}{\a}{\k}(X) = \Piii{0}{\a}{\k}(X) \) for many \( 1 \leq \a < \kappa^+ \) (see Theorem~\ref{hierarchy_theorem}).

By the previous discussion, \( \ord_\gamma(X) \leq \gamma^*\) always holds. When \( \ord_\gamma(X) < \gamma^* \) we say that the \( \gamma \)-Borel hierarchy on \( X \) \markdef{collapses}.
   The fact that the \( \g \)-Borel hierarchy collapses automatically transfers to subspaces, and it is independent of the chosen topology if the space has small enough weight.

\begin{proposition} \label{prop:abstractpropertiesofcollapse}
Let \( (X,\tau) \) be a topological space such that \( \ord_\g(X,\tau) < \g^* \).
\begin{enumerate-(1)}
\item \label{prop:abstractpropertiesofcollapse-1}
If \( Y \subseteq X \), then \( \ord_\g(Y, \tau \restriction Y) \leq \ord_\g(X,\tau) < \g^* \).
\item \label{prop:abstractpropertiesofcollapse-2}
If \( \weight(X,\tau) < \gamma^* \), then for every topology \( \tau' \) on \( X \) such that \( \Bor{\g}(X,\tau') = \Bor{\g}(X,\tau) \) we have \( \ord_\g(X,\tau') < \g^* \) too.
\end{enumerate-(1)}
\end{proposition}

\begin{proof}
Item~\ref{prop:abstractpropertiesofcollapse-1} follows from the fact that the pointclasses \( \Bor{\g} \) and \( \Siii{0}{\a}{\g} \) are hereditary.
As for~\ref{prop:abstractpropertiesofcollapse-2}, let \( \B \) be a basis for \( \tau \) with \( |\B| < \g^* \), and let \( \a < \g^* \) be such that \( \ord_\g(X,\tau) = 1+\a \). Since \( \g^* \) is regular and \( \tau \subseteq \Bor{\g}(X,\tau) = \Bor{\g}(X,\tau') \), there is \( 1 \leq \beta < \g^* \) such that \( \B \subseteq \Siii{0}{\b}{\g}(X,\tau') \). By recursion on \( \a \), one thus obtain \( \Siii{0}{1+\a}{\g}(X,\tau) \subseteq \Siii{0}{\b+\a}{\g}(X,\tau') \). Since \(\Bor{\g}(X,\tau') = \Bor{\g}(X,\tau) = \Siii{0}{1+\a}{\g}(X,\tau) \), this proves that \( \ord_\g(X,\tau') \leq \beta + \alpha < \g^*\).
\end{proof}

\begin{corollary} \label{cor:abstractpropertiesofcollapse}
If \( X \) and \( Y \) are two \( \g \)-Borel isomorphic topological spaces of weight smaller than \( \g^* \), then \( \ord_\g(X) < \g^* \IFF \ord_\g(Y) < \g^* \).
\end{corollary}

\begin{proof}
Let \( f \colon X \to Y \) be a \( \g \)-Borel isomorphism. Suppose that \( \ord_\g(X) < \g^* \). Let \( \tau \) be the topology of \( X \), and let \( \tau' \) be the pull-back of the topology of \( Y \) along \( f \). Then \( \Bor{\g}(X,\tau) = \Bor{\g}(X,\tau') \) because \( f \) is a \( \g \)-Borel isomorphism, hence \( \ord_{\g}(X,\tau') < \g^*  \) by Proposition~\ref{prop:abstractpropertiesofcollapse}\ref{prop:abstractpropertiesofcollapse-2}, and thus also \( \ord_\g(Y) < \g^*\) because \( f \colon (X,\tau') \to Y \) is a homeomorphism by definition of \( \tau' \).
The other implication can be proved in a similar way, using \( f^{-1} \) instead of \( f \).
\end{proof}

The following result provides several criteria for the collapse of the \( \gamma \)-Borel hierarchy on a space \( X \).

\begin{proposition}\label{prop:conditions_collapse}
Let $X$ be a topological space. Then for every \( \a \geq 1 \), the following are equivalent:
\begin{enumerate}[label={\upshape (\arabic*)}, leftmargin=2pc]
\item \label{prop:conditions_collapse-1}
$\ord_\gamma(X) \leq \alpha$;
\item \label{prop:conditions_collapse-2} 
for every \( \beta \geq \alpha\), we have \( \Siii{0}{\b}{\g}(X) = \Piii{0}{\b}{\g}(X) = \Deee{0}{\b}{\g}(X) = \Bor{\g}(X)\);
\item \label{prop:conditions_collapse-3} 
for some \( \beta > \alpha \), we have 
\( \Deee{0}{\a}{\g}(X) = \Siii{0}{\b}{\g}(X) \) (equivalently, \( \Deee{0}{\a}{\g}(X) = \Piii{0}{\b}{\g}(X) \));
\item \label{prop:conditions_collapse-4} for some \( \beta > \alpha \), we have 
\( \Siii{0}{\a}{\g}(X) = \Siii{0}{\b}{\g}(X) \) (equivalently, 
\( \Piii{0}{\a}{\g}(X) = \Piii{0}{\b}{\g}(X) \)).
\item \label{prop:conditions_collapse-5}
 the $\gamma$-Borel hierarchy on $X$ is increasing above $\alpha$ and for some \( \beta > \alpha \), one of \( \Siii{0}{\a}{\g}(X) \), \( \Piii{0}{\a}{\g}(X) \), or \( \Deee{0}{\a}{\g}(X) \) coincides with one of \( \Siii{0}{\b}{\g}(X) \) or \( \Piii{0}{\b}{\g}(X) \);
\item\label{prop:conditions_collapse-6}\( \Siii{0}{1}{\g}(X)\cup \Piii{0}{1}{\g}(X) \subseteq \Siii{0}{\a}{\g}(X) \) and the class $\Siii{0}{\alpha}{\g}(X)$ is closed under intersections shorter than $\g$ (equivalently, $\Piii{0}{\alpha}{\g}(X)$ is closed under unions shorter than $\g$);
\item \label{prop:conditions_collapse-7}
\( \Siii{0}{1}{\g}(X)\cup \Piii{0}{1}{\g}(X)\subseteq \Siii{0}{\a}{\g}(X) \) and the class $\Deee{0}{\alpha}{\g}(X)$ is closed under intersections shorter than $\g$ (i.e.\ \( \Deee{0}{\alpha}{\g}(X) \) is a \( \g \)-algebra).
\end{enumerate}
If $\gamma$ is not a singular cardinal, the above conditions are also equivalent to:
\begin{enumerate}[label={\upshape (\arabic*)}, leftmargin=2pc,resume]
\item \label{prop:conditions_collapse-8}  
the $\gamma$-Borel hierarchy on $X$ is increasing above $\alpha$ and for some \( \beta > \alpha \), \( \Deee{0}{\a}{\g}(X) \) coincides with \( \Deee{0}{\b}{\g}(X) \);
\item \label{prop:conditions_collapse-9} 
for some \( \beta > \alpha \), we have \( \Siii{0}{\a}{\g}(X) = \Deee{0}{\b}{\g}(X) \) (equivalently, \( \Piii{0}{\a}{\g}(X) = \Deee{0}{\b}{\g}(X) \));
\item \label{prop:conditions_collapse-10}
$\Siii{0}{\alpha}{\g}$ is selfdual on \( X \), i.e.\ $\Siii{0}{\alpha}{\g}(X) = \Piii{0}{\alpha}{\g}(X)$;
\end{enumerate}
\end{proposition}

\begin{proof}
\ref{prop:conditions_collapse-1}~\( \Rightarrow \)~\ref{prop:conditions_collapse-2}. Without loss of generality, we may assume $\ord_\gamma(X)=\alpha$, as if \ref{prop:conditions_collapse-2} holds for every $\beta\geq  \ord_\gamma(X)$, then it holds for every $\beta\geq\alpha\geq  \ord_\gamma(X)$.
As already observed, \( \Siii{0}{\a}{\g}(X) = \Bor{\g}(X) \) is equivalent to \( \Deee{0}{\a}{\g}(X) = \Bor{\g}(X) \). 
By Remark~\ref{rmk:always_true_inclusions}, for every $\beta \geq \alpha$ we get \( \Deee{0}{\a}{\g}(X) \subseteq \Deee{0}{\b}{\g}(X) \subseteq \Bor{\g}(X) = \Deee{0}{\a}{\g}(X)\), and hence
also \( \Siii{0}{\b}{\g}(X) = \Piii{0}{\b}{\g}(X) =  \Deee{0}{\b}{\g}(X) = \Bor{\g}(X) \).

\ref{prop:conditions_collapse-2}~\( \Rightarrow \)~\ref{prop:conditions_collapse-3}. Obvious.

\ref{prop:conditions_collapse-3}~\( \Rightarrow \)~\ref{prop:conditions_collapse-4}.
Let \( \b > \a \) be such that \( \Deee{0}{\a}{\g}(X) = \Siii{0}{\b}{\g}(X) \),
so that also \( \Deee{0}{\a}{\g}(X) = \Piii{0}{\a}{\g}(X) \).
By Remark~\ref{rmk:always_true_inclusions},  \[
\Deee{0}{\a}{\g}(X)\subseteq  \Siii{0}{\a}{\g}(X)\subseteq \Siii{0}{\b}{\g}(X)\cup \Piii{0}{\b}{\g}(X)=\Deee{0}{\a}{\g}(X),
\]
hence \( \Siii{0}{\a}{\g}(X)= \Deee{0}{\a}{\g}(X)=\Siii{0}{\b}{\g}(X) \).

\ref{prop:conditions_collapse-4}~\( \Rightarrow \)~\ref{prop:conditions_collapse-5}.
By Lemma~\ref{lem:hierarchy_increasing_above_2}, we only need to consider the case \( \a = 1 \) and show that \( \Siii{0}{1}{\g}(X) \subseteq \Siii{0}{2}{\g}(X) \) (the other condition is trivially satisfied because it is a weakening of~\ref{prop:conditions_collapse-4}). Let \( \b > \a = 1 \) be a witness for~\ref{prop:conditions_collapse-4}. Then \( \Piii{0}{1}{\g}(X) \subseteq \Siii{0}{\b}{\g}(X) = \Siii{0}{1}{\g}(X) \), and therefore \( \Siii{0}{1}{\g}(X) = \Deee{0}{1}{\g}(X) \subseteq \Deee{0}{2}{\g}(X) \subseteq \Siii{0}{2}{\g}(X) \) by Remark~\ref{rmk:always_true_inclusions}.
  
\ref{prop:conditions_collapse-5}~\( \Rightarrow \)~\ref{prop:conditions_collapse-6}.
Assume that~\ref{prop:conditions_collapse-5} holds, as witnessed by \( \b > \a \).
Let first \( \boldsymbol{\Gamma}(X) = \Siii{0}{\b}{\g}(X) \), and assume that one of \( \Deee{0}{\a}{\g}(X) \), \( \Piii{0}{\a}{\g}(X) \), or \( \Siii{0}{\a}{\g}(X) \) equals \( \boldsymbol{\Gamma}(X) \).

We first show that, necessarily, \( \Deee{0}{\a}{\g}(X) = \Siii{0}{\a}{\g}(X) = \Piii{0}{\a}{\g}(X) = \boldsymbol{\Gamma}(X) \).
If either \( \Deee{0}{\a}{\g}(X) = \boldsymbol{\Gamma}(X) \)
or \( \Piii{0}{\a}{\g}(X) = \boldsymbol{\Gamma}(X) \), then in both cases \( \boldsymbol{\Gamma}(X) \subseteq \Piii{0}{\a}{\g}(X) \). Therefore \(  \Siii{0}{\a}{\g}(X) \subseteq \boldsymbol{\Gamma}(X) \subseteq \Piii{0}{\a}{\g}(X) \) (since we assumed that the $\g$-Borel hierarchy is increasing above $\alpha$), and thus \(  \Siii{0}{\a}{\g}(X) = \Piii{0}{\a}{\g}(X) = \Deee{0}{\a}{\g}(X) \).
 If instead \( \Siii{0}{\a}{\g}(X) = \boldsymbol{\Gamma}(X) \), then by \( \b > \a \) we get \( \Piii{0}{\a}{\g}(X) \subseteq \boldsymbol{\Gamma}(X) = \Siii{0}{\a}{\g}(X) \), and thus \( \Siii{0}{\a}{\g}(X) = \Piii{0}{\a}{\g}(X) = \Deee{0}{\a}{\g}(X) \) again.

Thus in all three cases we must have \( \Deee{0}{\a}{\g}(X) = \Siii{0}{\a}{\g}(X) = \Piii{0}{\a}{\g}(X) = \boldsymbol{\Gamma}(X) = \check{\boldsymbol{\Gamma}}(X)\). In particular, $\Siii{0}{1}{\g}(X)\cup \Piii{0}{1}{\g}(X)\subseteq \Siii{0}{\a}{\g}(X)\cup \Piii{0}{\a}{\g}(X)= \Siii{0}{\a}{\g}(X)$ by
Remark~\ref{rmk:always_true_inclusions}.
Moreover, if $\seq{A_i}{i<\delta}$ is a sequence of sets in $\Siii{0}{\a}{\g}(X)$ (for some $\delta<\gamma$),
then \( \bigcap_{i<\delta}A_i\in \Siii{0}{\a}{\g}(X) \) because \( \Piii{0}{\a+1}{\g}(X)\subseteq \check{\boldsymbol{\Gamma}}(X)\) by \( 2 \leq \a+1 \leq \b \) and Lemma~\ref{lem:hierarchy_increasing_above_2}, and we already showed that \( \check{\boldsymbol{\Gamma}}(X) = \Siii{0}{\a}{\g}(X) \). 

This concludes the proof in the case \( \boldsymbol{\Gamma}(X) = \Siii{0}{\b}{\g}(X) \).
The case $\boldsymbol{\Gamma}(X)=\Piii{0}{\b}{\g}(X)$ can be treated in a similar way, hence we are done.

\ref{prop:conditions_collapse-6}~\( \Rightarrow \)~\ref{prop:conditions_collapse-7}.
It suffices to show that~\ref{prop:conditions_collapse-6} implies \( \Siii{0}{\a}{\g}(X)  = \Deee{0}{\a}{\g}(X) \).
If $\a=1$, the inclusion \( \Siii{0}{1}{\g}(X) \cup \Piii{0}{1}{\g}(X) \subseteq \Siii{0}{\a}{\g}(X) = \Siii{0}{1}{\g}(X)\) easily yields the desired equality.
Suppose now that \( \a \geq 2 \).
By Lemma~\ref{lem:hierarchy_increasing_above_2} and the assumption \( \Siii{0}{1}{\g}(X) \cup \Piii{0}{1}{\g}(X) \subseteq \Siii{0}{\a}{\g}(X)\), we have $\bigcup_{1 \leq \b<\a} \Siii{0}{\b}{\g}(X)\subseteq \Siii{0}{\a}{\g}(X)$, so that \( \Piii{0}{\a+1}{\g}(X) = \left( \Siii{0}{\a}{\g}(X) \right)_{\d_{<\g}} \). Since the latter equals \( \Siii{0}{\a}{\g}(X) \) by~\ref{prop:conditions_collapse-6}, we get \( \Piii{0}{\a}{\g}(X) \subseteq \Piii{0}{\a+1}{\g}(X) = \Siii{0}{\a}{\g}(X) \) (by \( \a \geq 2 \) and Lemma~\ref{lem:hierarchy_increasing_above_2}), and we are done.

\ref{prop:conditions_collapse-7}~\( \Rightarrow \)~\ref{prop:conditions_collapse-1}.
 Notice that \( \Siii{0}{1}{\g}(X)\cup \Piii{0}{1}{\g}(X)\subseteq \Siii{0}{\a}{\g}(X) \) implies \( \Siii{0}{1}{\g}(X)\cup \Piii{0}{1}{\g}(X)\subseteq \Piii{0}{\a}{\g}(X) \) because \( \Siii{0}{1}{\g}(X)\cup \Piii{0}{1}{\g}(X) \) is closed under complements, and thus \( \Siii{0}{1}{\g}(X)\cup \Piii{0}{1}{\g}(X)\subseteq \Deee{0}{\a}{\g}(X) \). Then \( \Deee{0}{\a}{\g}(X) \) is a \( \g\)-algebra containing all open sets, hence \( \Bor{\g}(X) \subseteq \Deee{0}{\a}{\g}(X) \)  and so \( \ord_\g(X) \leq \a \).

Assume now that $\gamma$ is not a singular cardinal.

\ref{prop:conditions_collapse-1}~\( \Rightarrow \)~\ref{prop:conditions_collapse-8}.
Notice that \( \ord_{\g}(X) \leq \a \) implies \( \Bor{\g}(X) \subseteq \Deee{0}{\a}{\g}(X) \), hence also \( \Deee{0}{\a}{\g}(X) = \Deee{0}{\b}{\g}(X) \) for any \( \b \geq \a \) by Remark~\ref{rmk:always_true_inclusions}. As for increasingness, by Lemma~\ref{lem:hierarchy_increasing_above_2} it is enough to consider that case \( \a = 1 \) and show that \( \Siii{0}{1}{\g}(X) \subseteq \Siii{0}{2}{\g}(X) \). In this case, \ref{prop:conditions_collapse-1} reads as \( \Bor{\g}(X) = \Deee{0}{1}{\g}(X)\), hence \( \Siii{0}{1}{\g}(X) = \Deee{0}{1}{\g}(X) = \Deee{0}{2}{\g}(X) \subseteq \Siii{0}{2}{\g}(X) \) by Remark~\ref{rmk:always_true_inclusions} again.

\ref{prop:conditions_collapse-8}~\( \Rightarrow \)~\ref{prop:conditions_collapse-9}.
Follows from the fact that when the \( \g \)-Borel hierarchy on \( X \) is increasing above \( \a \), $\Deee{0}{\a}{\g}(X) \subseteq \Siii{0}{\a}{\g}(X)\subseteq \Deee{0}{\b}{\g}(X)$.

\ref{prop:conditions_collapse-9}~\( \Rightarrow \)~\ref{prop:conditions_collapse-10}.
Obvious.

\ref{prop:conditions_collapse-10}~\( \Rightarrow \)~\ref{prop:conditions_collapse-1}.
By Remark~\ref{rmk:easyclosure}, the class \( \Siii{0}{\a}{\g}(X) = \Piii{0}{\a}{\g}(X) \) is a \( \gamma \)-algebra because \( \g \) is not a singular cardinal. By Remark~\ref{rmk:always_true_inclusions}, we have \( \Siii{0}{1}{\g}(X) \subseteq \Siii{0}{\a}{\g}(X) \cup \Piii{0}{\a}{\g}(X)=\Siii{0}{\a}{\g}(X)\), therefore \( \Bor{\g}(X) \subseteq \Siii{0}{\a}{\g}(X) \) and thus \( \ord_\gamma(X) \leq \alpha \).
  \end{proof}

Notice that the assumption \enquote{the $\gamma$-Borel hierarchy on $X$ is increasing above $\alpha$} is automatically satisfied if either \( \a \geq 2 \) or \( \Siii{0}{1}{\g}(X) \subseteq \Siii{0}{2}{\g}(X) \), and the assumption \enquote{\( \Siii{0}{1}{\g}(X)\cup \Piii{0}{1}{\g}(X) \subseteq \Siii{0}{\a}{\g}(X) \)} is automatically satisfied if either \( \a \geq 3 \), or \( \a = 2 \) and \( \Siii{0}{1}{\g}(X) \subseteq \Siii{0}{2}{\g}(X) \).
In particular, we get:

\begin{corollary} \label{cor:conditions_collapse}
Let \( X \) be a topological space whose \( \g \)-Borel hierarchy is increasing, and suppose that \( \g \) is not a singular cardinal. Then the \( \g \)-Borel hierarchy on \( X \) collapses if and only if two of the pointclasses appearing in it coincide.
\end{corollary}

Nevertheless, in the general case that the assumptions \enquote{the $\gamma$-Borel hierarchy on $X$ is increasing above $\alpha$} in items~\ref{prop:conditions_collapse-5} and~\ref{prop:conditions_collapse-8} and \enquote{\( \Siii{0}{1}{\g}(X)\cup \Piii{0}{1}{\g}(X)\subseteq \Siii{0}{\a}{\g}(X) \)} in items~\ref{prop:conditions_collapse-6} and~\ref{prop:conditions_collapse-7} are necessary.
For example, set \( \g = \o \), \( X = \RR \), and \( \a = 2 \).
Then \( \Siii{0}{1}{\omega}(\RR) = \Piii{0}{2}{\omega}(\RR) \), \( \Piii{0}{1}{\omega}(\RR) = \Siii{0}{2}{\omega}(\RR) \), and \( \Deee{0}{1}{\omega}(\RR) = \Deee{0}{2}{\omega}(\RR)=\{\emptyset, \RR\} \); moreover, all the mentioned classes  are closed under finite (i.e.\ shorther than \( \o \)) unions and intersections, yet $\ord_\omega(\RR)\geq3$.

Also, the assumption that $\gamma$ is not singular is necessary for the equivalence between~\ref{prop:conditions_collapse-10} and~\ref{prop:conditions_collapse-1} and for Corollary~\ref{cor:conditions_collapse}, as in this case for certain values of \( \alpha \geq 1\) it can happen that \( \Siii{0}{\a}{\g}(X) = \Piii{0}{\a}{\g}(X) \), even though the \( \gamma \)-Borel hierarchy does not collapse (see Theorem~\ref{hierarchy_theorem}). In contrast, we will show that under certain mild assumptions on \( X \) and \( \gamma \), items~\ref{prop:conditions_collapse-8} and~\ref{prop:conditions_collapse-9} are equivalent to~\ref{prop:conditions_collapse-1} also in the singular case (see Proposition~\ref{prop:conditions_collapse_singular}).

Finally, we notice the following specific collapsing phenomenon, that will be relevant for some of the results in Sections~\ref{sec_borel_hierarchy} and~\ref{sec_alt}.

\begin{lemma}\label{lem:door_spaces}
Let \( X \) be a Hausdorff topological space.
If the \( \g \)-Borel hierarchy on \( X \) is increasing, then the following are equivalent:
\begin{enumerate-(1)}
\item\label{door_space-1} 
$\Siii{0}{1}{\g}(X) \cup \Piii{0}{1}{\g}(X) = \Deee{0}{2}{\g}(X)$,
\item\label{door_space-2} 
$\Siii{0}{1}{\g}(X) \cup \Piii{0}{1}{\g}(X) = \Bor{\g}(X)=\powset(X)$,
\item\label{door_space-3} 
$X$ has at most one non-isolated point.%
\footnote{Spaces satisfying one of these equivalent conditions are usually called door spaces. See~\cite[Chapter~2, Exercise~C, page 76]{KelleyMR370454}, which already states the equivalence~\ref{door_space-2}~\( \IFF \)~\ref{door_space-3} for Hausdorff spaces.}
\end{enumerate-(1)}
\end{lemma}

\begin{proof}
\ref{door_space-1}~$\Rightarrow$~\ref{door_space-3}.
We prove the contrapositive.
Suppose that there are two distinct non-isolated points $x,y \in X$.
Since $X$ is Hausdorff, there exist disjoint open sets $U,V\subseteq X$ such that $x\in U$ and $y\in V$. Consider the set
\( A=\{x\}\cup V\setminus\{y\} \).
Then $A$ is neither open nor closed by the choice of \( x \) and \( y \).
However, we have that \( \{ x \} \) and \( \{ y \} \) are closed sets (because \( X \) is Hausdorff), the \( \g \)-Borel hierarchy on \( X \) is increasing by hypothesis, and
$\Deee{0}{2}{\g}(X)$ is an  $\omega$-algebra by Remark~\ref{rmk:closureunderfinite operations}.
Therefore $A\in \Deee{0}{2}{\g}(X) \setminus (\Siii{0}{1}{\g}(X) \cup \Piii{0}{1}{\g}(X))$, and so~\ref{door_space-1} fails.

\ref{door_space-3}~$\Rightarrow$~\ref{door_space-2}.
If \( X \) is discrete, then~\ref{door_space-2} is trivially satisfied, so assume that $X$ has exactly one non-isolated point $x$ and consider an arbitrary \( A \subseteq X \).
If $x\notin A$, then $A$ is open because it consists of isolated points. For the same reason,
if $x\in A$ then $X\setminus A$ is open, and so $A$ is closed.
Therefore \( \powset(X) \subseteq \Siii{0}{1}{\g}(X) \cup \Piii{0}{1}{\g}(X) \) and we are done.

\ref{door_space-2}~$\Rightarrow$~\ref{door_space-1}. Obvious. \qedhere
\end{proof}

\subsection{Decomposition theorems} \label{subsec:decomposition}

In analogy with the Cantor-Bendixson theorem, we show that there is a canonical decomposition of each topological space, separating a part on which the \( \g \)-Borel hierarchy collapses from a part where the \( \g \)-Borel hierarchy does not collapse even locally.
Recall that \( \g^* \) denotes the smallest regular cardinal above \( \g \).

\begin{prop}\label{prop:CB_for_hierarchy_collapse}
Let $X$ be a topological space with \( \weight(X)< \g^* \). 
 Then $X$ can be partitioned into an open subspace $C$ and a closed  subspace $P$ such that the \( \g \)-Borel hierarchy collapses on $C$, but it does not collapse on every nonempty open subset of $P$. Moreover, such a decomposition of \( X \) is unique.
\end{prop}

\begin{proof}
Given a basis \( \B \) of $X$ with $|\B| = \weight(X)$,
let $\B' = \{ B \in \B \mid \ord_\g(B) < \g^* \}$. 
Let $C=\bigcup \B'$ and $P=X\setminus C$. We claim that \( C \) and \( P \) are as required.

For each $B\in \B'$, let $1 \leq \alpha_B < \g^*$ be such that $\Siii{0}{\a_B}{\g}(B) = \Bor{\g}(B)$, and let $\alpha=\sup\{\alpha_B + 1\mid B\in \B'\}$. In this way, $\Deee{0}{\a_B}{\g}(B) = \Bor{\g}(B)$ and \( 2 \leq \alpha < \g^* \). 
  Take now any \( A \in \Bor{\g}(C) \). For every \( B \in \B'\), we have \( A \cap B \in \Bor{\g}(B) = \Deee{0}{\a_B}{\g}(B) \subseteq \Deee{0}{\a}{\g}(B) \subseteq \Piii{0}{\a}{\g}(B) \) by Lemma~\ref{lem:hierarchy_increasing_above_2}. Since \( \Piii{0}{\a}{\g} \) is hereditary, there is \( A'_B \in \Piii{0}{\a}{\g}(C) \) such that \( A'_B \cap B = A \cap B \),
and since \( B \in \Siii{0}{1}{\g}(C) \subseteq \Piii{0}{2}{\g}(C) \subseteq \Piii{0}{\a}{\g}(C) \) by Lemma~\ref{lem:hierarchy_increasing_above_2}, we get that \( A'_B \cap B \in \Piii{0}{\a}{\g}(C) \) too.
It follows that \( A = \bigcup_{B \in \B'} (A'_B \cap B)  \) is a union of at most \( \weight(X) \)-many \( \Piii{0}{\a}{\g}(C) \) sets. By our hypotheses, \( \weight(X) \leq \gamma \). We distinguish two cases.
If \( \weight(X) < \g\), then \( A \in \Siii{0}{\a+1}{\g}(C) \subseteq\Siii{0}{\a+3}{\g}(C) \).
If \( \weight(X) = \g \), then \( \gamma \) needs to be a singular cardinal. Write \( A = \bigcup_{i < \g} A_i \) with \( A_i \in \Piii{0}{\a}{\g}(C) \), and fix an increasing sequence of ordinals \( \seq{\gamma_j}{j < \cof(\gamma)} \) cofinal in \( \gamma \). Then \( A'_j = \bigcup_{i < \gamma_j} A_i \in \Siii{0}{\a+1}{\g}(C) \subseteq \Piii{0}{\a+2}{\g}(C)  \) for every \( j < \cof(\g) \), hence \( A = \bigcup_{j < \cof(\gamma)} A'_j \in \Siii{0}{\a+3}{\g}(C) \) because \( \cof(\gamma) < \gamma \). Since we showed that in all cases \( \Bor{\g}(C) \subseteq \Siii{0}{\a+3}{\g}(C) \), this means that \( \ord_\g(C) \leq \a+3 < \g^* \).

Next suppose towards a contradiction that there is an open set \( U \subseteq X \) such that \( U \cap P \neq \emptyset \) and \( \ord_\g(U \cap P) < \g^* \). Without loss of generality we may assume that \( U \in \B \). Let \( \b = \max \{ 3,  \ord_\g(C), \ord_\g(U \cap P) \}\), so that \( 3 \leq \b < \g^* \) and \( \ord_\g(U \setminus P) \leq \b \) too by Proposition~\ref{prop:abstractpropertiesofcollapse}\ref{prop:abstractpropertiesofcollapse-1}.
Every \( A \subseteq U  \) can be written as \( A = (A \cap P) \cup (A \setminus P) \). Since \( U \cap P, U \setminus P \in \Deee{0}{3}{\g}(U) \subseteq \Deee{0}{\b}{\g}(U) \), using that both \( \Bor{\g} \) and \( \Piii{0}{\b}{\g} \) are hereditary pointclasses we easily get that each \( A \in \Bor{\g}(U) \) is the union of two \( \Deee{0}{\b}{\g}(U) \) sets. This shows that \( \ord_\g(U) \leq \beta+1 < \g^*\), and hence \( U \in \B' \), contradicting \( U \cap P \neq \emptyset \).

For the uniqueness part, suppose that \( C' \) and \( P' \) are disjoint subsets of \( X \) such that \( X = C' \cup P' \), \( C' \) is open, \( \ord_\g(C') < \g^* \), and \( \ord_\g(U \cap P') = \g^* \) for every open \( U \subseteq X \) such that \( U \cap P' \neq \emptyset \).
Then we must have \( C' \subseteq C \), as \( C \) can equivalently be described as the union of all open sets \( V \subseteq X \) such that \( \ord_\g(V) < \g^*\). If the inclusion were proper, then \( C \cap P' \neq \emptyset \) because \( P' = X \setminus C' \). But then \( \ord_\g (C \cap P') \leq \ord_\g(C) < \g^*\) by Proposition~\ref{prop:abstractpropertiesofcollapse}\ref{prop:abstractpropertiesofcollapse-1}, contradicting the choice of \( P' \). This shows that \( C' = C \), and hence also \( P' = P \).
\end{proof}

Proposition~\ref{prop:CB_for_hierarchy_collapse} admits a refined level-by-level version that, although a bit more technical, will be useful later on.

\begin{prop}\label{prop:CB_for_hierarchy_collapse_successor_case} 
Suppose that the \( \g \)-Borel hierarchy on \( X \) is increasing. 
Let \( \l < \g^*\) and \( 1 \leq \beta < \alpha < \g^*\) be such that both \( \Siii{0}{\b}{\g}(X) \) and \( \Siii{0}{\a}{\g}(X) \) are closed under unions of size \( \lambda \),
and let $\B\subset \Siii{0}{\b}{\g}(X)$ be such that $ |\B| \leq \lambda$.
Then $X$ can be partitioned into a $\Siii{0}{\b}{\g}(X)$ set \( C \) and a $\Piii{0}{\b}{\g}(X)$ set $P = X \setminus C$ such that $\ord_\g(C)\leq \a$, while $\ord_\g(B\cap P)=\ord_\g(B)>\a$ for every $B\in \B$ such that $B\cap P \neq \emptyset$.
Moreover, if $\ord_\g(X)>\a$, then $P \neq \emptyset$ and $\ord_\g(X)=\ord_\g(P)>\a$.
\end{prop}

\begin{proof} 
The proof follows closely that of Proposition~\ref{prop:CB_for_hierarchy_collapse}.
Let $\B' = \{ B \in \B \mid \ord_\g(B)\leq \a \}$, and let
$C=\bigcup \B'$ and $P=X\setminus C$. Then $C\in\Siii{0}{\b}{\g}(X)$ because we assumed that \( \Siii{0}{\b}{\g}(X) \) is closed under unions of $\l$, and thus $P\in \Piii{0}{\b}{\g}(X)$.

We know that \( \Siii{0}{\a}{\g}(X) \) is closed under finite intersections (Remark~\ref{rmk:closureunderfinite operations}) and unions of size \( \l \) (by hypothesis), hence the same is true of \( \Siii{0}{\a}{\g}(C) \) because \( \Siii{0}{\a}{\g} \) is hereditary. This easily gives that \( \ord_\g(C) \leq \alpha \): if \( A \in \Bor{\g}(C) \) then for every \( B \in \B' \) we have \( A \cap B \in \Bor{\g}(B) = \Siii{0}{\a}{\g}(B) \), hence there is \( A_B \in \Siii{0}{\a}{\g}(C) \) such that \( A \cap B = A_B \cap B \), and since \( B \in \Siii{0}{\a}{\g}(C)\) too (because \( B \in \Siii{0}{\b}{\g}(X) \subseteq \Siii{0}{\a}{\g}(X) \) and \( B \subseteq C\)), we obtain \( A = \bigcup_{B \in \B'} (A_B \cap B) \in \Siii{0}{\a}{\g}(C) \).

Now consider \( D \in \Bor{\g}(X) \), and let $\a'=\ord_\g(D\cap P)$. 
For every \( A \in \Bor{\g}(D) \) we get \( A \cap C \in \Siii{0}{\a}{\g}(D \cap C) \) (since \( \ord_\g(D \cap C) \leq \ord_\g(C) \leq \alpha \)) and \( A \cap P \in \Siii{0}{\a'}{\g}(D \cap P) \).
Since \( P \in \Piii{0}{\b}{\g}(X) \subseteq \Deee{0}{\a}{\g}(X) \), the usual computation yields \( A = (A \cap C) \cup (A \cap P) \in \Siii{0}{\max(\a,\a')}{\g}(D)\). Hence 
\begin{equation}\label{eq:ord=max}
\ord_\g(D \cap P) \leq \ord_\g(D) \leq \max\{ \a,\ord_\g(D\cap P) \}.
\end{equation}

Applying~\eqref{eq:ord=max} to $D=B\in \B$, we get that if \( B \cap P \neq \emptyset \), then \( \ord_{\g}(B \cap P) > \a \) (as otherwise \( \ord_\g(B) \leq \a \), and so \( B \subseteq C \) by construction), and therefore also \( \ord_\g(B \cap P) = \ord_\g(B) > \a \).

If instead we apply~\eqref{eq:ord=max} to $D=X$, we get that if $\ord_\g(X)>\alpha$ then also $\ord_\g(X)=\ord_\g(P)>\alpha$, which in particular implies $P \neq \emptyset$.  
\end{proof}

Notice that, in general, the sets \(C\) and \(P\) from Proposition~\ref{prop:CB_for_hierarchy_collapse_successor_case} are not unique. For example, when \(\beta\geq 3\), removing a closed discrete subspace from \(P\) and adding it to \(C\) still yields a partition with the required properties. 
However, the partition obtained in the proof of Proposition~\ref{prop:CB_for_hierarchy_collapse_successor_case} is optimal in the following sense: for every other partition of $X$ into a $\Siii{0}{\b}{\g}(X)$ set \( C' \) and a $\Piii{0}{\b}{\g}(X)$ set $P' = X \setminus C'$ such that $\ord_\g(C')\leq \a$, while $\ord_\g(B\cap P')=\ord_\g(B)>\a$ for every $B\in \B$ with $B\cap P' \neq \emptyset$, we have \( C \subseteq C' \) and \( P' \subseteq P \).

\subsection{Universal sets} \label{subsec:universalsets}

We now turn our attention to the existence of universal sets for the classes appearing in the \( \g \)-Borel hierarchy. The key lemma, which is essentially folklore (see e.g.\ the proof of~\cite[Theorem 22.3]{KEC95}), is the following.

\begin{lemma} \label{lem:universalinborelhierarchy}
Let \( \mu \leq \lambda \) be infinite cardinals, and let \( X \) be a topological space.
\begin{enumerate-(1)}
\item \label{lem:universalinborelhierarchy-1}
If \( \weight(X) \leq \l \), then there is a \( \pre{\l}{2}\)-universal set \( \U \) for \( \Siii{0}{1}{\g}(X) \).
\item \label{lem:universalinborelhierarchy-2}
Suppose that for each \( \b < \mu \) we are given a boldface pointclass
\( \boldsymbol{\Gamma}_\b \) such that \( \emptyset \in \boldsymbol{\Gamma}_\b(X) \) and \( \boldsymbol{\Gamma}_\b(X) \) admits a \( \pre{\l}{2}\)-universal set \( \U_\b \). Then there is also a \( \pre{\l}{2}\)-universal set \( \U \) for 
\[
\left( \bigcup_{\b < \mu} \boldsymbol{\Gamma}_\b(X) \right)_{\sigma_\mu}.
\]
\end{enumerate-(1)}
\end{lemma}

\begin{proof}
\ref{lem:universalinborelhierarchy-1}
Fix  a basis \(\{U_\a \mid \a<\l \}\) for \(X\), and let  
\[
\U = \bigcup_{s\in \pre{<\l}{2}} \left(\clopen{s}\times \bigcup \{U_\a \mid \a<\leng{s}, s(\a)=1\} \right) .
\]
Clearly, \(\U \in \Siii{0}{1}{\g}(\pre{\l}{2} \times X)\). 
Given \(U \subseteq X\) open, let \(y \in \pre{\l}{2}\) be such that \(y(\a)=1 \IFF U_\a \subseteq U\): then \(\U_y=U\) because \(U=\bigcup\{U_\a \mid y(\a)=1\}\) by choice of \( y \). This shows that $\U$ is \(\pre{\l}{2}\)-universal for \(\Siii{0}{1}{\g}(X)\).

\ref{lem:universalinborelhierarchy-2}
Let \( \langle \cdot , \cdot \rangle \colon \l \times \l \to \l\) be the usual G\"odel pairing function, and let \( ( \cdot)_0,(\cdot)_1 \colon \l \to \l \) be its left and right inverses, so that \( \langle (\delta)_0,(\delta)_1 \rangle = \delta \) for every \( \delta < \lambda \). Recall also that since \( \mu \) is a cardinal, then \( \langle \cdot , \cdot \rangle \) maps \( \mu \times \mu \) onto \( \mu \). 

For any \( \d < \mu \), let  $f_\d \colon \pre{\l}{2} \to \pre{\l}{2}$ be defined by \( f_\d(x)(i) = x \left(  \langle \d, i \rangle  \right)\) for every \( i < \l \).
Notice that each $f_{\d}$ is continuous.  Define $\U \subseteq \pre{\l}{2} \times X $ by 
\[
(y, x) \in \mathcal{U} \IFF \exists \d < \mu \, [(f_{\d}(y), x) \in \U_{(\d)_0} ]. 
\]
For each \(\d<\mu \), the set $\{(y, x) \in \pre{\l}{2} \times X \mid (f_{\d}(y), x) \in \mathcal{U}_{(\d)_0}\} $ belongs to $\boldsymbol{\Gamma}_{(\d)_0}(\pre{\l}{2}\times X)$ because it is the preimage of $\U_{(\d)_0}$ under the continuous function $f_\d \times \id_X$, thus $\U \in ( \bigcup_{\b < \mu} \boldsymbol{\Gamma}_\b(\pre{\l}{2}\times X) )_{\sigma_\mu}$. 
Let now $A = \bigcup_{i < \mu} A_i$ for \( A_i \in \bigcup_{\b < \mu} \boldsymbol{\Gamma}_\b(X) \).
Let \(\iota \colon \mu \to \mu \) be an injective map such that \( A_i \in \boldsymbol{\Gamma}_{(\iota(i))_0}(X) \).
For each \( \delta < \mu \), let \( y_\delta \in \pre{\l}{2}\) be such that \( (\U_{(\delta)_0})_{y_\d} = A_i \) if \( \delta = \iota(i) \) for some \( i < \mu \), and \( (\U_{(\delta)_0})_{y_\d} = \emptyset\) otherwise. For \( \mu \leq \delta < \l \), pick an arbitrary \( y_\delta \in \pre{<\l}{2}\). Finally, let \( y\) be the only element of \( \pre{\l}{2} \) such that \( f_\delta(y) = y_\delta \) for every \( \delta < \l \). 
Since \( \U_y = \bigcup_{\delta<\mu} (\U_{(\delta)_0})_{y_\d} \), we get $\U_{y}=A$ and we are done.
\end{proof}

Recall once again that \( \g^* \) is the least regular cardinal above \( \g \), and that it is an upper bound for \( \ord_\g(X) \).

\begin{corollary} \label{cor:universalinborelhierarchy}
Let \( \l \) be an infinite cardinal, and let \( X \) be a topological space such that \( \weight(X) \leq \l \).
Assume that \( \gamma \) is not a limit cardinal, and that \( \gamma \leq \l^+ \). Then for every \( 1 \leq \a < \g^* \) there are \( \pre{\l}{2}\)-universal sets for both \( \Siii{0}{\a}{\g}(X) \) and \( \Piii{0}{\a}{\g}(X) \).
\end{corollary}

\begin{proof}
By the assumption on \( \g \), without loss of generality we may assume that \( \g = \mu^+ \) is a successor cardinal (see Remark~\ref{rmk:refining}), so that \( \mu \leq \l \). In this situation, \( \Siii{0}{\a}{\g}(X) = \left( \bigcup_{1 \leq \beta < \alpha } \Piii{0}{\b}{\g}(X) \right)_{\sigma_\mu}\) for every \( \a > 1 \).

We proceed by induction on \( 1 \leq \a < \g^* = \mu^+ \). If \( \a = 1 \), Lemma~\ref{lem:universalinborelhierarchy}\ref{lem:universalinborelhierarchy-1} provides us with a \( \pre{\l}{2}\)-universal set \( \U \) for \( \Siii{0}{1}{\g}(X) \), and hence \( \U^c \) is \( \pre{\l}{2} \)-universal for \( \Piii{0}{1}{\g}(X) \).
Suppose now that \( \a > 1 \). Let \( \rho \colon \mu \to (\a \setminus \{ 0 \}) \) be a surjection, and for every \( \b < \mu \) let \( \boldsymbol{\Gamma}_\b = \Piii{0}{\rho(\b)}{\g}(X) \). Then \( \Siii{0}{\a}{\g}(X) = ( \bigcup_{\beta < \mu } \boldsymbol{\Gamma}_\b(X) )_{\sigma_\mu} \). The inductive hypothesis ensures that we can apply Lemma~\ref{lem:universalinborelhierarchy}\ref{lem:universalinborelhierarchy-2} and obtain a \( \pre{\l}{2}\)-universal set \( \U \) for \( \Siii{0}{\a}{\g}(X) \), so that \( \U^c \) is universal for \( \Piii{0}{\a}{\g}(X) \) and we are done.
\end{proof}

\subsection{More spaces}

Lemma~\ref{lem:change_topology_hierarchy} and Lemma~\ref{lem:hierarchy_increasing_above_2} can be used to show that the infinite levels of the \( \g \)-Borel hierarchy on a space \( X \) of weight smaller than \( \g^* \) do not depend too much on the chosen topology, in the sense that the original topology of \( X \) can be enhanced without altering the classes \( \Siii{0}{\a}{\g}(X)\) (and hence also \( \Piii{0}{\a}{\g}(X)\) and \( \Deee{0}{\a}{\g}(X) \)) for \( \a \geq \o \). For example, we can increase the additivity of the space under certain hypotheses on its weight, which are satisfied if e.g.\ \( \weight(X) = \kappa \) and \( \nu = \cof(\kappa) \).

\begin{proposition}\label{lem:hierarchy_for_mu_additive_topology} 
Let $(X,\tau)$ be a topological space of weight $\lambda$, and let \( \nu \) be a cardinal such that \( \lambda^{<\nu} < \gamma \). Let $\tau'$ be the smallest $\nu$-additive topology refining $\tau$. 
      Then for every $\alpha\geq \omega$, we have
\[
\Siii{0}{\a}{\g}(X,\tau)= \Siii{0}{\a}{\g}(X,\tau'),
\]
and analogously for \( \Piii{0}{\a}{\g}\) and \( \Deee{0}{\a}{\g}\).
Moreover, 
\[
\bigcup\nolimits_{1 \leq n < \omega} \Siii{0}{n}{\g}(X,\tau) = \bigcup\nolimits_{1 \leq n < \omega } \Siii{0}{n}{\g}(X,\tau').
\]
\end{proposition}

\begin{proof}
By Lemma~\ref{lem:change_topology_hierarchy}, it is enough to check that the identity $\id_X \colon (X,\tau)\to (X,\tau')$ is a $\g$-Borel isomorphism such that both it and its inverse are $\Siii{0}{3}{\g}$-measurable.
Fix a basis \( \B \) for \( \tau \) of size \( \l \). Then a basis for \( \tau'\) is \( \B' = \{\bigcap \A\mid \A\subset \B, |\A|<\nu\} \); in particular, \( \weight(X,\tau') \leq \lambda^{<\nu} \).
On the one hand, $\id_X^{-1} \colon (X,\tau')\to (X,\tau)$ is continuous because $\tau' \supseteq \tau$, and thus it is also $\Siii{0}{3}{\g}$-measurable.
On the other hand, since \( \nu \leq \lambda^{<\nu} < \gamma \) we have that \( \B' \subseteq \Piii{0}{2}{\gamma}(X,\tau) \), thus \( \tau' \subseteq \Siii{0}{3}{\g}(X,\tau) \) and \( \id_X \) is \( \Siii{0}{3}{\g}\)-measurable too.
\end{proof}

Another evidence of the fact that, when restricting to the infinite levels of the \( \g \)-Borel hierarchy, the only relevant parameter is the weight of the space is given by the following result.

\begin{prop}\label{prop:embed_T_0_spaces_into_Cantor}
 Let $(X,\tau)$ be a $T_0$ topological space of weight at most $\lambda$, and suppose that \( \max \{ \mu, \cof(\mu)^+ \} \leq \gamma \) for \( \mu = 2^{< \lambda} \). Then \( \tau \) can be refined to a topology \( \tau' \) such that:
\begin{enumerate-(1)}
\item 
\( (X,\tau')\) embeds into \( \pre{\l}{2} \), hence \( \tau' \) is regular Hausdorff, \( (\omega,\cof(\lambda)) \)-Nagata--Smirnov, \( \cof(\l) \)-additive, and \( \weight(X,\tau') \leq \mu \);
\item 
$\tau'$ has a basis of size at most \( \mu \) that consists of \(\Deee{0}{3}{\l}(X,\tau) \) sets;  
 \item 
$\Siii{0}{\a}{\g}(X,\tau)= \Siii{0}{\a}{\g}(X,\tau')$ for all $\alpha\geq \omega$, $\Bor{\g}(X,\tau)=\Bor{\g}(X,\tau')$, and also \( \bigcup_{1 \leq n < \omega} \Siii{0}{n}{\g}(X,\tau) = \bigcup_{1 \leq n < \omega } \Siii{0}{n}{\g}(X,\tau') \).
\end{enumerate-(1)}
           \end{prop}

\begin{proof}
Let $\B=\{B_\alpha\mid \alpha<\lambda\}$ be an enumeration (possibly with repetitions) of a basis for $\tau$. Let $f \colon X\to \pre{\lambda}{2}$ be defined by $f(x)(\alpha)=1$ if and only if $x\in B_\alpha$. 
Since $X$ is $T_0$, the map \( f \) is injective, and it is open on its image because
\( f[B_\alpha]=\{y \in f[X] \mid y(\alpha) = 1\} \)
is open in $f[X]$.
For every basic clopen set $\clopen{s}$ of $\pre{\lambda}{2}$ we have that \( f^{-1}(\clopen{s})= A\cap C \), where 
\begin{align*}
A & =\bigcap \{B_\alpha\mid \alpha<\leng{s} \wedge s(\alpha)=1\}, \quad \text{and} \\
C& =\bigcap \{X\setminus B_\alpha\mid \alpha<\leng{s} \wedge  s(\alpha)=0\}. \end{align*}
Then $A \in \Piii{0}{2}{\l}(X,\tau)$ and $C \in \Piii{0}{1}{\l}(X,\tau)$, hence \( A \cap C \in \Deee{0}{3}{\l}(X,\tau) \) by Lemma~\ref{lem:hierarchy_increasing_above_2}.
Let \( \tau' \) be the topology generated by \( \B' =  \B \cup \{ f^{-1}([s]) \mid s \in \pre{<\l}{2} \} \).
Then by the above computation and Lemma~\ref{lem:hierarchy_increasing_above_2}, the basis of \( \tau' \) obtained by closing \( \B' \) under finite intersections has size at most \( 2^{< \l} = \mu  \) and consists of \( \Deee{0}{3}{\l}(X,\tau) \) sets. Moreover, \( f \colon (X,\tau') \to \pre{\l}{2}\) is an embedding because \( f^{-1}([s]) \in \B' \) for every \( s \in \pre{<\l}{2}\) and \( f[B] \) is open in \( f[X] \) for every \( B \in \B' \).

Finally, consider the map \( \id_X \colon (X,\tau) \to (X,\tau') \). Its inverse \( \id_X^{-1} \colon (X,\tau') \to (X,\tau) \) is continuous because \( \tau' \supseteq \tau \). Since \( \gamma \geq \mu \geq \lambda \) and \( |\B'| \leq \mu \), every \( \tau' \)-open set \( U \) is a union of \( \mu \)-many \( \Deee{0}{3}{\mu}(X,\tau) \) sets by Remark~\ref{rmk:refining}. We distinguish two cases. If \( \mu \) is regular, then \( \g \geq \cof(\mu)^+ = \mu^+ > \mu\), and so \( U \in \Siii{0}{3}{\mu^+}(X,\tau) \subseteq \Siii{0}{3}{\g}(X,\tau)\) by Remark~\ref{rmk:refining} again.
If instead \( \mu \) is singular, then \( U \in \Siii{0}{5}{\g}(X,\tau)\). Indeed, if \( U = \bigcup_{i < \mu} U_i \) with \( U_i \in \Deee{0}{3}{\mu}(X,\tau) \subseteq \Deee{0}{3}{\g}(X,\tau) \) and \( \seq{\mu_j}{j < \cof(\mu)} \) is an increasing sequence of ordinals cofinal in \( \mu \), then \( V_j = \bigcup_{i < \mu_j} U_i \in \Siii{0}{3}{\g}(X,\tau) \subseteq \Piii{0}{4}{\g}(X,\tau) \) for every \( j < \cof(\mu) \) (because \( \gamma \geq \mu > \mu_j \)), and hence \( U = \bigcup_{j < \cof(\mu)} V_j \in \Siii{0}{5}{\g}(X,\tau) \) because \( \g > \cof(\mu) \). This shows that \( \id_X \) is either \( \Siii{0}{3}{\g}\)-measurable or \( \Siii{0}{5}{\g} \)-measurable. In all cases, by Lemma~\ref{lem:change_topology_hierarchy} we get that
$\Siii{0}{\a}{\g}(X,\tau)= \Siii{0}{\a}{\g}(X,\tau')$ for all $\alpha\geq \omega$,  \( \bigcup_{1 \leq n < \omega} \Siii{0}{n}{\g}(X,\tau) = \bigcup_{1 \leq n < \omega } \Siii{0}{n}{\g}(X,\tau') \), and hence also $\Bor{\g}(X,\tau)=\Bor{\g}(X,\tau')$.
   \end{proof}

To see the relevance of Proposition~\ref{prop:embed_T_0_spaces_into_Cantor} in the context of this paper, notice that setting \( \lambda = \kappa \) we obtain that \( \mu = \kappa \) too (because we assumed \( 2^{< \kappa} = \kappa\)), and thus the above result can be applied with \( \gamma = \kappa^+ \), and also with \( \gamma = \kappa \) if \( \kappa \) is singular. 
On the one hand, this implies that for results that do not depend on the initial levels of the $\kappa^+$-Borel (or of the \( \kappa \)-Borel) hierarchy, without loss of generality we can work only with subspaces of $\pre{\kappa}{2}$. This applies e.g.\ to the question of whether such hierarchies collapse (Proposition~\ref{prop:abstractpropertiesofcollapse}), and thus it justifies our apparent restriction in scope while analyzing the connection between the \( \kappa \)-Perfect Set Property and the collapse of the \( \kappa^+ \)-Borel hierarchy in Section~\ref{sec: psp section}.
On the other hand, Proposition~\ref{prop:embed_T_0_spaces_into_Cantor} guarantees that all results contained herein touching on the infinite levels of the $\kappa^+$-Borel hierarchy hold in full generality for all $T_0$ topological spaces of weight at most \( \kappa \), without necessarily assuming them to be Hausdorff, nor regular.

As a corollary of Proposition~\ref{prop:embed_T_0_spaces_into_Cantor}, we obtain an elegant proof of the following well-known fact, which says that \( \kappa^+ \)-Borel spaces can be identified with the subspaces of the generalized Cantor space \( \pre{\k}{2} \).

\begin{proposition} \label{fact: wlog subsets cantor space}
The following are equivalent, for any $\kappa^+$-algebra \(\B\) on a set $X$:  
\begin{enumerate-(1)}
\item\label{k-Borel space-1} 
$(X,\B)$ is a $\kappa^+$-Borel space, i.e. $\B$ separates points and is generated by a subfamily \( \A \) of size at most $\kappa$;
\item\label{k-Borel space-2'} 
$\B$ is generated by a \( T_0 \) topology of weight at most $\kappa$;
\item\label{k-Borel space-2} 
$\B$ is generated by a regular Hausdorff topology of weight at most $\kappa$;
\item\label{k-Borel space-3} 
there is an injection $f \colon X \to \pre{\k}{2}$ such that \( A \in \B \IFF f[A] \in {\Bor{\k^+}(f[X])} \), for every \( A \subseteq X\).
 \end{enumerate-(1)}
\end{proposition}

\begin{proof}
The implications \ref{k-Borel space-3}~$\Rightarrow$~\ref{k-Borel space-2}, \ref{k-Borel space-2}~$\Rightarrow$~\ref{k-Borel space-2'}, and \ref{k-Borel space-2'}~$\Rightarrow$~\ref{k-Borel space-1} are obvious. To see \ref{k-Borel space-1}~$\Rightarrow$~\ref{k-Borel space-3},
notice that since $\B$ separates points, so does $\A$.
Then the topology $\tau$ generated by $\A$ is $T_0$. Thus an application of Proposition~\ref{prop:embed_T_0_spaces_into_Cantor} with \( \lambda = \mu = \kappa \) and \( \gamma = \kappa^+\) yields the desired result.
\end{proof}

\section{\texorpdfstring{The \( \k^+\)-Borel hierarchy}{The kappa+-Borel hierarchy}}\label{sec_borel_hierarchy}

As customary in generalized descriptive set theory, from this point onward (and unless stated otherwise),
\begin{quotation}
\emph{all topological spaces are  assumed to be regular Hausdorff and of weight at most \( \kappa \).}
\end{quotation}

In this section, we fix once and for all such a space \( X \), and we study its $\kappa^+$-Borel hierarchy. Notice however that, in view of Proposition~\ref{prop:embed_T_0_spaces_into_Cantor}, all results would remain valid for arbitrary \( T_0 \) spaces of weight at most \( \kappa \) if we restrict the attention to the infinite levels \( \alpha \geq \omega \). Also, recall that the \( \k^+\)-Borel hierarchy on \( X \) is increasing by Lemma~\ref{lem:increasing} (applied with \( \g = \k^+\)). This will be tacitly used throughout the section.

We begin with the closure properties of the pointclasses appearing in the \( \k^+ \)-Borel hierarchy on \( X \).
Notice that \( \ord_{\kappa^+}(X) \leq \k^+ \) because \( \k^+ \), being a successor cardinal, is regular.

\begin{prop}\label{k+_hierarchy_closure}
  Given any \( 1 < \a < \k^+ \),  let \( \widehat{\alpha} = \cof(\k) \) if \( \a \) is a successor ordinal, and \( \widehat{\a} = \cof(\a) \) if \( \a \) is limit. Then:

\begin{enumerate-(1)} \item \label{k+_hierarchy_closure-1}
\(\Siii{0}{\a}{\k^+}( X) \) is closed under unions of length \(\kappa\) and intersections of size smaller than \( \widehat{\a} \);
      \item \label{k+_hierarchy_closure-2}
\(\Piii{0}{\a}{\k^+}( X) \) is closed under intersections of length \(\kappa\) and unions of size smaller than \( \widehat{\a} \);
      \item \label{k+_hierarchy_closure-3}
\(\Deee{0}{\a}{\k^+}( X) \) is closed under complements and both unions and intersections of size smaller than \( \widehat{\a} \), that is, \(\Deee{0}{\a}{\k^+}( X) \) is an \( \widehat{\a} \)-algebra.
      \end{enumerate-(1)} 
Furthermore, the same is true for $\a=1$ if \(X\) is $\cof(\k)$-additive.%
\footnote{More generally, the closure properties of \( \Siii{0}{1}{\k^+}(X) \) under intersections (or, equivalently, of \( \Piii{0}{1}{\k^+}(X) \) under unions) precisely correspond, by definition, to the additivity of the space.}
\end{prop}

\begin{proof}
Part~\ref{k+_hierarchy_closure-2} follows from~\ref{k+_hierarchy_closure-1} by taking complements, and part~\ref{k+_hierarchy_closure-3} follows from~\ref{k+_hierarchy_closure-1} and~\ref{k+_hierarchy_closure-2}. Thus it is enough to prove~\ref{k+_hierarchy_closure-1}.

The fact that \( \Siii{0}{\a}{\g}(X) \) is closed under unions of size \( \k \) already follows from Remark~\ref{rmk:easyclosure} applied with \( \g = \k^+ \), so let us consider closure under intersections shorter than \( \widehat{\a} \). We argue by induction on \( 1 < \a < \k^+\).

Consider first the case where \( \a = \b+1 \) is a successor ordinal (which covers in particular the base case \( \a = 2 \)). 
Let \(\l < \cof(\k)\) and \(\seq{A_i}{i<\lambda}\) be a sequence of sets in \(\Siii{0}{\a}{\k^+}( X) \). By definition, for each \(i<\lambda\) we have that \(A_i=\bigcup_{j<\k} A_{i,j}\) with \(A_{i,j} \in \Piii{0}{\b}{\k^+}( X)  \) (here we implicitly use that the \( \k^+ \)-Borel hierarchy on \( X \) is increasing). Then, we can write 
\[
\bigcap_{i<\l} A_i= \bigcap_{i<\l} \left(\bigcup_{j<\k}A_{i,j}\right)= \bigcup_{s \in \pre{\l}{\k} } \left(\bigcap_{i<\l}A_{i,s(i)} \right). 
\]
Since \( \Piii{0}{\b}{\k^+}(X) \) is closed under intersections of size \( \k\) (Remark~\ref{rmk:easyclosure}), we have \(\bigcap_{i<\l}A_{i,s(i)} \in \Piii{0}{\b}{\k^+}( X) \) for all \(s \in \pre{\l}{\k}\). 
By choice of \( \l \), we get \(  |\pre{\l}{\k}|  = \k^\l \leq \k^{<\cof(\k)}=\k\), where the last equality follows from \(2^{<\k}=\k\). Thus  \(\bigcap_{i<\l} A_i = \bigcup_{s \in \pre{\l}{\k} } (\bigcap_{i<\l}A_{i,s(i)} )\in \Siii{0}{\a}{\k^+}( X) \), as desired.

Now let \(\a\) be a limit ordinal, and let us prove that \(  \Siii{0}{\a}{\k^+}( X) \) is closed under intersections shorter than \( \cof(\a)\). 
Fix \(\l < \cof(\a)\) and a sequence \(\seq{A_i}{i<\lambda}\) of sets in \(\Siii{0}{\a}{\k^+}( X) \). For every \( i < \l \) there is a sequence \( \seq{A'_{i,\ell}}{\ell < \k} \) of sets in \( \bigcup_{1 \leq \b < \a} \Piii{0}{\b}{\g}(X) \) such that \( A_i = \bigcup_{\ell < \k} A'_{i ,\ell} \). Fix an increasing sequence \( \seq{\b_j}{j < \cof(\a)} \) of ordinals cofinal in \( \a \), and for each \( j < \cof(\a) \) let \( A_{i,j} = \bigcup \{ A'_{i,\ell} \mid \ell < \k, A'_{i,\ell} \in \Piii{0}{\b_j}{\g}(X) \}\). Then \( A_{i,j} \in \Siii{0}{\b_j+1}{\g}(X) \subseteq \Piii{0}{\b_j+2}{\g}(X) \), and \( A_i = \bigcup_{j < \cof(\a)} A_{i,j} \).
By the same computation we used before, we can write 
\[
\bigcap_{i<\l} A_i= \bigcap_{i<\l} \left(\bigcup_{j<\cof(\a)}A_{i,j} \right)= \bigcup_{s \in \pre{\l}{\cof(\a)} } \left(\bigcap_{i<\l}A_{i,s(i)} \right).
\]
Fix any \( s \in \pre{\l}{\cof(\a)}\), and let \( \bar \beta  = \sup \{ \beta_{s(i)} + 2 \mid i < \l \} \). Since \( \l < \cof(\a) \), we have \( \bar \beta < \a \), and \( \bigcap_{i<\l}A_{i,s(i)} \in \Piii{0}{\bar \beta}{\g}(X) \subseteq \bigcup_{1 \leq \b < \a } \Piii{0}{\b}{\g}(X) \) because \( \Piii{0}{\bar \beta}{\g}(X) \) is closed under intersections of length \( \k \). It is thus enough to show that \( \cof(\a) ^\l = | \pre{\l}{\cof(\a)} |  \leq \k \).
If $\kappa$ is regular, then $\lambda < \cof(\alpha) \leq \kappa$, hence
$\cof(\a)^\l \leq \kappa^{<\kappa} = \kappa$ and we are done.\footnote{Recall that \( \kappa^{< \kappa} = \kappa \) is equivalent to \( \kappa \) being regular and such that \( 2^{<\kappa} = \kappa \).} If $\kappa$ is singular, then $\l < \cof(\alpha) < \kappa$, thus \( \cof(\a)^\l \leq 2^{\cof(\a)}<\k\) because in the singular case \(2^{<\kappa}=\kappa\) is equivalent to $\kappa$ being strong limit, hence we are done again. 

Finally, if \(\a=1\) then \(\Siii{0}{1}{\k^+}( X)\) is closed under unions of any size, and it is closed under intersections shorter than \(\cof(\k)\) if and only if $X$ is $\cof(\k)$-additive. This concludes our proof.
\end{proof}

When \( \kappa \) is regular, it is not hard to see that many of the closure properties stated in Proposition~\ref{k+_hierarchy_closure} are optimal if \( \ord_{\k^+}(X) > \a \). 

\begin{proposition} \label{prop:optimalityregular} 
Suppose that \( \k \) is \emph{regular}.
 Given any \( 1 \leq \alpha < \ord_{\k^+}(X)\), let \( \widehat{\alpha} = \cof(\k)=\k \) if \( \a \) is a successor ordinal, and \( \widehat{\a} = \cof(\a) \) if \( \a \) is limit. Then:
\begin{enumerate-(1)}
\item \label{prop:optimalityregular-1}
\( \Siii{0}{\a}{\k^+}(X) \) is neither closed under complements, nor under intersections of size \( \widehat{\a} \);
\item \label{prop:optimalityregular-2}
\( \Piii{0}{\a}{\k^+}(X) \) is neither closed under complements, nor under unions of size \( \widehat{\a} \);
\item \label{prop:optimalityregular-3}
if \( \a > 1 \), then \( \Deee{0}{\a}{\k^+}(X) \) is not closed under unions or intersections of size \( \widehat{\a} \), and the same is true for \( \a = 1 \) if \( X \subseteq \pre{\k}{2}\).
\end{enumerate-(1)}
\end{proposition}

Notice that part~\ref{prop:optimalityregular-3} may fail if \( \a = 1\): if e.g.\ the space \( X \) is connected, then \( \Deee{0}{1}{\k^+}(X) = \{ \emptyset, X \}  \) is closed under arbitrary unions and intersections. 

\begin{proof}
 Suppose first that \( 1 \leq \a < \k^+ \) is such that \( \widehat{\a} = \k \).
If \( \Siii{0}{\a}{\k^+}(X) \) (equivalently: \( \Piii{0}{\a}{\k^+}(X) \)) were closed under complements, or if \( \a > 1 \) (so that \( \Siii{0}{1}{\k^+}(X) \cup \Piii{0}{1}{\k^+}(X) \subseteq \Siii{0}{\a}{\k^+}(X) \) because the \( \k^+ \)-Borel hierarchy on \( X \) is increasing) and either \( \Siii{0}{\a}{\k^+}(X) \) were closed under intersections of size \( \k \) (equivalently: \( \Piii{0}{\a}{\k^+}(X) \) were closed under unions of size \( \k \)), or \( \Deee{0}{\a}{\k^+}(X) \) were closed under either unions or intersections of size \( \k \), then \( \ord_{\k^+}(X) \leq \a \) by Proposition~\ref{prop:conditions_collapse}, against our assumptions. 
Next, consider the case \( \a = 1 \). If \( \Siii{0}{\a}{\k^+}(X) \) were closed under intersections of size \( \k \) (equivalently: \( \Piii{0}{\a}{\k^+}(X) \) were closed under unions of size \( \k \)), then each singleton \( \{ x \} \subseteq X \) would be open because \( \weight(X) \leq \k \); thus \( X \) would be discrete, against \( \a = 1 < \ord_{\k^+}(X) \).
Now further assume that \( X \subseteq \pre{\k}{2}\). If \( \Deee{0}{1}{\k^+}(X) \) were closed under intersections (equivalently: unions) of size \( \k \), then each singleton \( \{ x \} \subseteq X \) would be clopen because \( X \) has a clopen basis of size at most \( \k \); this again implies that \( X \) is discrete, contradicting \( \a = 1 < \ord_{\k^+}(X) \).
This concludes the proof when \( \a \) is such that \( \widehat{\a} = \k \).

Assume now that \( 1 \leq  \a  < \ord_{\k^+}(X) \) is limit and \( \widehat{\a} =  \cof(\a) < \k \).  
The fact that \( \Siii{0}{\a}{\k^+}(X) \) and \( \Piii{0}{\a}{\k^+}(X) \) are not closed under complements follows again from Proposition~\ref{prop:conditions_collapse}.
As for the other closure properties, notice that \( \Piii{0}{\a}{\k^+}(X) \) is closed under unions of size \( \widehat{\a} \) if and only if \( \Siii{0}{\a}{\k^+}(X) \) is closed under intersections of size \( \widehat{\a} \).  Also, either condition implies that \( \Deee{0}{\a}{\k^+}(X) \) is closed under unions of size \( \widehat{\a} \), since \( \Siii{0}{\a}{\k^+}(X) \) is always closed under unions of size \( \k \) by Proposition~\ref{k+_hierarchy_closure}\ref{k+_hierarchy_closure-1}. 
Thus, it is enough to show that \( \Deee{0}{\a}{\k^+}(X) \) is not closed under unions of size \( \widehat{\a} \).
Arguing as in the limit case of the proof of Proposition~\ref{k+_hierarchy_closure}, we get that every \( A \in \Siii{0}{\a}{\k^+}(X) \) can be written as \( A = \bigcup_{i < \widehat{\a}} A_i \) with \( A_i  \in \bigcup_{1 \leq \beta < \alpha} \Piii{0}{\b}{\k^+}(X) \subseteq \Deee{0}{\a}{\k^+}(X) \).
Thus we would have that if \( \Deee{0}{\a}{\k^+}(X) \) were closed under unions of size \( \widehat{\a} \), then \( \Siii{0}{\a}{\k^+}(X) = \Deee{0}{\a}{\k^+}(X) \), which again contradicts \( \a < \ord_{\k^+}(X) \) by Proposition~\ref{prop:conditions_collapse}.
\end{proof}

\begin{remark} \label{rmk:optimalityregular}
One may wonder if the closure under \( \kappa \)-sized unions of \( \Siii{0}{\a}{\k^+}(X) \) (equivalently: the closure under \( \k \)-sized intersections of \( \Piii{0}{\a}{\k^+}(X) \)) is optimal as well when \( \k \) is regular and \( 1 < \a <  \ord_{\k^+}(X) \). 
 This is the case if e.g.\ \( \a = \b+1 \) is a successor ordinal and \( \k \) is a regular cardinal such that \( 2^{\k} = \k^+ \). Indeed, every \( A \in \Piii{0}{\a+1}{\k^+}(X) \) can be written as \( A = \bigcap_{i < \k} \bigcup_{j < \k} A_{i,j} \) with \( A_{i,j} \in \Piii{0}{\b}{\k^+}(X) \), and hence
\( A = \bigcup_{s \in \pre{\k}{\k}} \bigcap_{i < \k} A_{i,s(i)} \) by the usual computation.
Since \( \Piii{0}{\b}{\k^+}(X) \) is closed under \( \k \)-sized intersections (Remark~\ref{rmk:easyclosure}), this means that every \( A \in \Piii{0}{\a+1}{\k^+}(X) \) is a union of \( 2^\k \)-many \( \Piii{0}{\b}{\k^+}(X) \) sets, hence if \( \Siii{0}{\a}{\k^+}(X) \) were closed under unions of size \( \k^+ = 2^\k \) we would have \( \Piii{0}{\a+1}{\k^+}(X) = \Siii{0}{\a}{\k^+}(X) \), contradicting \( \ord_{\k^+}(X) > \a \).
\end{remark}

Proposition~\ref{prop:optimalityregular} remains true also when \( \k \) is singular, but the argument is more delicate and is postponed to Corollary~\ref{cor_hier_clos}, where it will be derived from an analogous result concerning the \( \k \)-Borel hierarchy.

\begin{proposition} \label{prop:proper_for_k+}
 Suppose that \( 1 \leq \a \leq \ord_{\k^+}(X) \), and that \( \a \neq \ord_{\k^+}(X) \) if either \( \ord_{\k^+}(X) = \k^+ \) or \( \ord_{\k^+}(X) = 2 \). Then
\begin{equation}\label{eq:properness}
\bigcup_{1 \leq \b < \a} \left( \Siii{0}{\b}{\k^+}(X) \cup \Piii{0}{\b}{\k^+}(X) \right) \subsetneq \Deee{0}{\a}{\k^+}(X). 
\end{equation}
Moreover, \eqref{eq:properness} holds for \( \a = \ord_{\k^+}(X) = 2 \) as well if \( X \) does not have exactly one non-isolated point.
\end{proposition}

\begin{proof}  
If \( \a = 1 \), then~\eqref{eq:properness} reduces to \( \Deee{0}{1}{\k^+}(X) \neq \emptyset\), which is true because e.g.\ \( X \in \Deee{0}{1}{\k^+}(X) \).
If instead \( \a = 2 = \ord_{\k^+}(X) \), then \( X \) is not discrete by \( \ord_{\k^+}(X) > 1 \), so under our additional assumption on \( X \) for this specific case we conclude that \( X \) must have at least two non-isolated points, and~\eqref{eq:properness} is satisfied by Lemma~\ref{lem:door_spaces}.
Therefore from this point onward we can assume that \( \ord_{\k^+}(X) \geq 3 \) and \( \a \geq 2 \).
    Recall also that $\Siii{0}{\a}{\k^+}(X)$ is always closed under unions of size $\kappa$ (Remark~\ref{rmk:easyclosure}), and that the $\kappa^+$-Borel hierarchy is increasing on $X$ (Lemma~\ref{lem:increasing}).

Assume first that \( \a = \a'+1 \) is a successor ordinal: then it is enough to show that there is \( A \in \Deee{0}{\a}{\k^+}(X) \setminus \left( \Siii{0}{\a'}{\k^+}(X) \cup \Piii{0}{\a'}{\k^+}(X) \right) \). 
Let \( P \subseteq X \) be the closed set obtained by applying Proposition~\ref{prop:CB_for_hierarchy_collapse_successor_case} with \( \g = \k^+ = \g^*  \), \( \l = \k \), \( \b = 1 < \a' \), and \( \B \) a basis for \( X \) of size at most \( \k \). The space \( P \) is nonempty because \( \ord_{\k^+}(X) > \a' \), and it is not discrete because \( \ord_{\k^+}(P) = \ord_{\k^+}(X) > 1 \). 
Pick two distinct points \( x_0,x_1 \in P \), and let \( B_0,B_1 \in \B \) be disjoint open neighborhoods of \( x_0 \) and \( x_1 \), respectively, so that \( B_0 \cap P \) and \( B_1 \cap P \) are both nonempty. (Here we use that \( P \) is not a singleton because it is not discrete.) Since \( \ord_{\k^+}(B_i \cap P) > \a' \) for both \(i = 0\) and \( i = 1 \), by Proposition~\ref{prop:conditions_collapse} there are \( A_0 \in \Siii{0}{\a'}{\k^+}(B_0 \cap P) \setminus \Piii{0}{\a'}{\k^+}(B_0 \cap P) \) and \( A_1 \in \Piii{0}{\a'}{\k^+}(B_1 \cap P) \setminus \Siii{0}{\a'}{\k^+}(B_1 \cap P) \). Then \( A = A_0 \cup A_1 \) is as desired.
Indeed, since \( \Siii{0}{\a'}{\k^+} \) and \( \Piii{0}{\a'}{\k^+} \) are hereditary, there are \( A'_0 \in \Siii{0}{\a'}{\k^+}(X) \) and \( A'_1 \in \Piii{0}{\a'}{\k^+}(X) \) such that \( A_0 = A'_0 \cap (B_0 \cap P) \) and \( A_1 = A'_1 \cap (B_1 \cap P) \). Since \( \Deee{0}{\a}{\k^+}(X) \) is at least an \( \omega \)-algebra by 
Remark~\ref{rmk:closureunderfinite operations},
 we get \( A \in \Deee{0}{\a}{\k^+}(X) \).
On the other hand, if \( A \in \Siii{0}{\a'}{\k^+}(X) \), then \( A_1 = A \cap (B_1 \cap P) \in \Siii{0}{\a'}{\k^+}(B_1 \cap P) \), contradicting the choice of \( A_1 \); similarly, if \( A \in \Piii{0}{\a'}{\k^+}(X) \), then \( A_0 = A \cap (B_0 \cap P) \in \Piii{0}{\a'}{\k^+}(B_0 \cap P) \), contradicting the choice of \( A_0 \).

Suppose now that \( \a \) is limit. 
In this case, we can assume \( X \subseteq \pre{\k}{2}\), by Proposition~\ref{prop:embed_T_0_spaces_into_Cantor} (applied with \( \g = \k^+ \) and \( \l = \k \), so that also \( \mu = \k \)).
We let \( \B = \{ \clopen{s} \cap X \mid s \in \pre{<\k}{2} \} \) be the canonical basis for $X$, which consists of clopen sets and has size at most $\kappa$. 
Let also $\seq{\a_j}{j<\cof(\a)}$ be an increasing sequence of ordinals cofinal in $\a$ such that \( \a_0 > 2 \). 

\begin{claim} \label{claim:newproperclaim}
There is a family \( \seq{C_j}{j < \cof(\a)} \) of nonempty pairwise disjoint closed subsets of \( X \) such that \( \ord_{\k^+}(C_j) > \a_j \) for each \( j < \cof(\a) \).
\end{claim}

\begin{proof}[Proof of the Claim]
We distinguish two cases. 
Suppose first there is an ordinal $1 \leq \a'<\a$ such that \(  \ord_{\k^+}(B)< \a \) implies \(  \ord_{\k^+}(B)\leq \a' \) for every $B\in \B$. 
Let \( P \subseteq X \) be the closed set obtained by applying Proposition~\ref{prop:CB_for_hierarchy_collapse_successor_case} with \( \g = \k^+ = \g^*  \), \( \l = \k \), \( \b = 1 < \a' \), and \( \B \)  the canonical basis for \( X \). Since \( \ord_{\k^+}(X) > \a' \), we have \( P \neq \emptyset \) and \( \ord_{\k^+}(P) = \ord_{\k^+}(X)\geq \a\). 
In particular, $\ord_{\kappa^+}(P) \geq 3$, hence $|P|>\kappa$ and \( P \) is not discrete.
Notice also that for every \( s \in \pre{<\k}{2}\), if \( \clopen{s} \cap P \neq \emptyset \) then \( \ord_{\k^+}(\clopen{s} \cap P) = \ord_{\k^+}(\clopen{s} \cap X) > \a' \), and hence \( \ord_{\k^+}(\clopen{s} \cap P)  \geq \a > \a_j \) (for every \( j < \cof(\a) \)) by case assumption and choice of \( \a' \). Thus it is enough to set \( C_j = \clopen{s_j} \cap P \), for a suitable choice of \( s_j \in \pre{<\k}{2}\). For this, we distinguish two subcases.

If $\kappa$ is regular, then  we can pick a limit point \( x \in P \) and a sequence \( \seq{x_j}{j < \k} \) of points from \( P \setminus \{ x \} \) converging to \( x \) because \( P \) is not discrete. For every \( j < \k \), let \( s_j = x_j \restriction \b_j \) for \( \b_j < \k  \) smallest such that \( x \restriction \b_j \neq x_j \restriction \b_j \), and notice that \( x_j \) witnesses \( \clopen{s_j} \cap P \neq \emptyset \). By regularity of \( \k \), without loss of generality we may assume that \( \b_i \neq \b_j \) for distinct \( i,j < \k \), so that \( \clopen{s_i} \cap \clopen{s_j} = \emptyset\). Then the sequence \( \seq{C_j}{j < \cof(\a)} \) with \( C_j = \clopen{s_j} \cap P  \) is as desired. 
If instead $\kappa$ is singular, then it is strong limit by \( 2^{< \k} = \k \), and moreover $\cof(\a)<\k$ because $\a<\k^+$. Let $\seq{\k_i}{i<\cof(\k)}$ be an increasing  sequence of ordinals cofinal in $\k$, and for each \( i < \cof(\k) \) let \( \B_i = \{ \clopen{s} \cap P \mid s \in \pre{\k_i}{2} \wedge \clopen{s} \cap P \neq \emptyset \}\). 
The elements of a given \( \B_i \) are clearly pairwise disjoint.
If \( |\B_i| < \cof(\a) \) for every \( i < \cof(\k) \), then \( |P| \leq \cof(\a)^{\cof(\k)} < \k \), a contradiction. Therefore it is enough to pick \( i < \cof(\k) \) such that \( |\B_i| \geq \cof(\a) \), and let \( \seq{C_j}{j < \cof(\a)} \) be an enumeration without repetitions of a large enough subfamily of \( \B_i \).

The remaining case is when for every ordinal $\a'<\a$ there is $B\in \B$ such that \(\a'<\ord_{\k^+}(B) < \a \).
For \( j < \cof(\a) \), recursively pick \( B_j \in \B \) such that \( \a_j < \ord_{\k^+}(B_j) < \a \) and \( \ord_{\k^+}(B_j) > \sup_{i < j} \ord_{\k^+}(B_i) \). Since \( \ord_{\k^+}\left( \bigcup_{i < j} B_i \right) = \sup_{i<j} \ord_{\k^+}(B_i) < \ord_{\k^+}(B_j) \), we have that \( C_j = B_j \setminus \bigcup_{i < j} B_i \neq \emptyset \).
Moreover, arguing as in the proof of Proposition~\ref{prop:CB_for_hierarchy_collapse_successor_case} one can check that since both \( \bigcup_{i < j} B_i \) and \(  C_j \) belong to \( \Deee{0}{2}{\k^+}(X) \), then \( \ord_{\k^+}(B_j) \leq \max \left\{ \ord_{\k^+}\left( \bigcup_{i < j} B_i \right), \ord_{\k^+}(C_j), 2 \right\} \).
We conclude that \( \ord_{\k^+}(C_j) = \ord_{\k^+}(B_j) > \a_j \) by \( \ord_{\k^+}(B_j) > \ord_{\k^+}\left( \bigcup_{i < j} B_i \right) \) and  \( \ord_{\k^+}(B_j) > \a_j \geq \a_0 > 2 \). 
 \end{proof}

For each \( j < \cof(\a) \), pick a set \( A_j \in \Siii{0}{\a_j}{\k^+}(C_j) \setminus \Piii{0}{\a_j}{\k^+}(C_j)\): we claim that \( A = \bigcup_{j < \cof(\a)} A_j  \) witnesses~\eqref{eq:properness}.
 To see this, for each \( j < \cof(\a) \) use the fact that \( \Siii{0}{\a_j}{\k^+} \) is hereditary to find \( A'_j \in \Siii{0}{\a_j}{\k^+}(X) \) such that \( A_j = A'_j \cap (\clopen{s_j} \cap P) \). Then
\begin{align*}
A & = \bigcup\nolimits_{j < \cof(\a)} (A'_j \cap C_j), \\
X \setminus A & = \left(X \setminus \bigcup\nolimits_{j < \cof(\a)} C_j \right) \cup \bigcup\nolimits_{j < \cof(\a)} \left(C_j \setminus A'_j\right),
\end{align*}
hence both \( A \) and \( X \setminus A \) belong to \( \Siii{0}{\a}{\k^+}(X) \) by Proposition~\ref{k+_hierarchy_closure} and \( 2 < \alpha < \k^+ \) (which also entails \( \cof(\a) \leq \k \)), so that \( A \in \Deee{0}{\a}{\k^+}(X) \).
Suppose towards a contradiction that \( A \in \Siii{0}{\b}{\k^+}(X) \cup \Piii{0}{\b}{\k^+}(X) \) for some \( 1 \leq \b < \a \). Let \( j < \cof(\a) \) be such that \( \a_j > \b \), so that \( A \in \Piii{0}{\a_j}{\k^+}(X) \).
Then \( A_j = A \cap C_j \in \Piii{0}{\a_j}{\k^+}(C_j) \), contradicting the choice of \( A_j \).
\end{proof}

A fundamental question to be addressed is whether the \( \k^+\)-Borel hierarchy always needs to be \markdef{proper}, that is, if the inclusions in~\eqref{eq:inclusionswhenincreasing}, when setting \( \g = \k^+ \), are strict for all relevant \(\alpha, \beta<\ord_{\k^+}(X) \).  We are now going to show that this is always the case.

\begin{corollary} \label{cor:proper_for_k+}
For every \( 1 \leq \a < \ord_{\k^+}(X) \), we have that
\begin{enumerate-(1)}
\item\label{cor:proper_for_k+-1} \( \Deee{0}{\a}{\k^+}(X) \subsetneq \Siii{0}{\a}{\k^+}(X) \) (equivalently, \( \Deee{0}{\a}{\k^+}(X) \subsetneq  \Piii{0}{\a}{\k^+}(X)\)), and
\item\label{cor:proper_for_k+-2} \( \bigcup_{1 \leq \b < \a} \left( \Siii{0}{\b}{\k^+}(X) \cup \Piii{0}{\b}{\k^+}(X) \right) \subsetneq \Deee{0}{\a}{\k^+}(X)\).
\end{enumerate-(1)}
\end{corollary}

\begin{proof}
Part~\ref{cor:proper_for_k+-1} follows directly from Proposition~\ref{prop:conditions_collapse}, so we only need to prove part~\ref{cor:proper_for_k+-2}.

If $\ord_{\k^+}(X)=1$ there is nothing to prove. 
If \( \ord_{\k^+}(X) > \a = 1 \), then the result follows from the fact that, for example, \( X \in \Deee{0}{1}{\k^+}(X) \). Finally, suppose that \( 1 < \a < \ord_{\k^+}(X) \).
Then \( \ord_{\k^+}(X) \geq 3 \), and the result then follows from Proposition~\ref{prop:proper_for_k+}.
   \end{proof}

Another natural question related to the properness of the $\kappa^+$-Borel hierarchy is the following.
Suppose that \( 1 \leq  \ord_{\k^+}(X) < \k^+ \). (Examples of spaces with this property will be provided in Section~\ref{sec: alpha-forcing} for finite values of \( \ord_{\k^+}(X) \), and in a follow-up to this paper for infinite values.) Is there any set \( A \in \Deee{0}{\ord_{\k^+}(X)}{\k^+}(X) = \Bor{\k^+}(X) \) that does not belong to lower levels of the \( \k^+ \)-Borel hierarchy on \( X \)?
Lemma~\ref{lem:door_spaces} and Proposition~\ref{prop:proper_for_k+} allow us to answer to this question as well: the answer is affirmative if and only if the space does not have exactly one non-isolated point.

\begin{corollary}\label{cor:proper_for_k+_final_level} 
Suppose that \( \ord_{\k^+}(X) < \k^+ \). Then
the following are equivalent:
\begin{enumerate-(1)}
\item\label{cor:proper_for_k+_final_level-1}  there is $A\in \Deee{0}{\ord_{\kappa^+}(X)}{\k^+}(X)\setminus \bigcup_{1 \leq \b < \ord_{\kappa^+}(X)} \left( \Siii{0}{\b}{\k^+}(X) \cup \Piii{0}{\b}{\k^+}(X) \right)$;
\item\label{cor:proper_for_k+_final_level-2} $X$ does not have exactly one non-isolated point.
\end{enumerate-(1)}
In particular, \ref{cor:proper_for_k+_final_level-1} holds whenever $\ord_{\kappa^+}(X)\neq 2$.
\end{corollary}

 When \( \a = 2 \), condition~\ref{cor:proper_for_k+_final_level-1} from Corollary~\ref{cor:proper_for_k+_final_level} may fail even for closed subsets \( X \) of \( \pre{\kappa}{2}\). 
For example, let \( X = \{ x \} \cup \{ x_i \mid i < \k \} \), where for every \( i < \k \) we let \( x_i(j) = x(j) \) if \( j\neq i  \) and \( x_i(i) = 1 - x(i) \). Since \( X \) consists of a \( \k \)-sequence of isolated points converging to \( x \), it is easy to see that it is as required.

We next move to the question of when the \( \k^+ \)-Borel hierarchy on \( X \) does not collapse. 
In classical descriptive set theory, the non-collapse of the Borel hierarchy of an uncountable Polish space $X$ follows from the Perfect Set Property and the fact that \( (\omega_1 \text{-}) \boldsymbol{\Sigma}^0_\a(X)\) has a $\pre{\omega}{2}$-universal set for each $1 \leq \alpha< \omega_1$ (see \cite[Section 22A]{KEC95}).
	 	 In the generalized context, the non-collapse of the \( \k^+ \)-Borel hierarchy on \( \pre{\k}{2} \) has been first obtained in~\cite[Proposition 4.19]{AMR19} for an \emph{arbitrary} infinite cardinal \( \k \), i.e.\ without assuming \( 2^{< \k} = \k \). This required different methods, because if the condition \(2^{<\k}=\k\) fails, then it might happen that neither \(\Siii{0}{\a}{\k^+}(\pre{\k}{2})\) nor  \(\Piii{0}{\a}{\k^+}(\pre{\k}{2})\) have a \( \pre{\k}{2} \)-universal set, for any \( 1 \leq \a < \k^+ \) (\cite[Corollary 4.16]{AMR19}).
However, with the assumption \( 2^{<\k} = \k \) the proof works exactly like in the classical argument.

\begin{proposition}\label{universal_sigma_pi}
 For each \(1 \leq \a <\k^+\) there are \(\pre{\k}{2}\)-universal sets for both \(\Siii{0}{\a}{\k^+}(X)\) and \(\Piii{0}{\a}{\k^+}(X)\). 
Therefore there are subsets of \( \pre{\k}{2}\) that are \( \k \)-complete for \( \Siii{0}{\a}{\k^+} \) and \( \Piii{0}{\a}{\k^+} \), respectively.
In contrast, there is no \( \pre{\k}{2} \)-universal set for \( \Deee{0}{\a}{\k^+}(\pre{\k}{2}) \) or \( \Bor{\k^+}(\pre{\k}{2}) \). 

Moreover, \( \pre{\k}{2}\) can systematically be replaced by \( X \) in all the above statements if \( \pre{\k}{2}\) embeds into \( X \).
\end{proposition}

\begin{proof}
For the first part, apply Corollary~\ref{cor:universalinborelhierarchy} with \( \l = \k \) and \( \g = \k^+\) to get the desired \( \pre{\k}{2}\)-universal sets, and when \( X = \pre{\k}{2}\) apply Lemma~\ref{univ_compl}  to obtain the existence of the complete sets, or Lemma~\ref{universal_selfdual} to obtain the non-existence of \( \pre{\k}{2} \)-universal sets for the selfdual classes (relatively to the space \( \pre{\k}{2}\) itself). 

Assume now that there is an embedding \( f \colon \pre{\k}{2} \to X \). Then Lemma~\ref{lemma_universal} and the first part ensure that there are \( X \)-universal sets for both \(\Siii{0}{\a}{\k^+}(X)\) and \(\Piii{0}{\a}{\k^+}(X)\). 
As for complete sets, let \( A' \subseteq \pre{\k}{2}\) be \( \k \)-complete for \( \Siii{0}{\a}{\k^+} \), and let \( A \in \Siii{0}{\a}{\k^+}(X) \) be such that \( A \cap f[\pre{\k}{2}] = f[A'] \): then \( A \) is \( \k \)-complete for \( \Siii{0}{\a}{\k^+} \) as well.
The case of \( \Piii{0}{\a}{\k^+}(X) \) is similar. The non-existence of \( X \)-universal sets for \( \Deee{0}{\a}{\k^+}(X) \) and \( \Bor{\k^+}(X) \) follows from Lemma~\ref{universal_selfdual} again.
\end{proof}

\begin{cor}\label{collapse_k+} 
The $\kappa^+$-Borel hierarchy on $\pre{\kappa}{2}$ does not collapse. 
\end{cor}

By Proposition~\ref{prop:abstractpropertiesofcollapse}\ref{prop:abstractpropertiesofcollapse-1}, this implies that the \(\k^+\)-Borel hierarchy does not collapse on any space that contains a homeomorphic copy of \( \pre{\k}{2}\). Actually, the same conclusion holds in greater generality.

\begin{theorem}\label{thm:non-collpase_general_space}
 If there is a $\kappa^+$-Borel embedding $f \colon\pre{\kappa}{2}\to X$, then the $\kappa^+$-Borel hierarchy on $X$ does not collapse.
\end{theorem}

\begin{proof}
Assume towards a contradiction that \( \ord_{\k^+}(X) < \k^+ \), and let \( Y = f[\pre{\kappa}{2}] \).
Then \( \ord_{\k^+}(Y) < \k^+\) by Proposition~\ref{prop:abstractpropertiesofcollapse}\ref{prop:abstractpropertiesofcollapse-1}, hence \( \ord_{\k^+}(\pre{\k}{2}) < \k^+ \) by Corollary~\ref{cor:abstractpropertiesofcollapse}, against Corollary~\ref{collapse_k+}.
\end{proof}

Theorem~\ref{thm:non-collpase_general_space} provides an important sufficient condition for the non-collapse of the $\k^+$-Borel hierarchy on a space \( X \), but in Corollary~\ref{cor: closed set with full order} we will show that, consistently, there may be very simple spaces (in fact, even closed subspaces of \( \pre{\k}{\k}\), for \( \k \) regular)  such that \( \pre{\k}{2} \) does not $\kappa^+$-Borel embed into them, yet they have a non-collapsing $\kappa^+$-Borel hierarchy.

As a consequence of Theorem~\ref{thm:non-collpase_general_space}, we get that if e.g.\ all \( \k^+ \)-Borel subsets of \( \pre{\cof(\k)}{\k}\) (or equivalently: of \( \pre{\k}{2}\)) satisfy some form of the \( \k \)-Perfect Set Property, then the \( \k^+ \)-Borel hierarchy does not collapse on every topological space of weight at most \( \kappa \) that is $\kappa^+$-Borel isomorphic to a $\kappa^+$-Borel subset of $\pre{\cof(\k)}{\kappa}$ of size greater than \( \k \). (The form of \( \k \)-Perfect Set Property we are alluding to is the following: Either \( |A| \leq \k \), or there is a \( \k^+ \)-Borel embedding of \( \pre{\k}{2}\) into \( A \). 
Note that, when \(\kappa\) is regular, this is equivalent to the $\kappa$-Perfect Set Property defined in Section~\ref{sec: psp section} -- see, e.g., \cite[Proposition~2]{FriedmanSilverIsConsistentMR3257956}. 
See~\cite{AgosMottoSchlichtRegular} and~\cite{DMR} for other variants and their relationships when \( \k \) is regular or \( \cof(\k) = \o \), respectively.) This scenario is consistent if \( \k \) is regular (\cite{schlicht_perfect_2017}), and is necessarily true if \( \k \) has countable cofinality (\cite{DMR}) --- the case of a singular cardinal \( \k \) of uncountable cofinality, instead, has not yet been studied.
However, it is also known that if \( \k \) is regular, then it is consistent to have 
 $\k^+$-Borel (or even closed) sets \( X \) that do not satisfy any form of $\kappa$-Perfect Set Property (see Section~\ref{sec: psp section}): for such sets, the non-collapse of the $\kappa^+$-Borel hierarchy cannot be derived merely from the existence of $\pre{\kappa}{2}$-universal sets, and indeed we will show that it is consistent to have \( 2 < \ord_{\k^+}(X) < \k^+ \) (see Section~\ref{sec: alpha-forcing}).

We conclude this section with a brief digression on \( \k^+ \)-Borel measurable functions.
Recall that a function \(f \colon X \to Y \) is \(\boldsymbol{\Gamma} \)-measurable if \( f^{-1}(U) \in \boldsymbol{\Gamma}(X) \) for every open \( U \subseteq Y \). 
In particular, \(f \) is called \(\k^+ \)-Borel measurable if it is \( \Bor{\k^+} \)-measurable or, equivalently, if \( f^{-1}(B) \in \Bor{\k^+}(X) \) for every \( B \in \Bor{\k^+}(Y) \).
We denote by \( \mathcal{B}(X,Y) \) the collection of all \( \k^+ \)-Borel measurable functions from \( X \) to \( Y \), and for each ordinal \( \a \geq 1 \) we let \( \mathcal{M}_\a(X,Y) \) be the collection of all \( \Siii{0}{\a}{\k^+} \)-measurable functions from \( X \) to \( Y \). 
Clearly, \( \mathcal{B}(X,Y) = \mathcal{M}_{\ord_{\k^+}(X)}(X,Y)\). In particular, we always have \( \mathcal{B}(X,Y) = \mathcal{M}_{\k^+}(X,Y) \), and if \( \weight(Y) \leq \k \) then \( \mathcal{B}(X,Y) = \bigcup_{1 \leq \a < \k^+} \mathcal{M}_\a(X,Y) \). This means that the classes \( \mathcal{M}_\a(X,Y) \) stratify the \( \k^+ \)-Borel measurable functions in a hierarchy of length at most \( \k^+ \). In analogy with Definition~\ref{def:order}, we introduce the following parameter, that measures the length of such hierarchy.

\begin{definition}\label{def:order_functions}
Let $X$ and $Y$ be topological spaces. The \markdef{order} of the hierarchy of $\k^+$-Borel measurable functions from $X$ to $Y$ is
\[
\ord^{\mathrm{Fun}}_{\k^+}(X,Y) = \min\left\{\alpha \in \Ord \mid \mathcal{M}_\a(X,Y) =\mathcal{B}(X,Y)\right\}.
\]
\end{definition} 

By the previous discussion, \( \ord^{\mathrm{Fun}}_{\k^+}(X,Y)  \leq \k^+ \); when \( \ord^{\mathrm{Fun}}_{\k^+}(X,Y)  < \k^+ \) we say that the hierarchy of $\k^+$-Borel measurable functions from $X$ to $Y$ \markdef{collapses}. The next proposition shows that if \( Y \) contains at least two points, then this happens precisely when the \( \k^+ \)-Borel hierarchy \emph{of sets} on \( X \) collapses.

\begin{proposition} \label{prop:stratificationofBorelfunctions}
Let \( Y \) be a Hausdorff topological space such that 
 \(|Y|\geq 2\). 
Then
 \[
\ord_{\k^+}(X) = \ord^{\mathrm{Fun}}_{\k^+}(X,Y),
\]
 
Moreover, the hierarchy of \( \k^+ \)-Borel measurable functions is proper, meaning that 
\( \mathcal{M}_\a \setminus \bigcup_{1\leq \beta < \a } \mathcal{M}_\b(X,Y) \neq \emptyset \)
for every \( 1 \leq \a < \k^+ \) with \( \a \leq \ord_{\k^+}^{\mathrm{Fun}}(X,Y) \).
\end{proposition}

\begin{proof}
The inequality \( \ord_{\k^+}^{\mathrm{Fun}}(X,Y) \leq \ord_{\k^+}(X) \) is obvious, so let us prove the reverse inequality.
Fix distinct points \( y_0,y_1 \in Y \), and take any \( A \in \Bor{\k^+}(X) \). Let \( f_A \colon X \to Y \) be defined by \( f_A(x) = y_0 \) if \( x \in A \), and \( f_A(x) = y_1 \) otherwise. Then \( f_A \in \mathcal{B}(X,Y) = \mathcal{M}_{\ord_{\k^+}^{\mathrm{Fun}}(X,Y)}(X,Y) \); since \( A = f_A^{-1}(U) \) for any open \( U \subseteq Y\) such that \( y_0 \in U \) and \( y_1 \notin U \), we then get \( A \in \Siii{0}{\ord_{\k^+}^{\mathrm{Fun}}(X,Y)}{\k^+}(X) \). The previous argument shows that \( \Bor{\k^+}(X) \subseteq \Siii{0}{\ord_{\k^+}^{\mathrm{Fun}}(X,Y)}{\k^+}(X) \), and thus \( \ord_{\k^+}(X) \leq \ord_{\k^+}^{\mathrm{Fun}}(X,Y) \), as desired.

We now move to properness. To simplify the notation, let \( \overline{\a} = \ord_{\k^+}(X) = \ord_{\k^+}^{\mathrm{Fun}}(X,Y)\). Given \( 1 \leq \a < \overline{\a} \), let \( A \in \Deee{0}{\a}{\k^+}(X) \setminus \bigcup_{1 \leq \b < \a} \Siii{0}{\b}{\k^+}(X) \); the existence of such an \( A \) is granted by Proposition~\ref{prop:proper_for_k+}. Then \( f_A \in \mathcal{M}_{\a}(X,Y) \), but \( f_A \notin \mathcal{M}_\b(X,Y) \) for any \( 1 \leq \b < \a \) because otherwise \( A = f^{-1}_A(U) \) would belong to \( \Siii{0}{\b}{\k^+}(X) \). The same argument works if \( \a = \overline{\a} < \k^+ \) too, except that when \(\overline{\a} = 2 \), instead of using Proposition~\ref{prop:proper_for_k+} to find the desired \( A \) use the fact that \( \Deee{0}{2}{\k^+}(X) \setminus \Siii{0}{1}{\k^+}(X) \) is nonempty because \( X \) must contain a proper closed set (otherwise it would be discrete, against \( \ord_{\k^+}(X) = 2 \)).
\end{proof}

An extensive study of the structure of the \(\k^+\)-Borel measurable functions  with respect to various kinds of limits, and the suitable notion of generalized Baire class \(\a\) function can be found in~\cite{MRP25}.
Here we briefly touch on this subject at the end of Section~\ref{sec: alpha-forcing}.
	
\section{\texorpdfstring{An alternative hierarchy when \(\k\) is singular}{An alternative hierarchy when kappa is singular}}\label{sec_alt}

Recall that, starting from Section~\ref{sec_borel_hierarchy}, we agreed that all topological spaces are assumed to be regular Hausdorff and of weight at most \( \k \). Fix such a space \( X \).
When \(\k\) is singular, we have that \(\Bor{\k^+}(X)=\Bor{\k}(X)\). Indeed, \( \Bor{\k}(X) \subseteq  \Bor{\k^+}(X)\) by Remark~\ref{rmk:refining}. For the other inclusion, it is enough to check that \(\Bor{\k}(X)\) is closed under unions of length $\kappa$. Let \(A = \bigcup_{i<\kappa} A_{i}\), where  \(A_{i} \in \Bor{\k}(X) \) for every \(i< \kappa\), and fix any sequence of ordinals \( \seq {\kappa_j }{j<\cof(\kappa) } \) cofinal in \(\kappa\). Then for all \(j<\cof(\k)\), we have \(A'_j = \bigcup_{i < \k_j}  A_{i} \in \Bor{\k}(X) \) because \(\k_j <\k\), hence \(A=\bigcup_{j<\cof(\k)}A'_j \in \Bor{\k}(X)\) because \(\cof(\k)<\k\).

Thus, when $\k$ is singular, the \(\kappa^+\)-Borel sets can naturally be stratified into an alternative hierarchy: the \(\k\)-Borel hierarchy. The study of such hierarchy, together with its interplay with the \(\k^+\)-Borel hierarchy studied in Section~\ref{sec_borel_hierarchy}, is the subject of this section. To this aim, we will often assume \(\Siii{0}{1}{\k}(X) \subseteq \Siii{0}{2}{\cof(\k)^+}(X)\). 
On the one hand, this is due to the fact that in order to guarantee that the \( \k \)-Borel hierarchy is increasing, we need to have at least \( \Siii{0}{1}{\k}(X) \subseteq \Siii{0}{2}{\k}(X) \), and this is no longer granted by the other tacit hypotheses on \( X \) (as it was the case for the \( \k^+ \)-Borel hierarchy). On the other hand, the fact that every open set can be written as a \( \cof(\k) \)-sized union of closed sets allows for a neater presentation of the combinatorial properties of the \( \k \)-Borel hierarchy. By Fact~\ref{lem:NS_implies_increasing_hierarchy}, such condition is anyway satisfied by virtually all spaces of interest in generalized descriptive set theory, including all subspaces of \( \pre{\k}{2} \) and \( \pre{\cof(\k)}{\k} \), so it does not really affect the intended scope of this paper. 
However, we also remark that by Proposition~\ref{prop:embed_T_0_spaces_into_Cantor}, all our results are valid for arbitrary \( T_0 \) spaces of weight at most \( \k \) if we only consider infinite levels \( \a \geq \o \).

We begin with a result describing the relationship between the \( \k \)-Borel hierarchy and the \( \k^+ \)-Borel hierarchy. From this, we will derive various other properties of the \( \k \)-Borel hierarchy and, complementing some results obtained in Section~\ref{sec_borel_hierarchy}, on the \( \k^+ \)-Borel hierarchy when \( \k \) is singular.

Recall that every ordinal \(\a\) can be written uniquely as \(\a=\g+n\), with \(n <\o \) and either \( \g=0 \) or \( \g \) limit.
Accordingly, we say that \(\a\) is even (respectively, odd) if \(n\) is even (respectively, odd). If \(\a = \g+n\) is even, to simplify the notation we set \(\frac{\a}{2}=\g+ \frac{n}{2} \).
Notice that \( \frac{2 \cdot \a}{2} = \a \) for every ordinal \( \a \), and that if \( \a = 0 \) or \( \a \) is limit, then \( \frac{\a}{2} = \frac{2 \cdot \a}{2} =  \a \).

\begin{theorem}\label{hierarchy_theorem}
Assume that $\k$ is \emph{singular}, and \( \Siii{0}{1}{\k}(X) \subseteq \Siii{0}{2}{\cof(\k)^+}(X) \). 
 Given \(\alpha < \kappa^+\), let \( \widehat{\a} = \cof(\a) \) if \( \a \) is limit, and \( \widehat{\a} = \cof(\kappa) \) otherwise.
\begin{enumerate-(1)} \item \label{hierarchy_theorem_even} 
If $\alpha\neq 0$ is even, then  
\[
\Siii{0}{1+\alpha}{\k}(X)= \left( \bigcup\nolimits_{\beta < \alpha}\Piii{0}{1+\beta}{\k}(X)\right) _{\sigma_{\widehat{\a}}} = \Siii{0}{1+\frac{\alpha}{2}}{\k^+}(X) ,
\]
and dually for $\Piii{0}{1+\alpha}{\k}(X)$.

\item\label{hierarchy_theorem_odd} 
If $\alpha$ is odd, then $\Siii{0}{1+\alpha}{\k}(X) = \Piii{0}{1+\alpha}{\kappa}(X)=\Deee{0}{1+\alpha}{\k}(X)$.
\end{enumerate-(1)}
\end{theorem}

\begin{proof}
By induction on  \(\alpha< \kappa^+\). 
 
We first consider~\ref{hierarchy_theorem_even}, so assume that \( \a \) is even and write it as \( \a = \g + 2n \) for some \( n < \o \) and \( \g = 0 \) or \( \g \) limit.
If \( n = 0 \), then \( \a = \g > 0 \) is limit, hence \( \frac{\a}{2} = \a \). 
Recalling also Remark~\ref{rmk:refining}, the inclusions \( \left( \bigcup\nolimits_{\beta < \alpha}\Piii{0}{1+\beta}{\k}(X)\right) _{\sigma_{\widehat{\a}}} \subseteq \Siii{0}{1+\a}{\k}(X) \subseteq \Siii{0}{1+\a}{\k^+}(X) \) are obvious, because \( \widehat{\a} = \cof(\a) < \k \).
The inductive hypothesis and 
the fact that, being \( \a \) limit, \( \b < \a \Rightarrow 2 \cdot \b < \a\), ensure that \( \Siii{0}{1+\b}{\k^+}(X) = \Siii{0}{1+2 \cdot \b}{\k}(X) \), and hence \( \bigcup_{\b < \a} \Siii{0}{1+\b}{\k}(X) = \bigcup_{\b < \a} \Siii{0}{1+\b}{\k^+}(X) \).
Arguing as in the limit case of the proof of Proposition~\ref{k+_hierarchy_closure}, every \( A \in \Siii{0}{1+\a}{\k^+}(X) \) can be written as \( A = \bigcup_{j < \cof(\a)} A_j \) with \( A_j \in \bigcup_{\b < \a} \Siii{0}{1+\b}{\k^+}(X) \) for every \( j < \cof(\a) \). Since \( \widehat{\a} = \cof(\a) \), we get \( A \in \left( \bigcup\nolimits_{\beta < \alpha}\Piii{0}{1+\beta}{\k}(X)\right) _{\sigma_{\widehat{\a}}} \). This shows that \( \Siii{0}{1+\a}{\k^+}(X) \subseteq \left( \bigcup\nolimits_{\beta < \alpha}\Piii{0}{1+\beta}{\k}(X)\right) _{\sigma_{\widehat{\a}}} \), as desired.

Assume now that \( n = m+1 \), so that \( \a=\g+2m+2\) and \( \frac{\a}{2} = \g+m+1 \). The inclusion \( \left( \bigcup\nolimits_{\beta < \alpha}\Piii{0}{1+\beta}{\k}(X)\right) _{\sigma_{\widehat{\a}}} \subseteq \Siii{0}{1+\a}{\k}(X) \) is again obvious because \( \widehat{\a} = \cof(\k) < \k \). Since by inductive hypothesis \( \Piii{0}{1+\g+2m}{\k}(X) = \Piii{0}{1+\g+m}{\k^+}(X) \), we get \( \Siii{0}{1+\g+2m+1}{\k}(X) \subseteq \Siii{0}{1+\g+m+1}{\k^+}(X) \). As \( \g+2m+1 \) is odd, \( \Piii{0}{1+\g+2m+1}{\k}(X) = \Siii{0}{1+\g+2m+1}{\k}(X)\) by inductive hypothesis again. Therefore \( \Piii{0}{1+\g+2m+1}{\k}(X) \subseteq \Siii{0}{1+\g+m+1}{\k^+}(X) \), and since the latter is closed under unions of size $\kappa$ by Theorem~\ref{k+_hierarchy_closure}\ref{k+_hierarchy_closure-1}, 
it follows that 
\[
\Siii{0}{1+\a}{\k}(X) = \Siii{0}{1+\g+2m+2}{\k}(X) \subseteq \Siii{0}{1+\g+m+1}{\k^+}(X) = \Siii{0}{1+\frac{\a}{2}}{\k^+}(X).
\]
Let now \( A \in \Siii{0}{1+\frac{\a}{2}}{\k^+}(X) = \Siii{0}{1+\g+m+1}{\k^+}(X) \), and let \( A_i \in \Piii{0}{1+\g+m}{\k^+}(X)\) be such that \( A = \bigcup_{i < \k} A_i \). Fix a sequence \( \seq{\kappa_j}{j < \cof(\kappa)} \) cofinal in \(\k\), and let \(B_j=\bigcup_{i<\k_j}A_i\) for every \( j < \cof(\k) \).
By inductive hypothesis, \( \Piii{0}{1+\g+m}{\k^+}(X) = \Piii{0}{1+\g+2m}{\k}(X) \), hence \( B_j \in \Siii{0}{1+\g+2m+1}{\k}(X) = \Piii{0}{1+\g+2m+1}{\k}(X) \) because \( \g+2m+1 \) is odd. Since \( A = \bigcup_{j < \cof(\k)} B_j \) and \( \widehat{\a} = \cof(\k) \), we conclude that \( A \in \left( \bigcup\nolimits_{\beta < \alpha}\Piii{0}{1+\beta}{\k}(X)\right) _{\sigma_{\widehat{\a}}} \). This shows that \( \Siii{0}{1+\frac{\a}{2}}{\k^+}(X) \subseteq \left( \bigcup\nolimits_{\beta < \alpha}\Piii{0}{1+\beta}{\k}(X)\right) _{\sigma_{\widehat{\a}}} \), as desired.

We now move to~\ref{hierarchy_theorem_odd}, so assume that \( \a \) is odd, and let \( \d \) be its predecessor, which is an even ordinal. It is enough to prove that $\Piii{0}{1+\a}{\k}(X)\subseteq \Siii{0}{1+\a}{\k}(X)$.
 Let \(A \in \Piii{0}{1+\d+1}{\k}(X)\), and let \(\l<\k\) and \(A_i \in \Siii{0}{1+\d}{\k}(X) \) be such that \(A=\bigcap_{i<\l}A_i\). 
We claim that for every \( i < \l \) there are \( \Piii{0}{1+\d}{\k}(X)\) sets \( A_{i,j} \), for \( j < \widehat{\d} \), such that \( A_i = \bigcup_{j < \widehat{\d}} A_{i,j} \). If \( \d = 0 \) this is true by the assumption \( \Siii{0}{1}{\k}(X) \subseteq \Siii{0}{2}{\cof(\k)^+}(X) \). If \( \d > 0 \), then this follows by the inductive hypothesis together with \( \bigcup_{\b < \d} \Piii{0}{1+\b}{\k}(X)  \subseteq \Piii{0}{1+\d}{\k}(X) \). Therefore
\[
A=\bigcap_{i<\l} \left(\bigcup_{j<\widehat{\d}}A_{i,j} \right)=\bigcup_{s \in \pre{\l}{(\widehat{\d})}}\left(\bigcap_{i<\l}A_{i,s(i)} \right).
\]
Since \( \d \) is even, \( \Piii{0}{1+\d}{\k}(X) = \Piii{0}{1+\frac{\d}{2}}{\k^+}(X) \) by inductive hypothesis, and since the latter is closed under intersections of size \( \k \) by Proposition~\ref{k+_hierarchy_closure}\ref{k+_hierarchy_closure-2}, we have \( \bigcap_{i<\l}A_{i,s(i)} \in \Piii{0}{1+\d}{\k}(X) \) for every \( s \in \pre{\l}{(\widehat{\d})} \). Since \( \k \) is strong limit (because it is singular and \( 2^{<\k} = \k \)) and \( \widehat{\d},\l < \k \), we have \( \left|\pre{\l}{(\widehat{\d})}\right| < \k \), and so \( A \in \Siii{0}{1+\d+1}{\k}(X)\), as desired.
\end{proof}

\begin{cor} \label{cor:ordersingular}
Suppose that \(\k\) is \emph{singular}, and that \( \Siii{0}{1}{\k}(X) \subseteq \Siii{0}{2}{\cof(\k)^+}(X) \). Then for every \( \a < \k^+ \)  \[
\ord_{\k^+}(X) \leq 1+ \a \quad \IFF \quad \ord_{\k}(X) \leq 1+2 \cdot \a.
\]
In particular, if one of \( \ord_{\k^+}(X) \) or \( \ord_{\k}(X) \) is a limit ordinal \( \a \), then \( \ord_{\k^+}(X) = \ord_{\k}(X) = \a \). \end{cor}

Restricting the attention to infinite ordinals \( \a \), and using Proposition~\ref{prop:embed_T_0_spaces_into_Cantor} and Theorem~\ref{thm:non-collpase_general_space}, we also get:

\begin{cor}\label{collapse_k}
Suppose that \(\k\) is \emph{singular}. Then the \( \k \)-Borel hierarchy on \( X \) collapses if and only if the same happens for the \( \k^+ \)-Borel hierarchy on \( X \).

In particular, the \( \k \)-Borel hierarchy does not collapse on \( \pre{\k}{2} \) and all other spaces \( X \) such that there is a $\kappa^+$-Borel embedding of $\pre{\kappa}{2}$ into $X$.
\end{cor}

We now move to the closure properties of the pointclasses appearing in the \( \k \)-Borel hierarchy on \( X \).

\begin{prop} \label{k_hierarchy_closure1} 
Suppose that $\k$ is \emph{singular}, and that \( \Siii{0}{1}{\k}(X) \subseteq \Siii{0}{2}{\cof(\k)^+}(X) \). 
Given \( 0<\a < \k^+ \), let \( \widehat{\a} = \cof(\a) \) if \( \a \) is limit, and \( \widehat{\a} = \cof(\kappa) \) otherwise.  Then:
\begin{enumerate-(1)}
\item \label{k_hierarchy_closure1-1}
If \( \a \) is even, then \(\Siii{0}{1+\a}{\k}(X)\) is closed under unions of size \( \k\)  and intersections of size smaller than \( \widehat{\a} \). Therefore \(\Piii{0}{1+\a}{\k}(X) \) is closed under intersections of size \( \k\)  and unions of size smaller than \( \widehat{\a} \), and \(\Deee{0}{1+\a}{\k}(X)\) is a \( \widehat{\a} \)-algebra.
\item \label{k_hierarchy_closure1-2}
If \(\a\) is odd, then \(\Siii{0}{1+\a}{\k}(X)= \Piii{0}{1+\a}{\k}(X)=\Deee{0}{1+\a}{\k}(X)\) is a \( \cof(\k) \)-algebra.
\end{enumerate-(1)}
Furthermore, the same is true for \( \a = 0 \) if \( X \) is \( \cof(\k) \)-additive.%
\footnote{Recall that, more generally, the closure properties of \( \Siii{0}{1}{\k^+}(X) \) under intersections (and of \( \Piii{0}{1}{\k^+}(X) \) under unions) are completely determined by the additivity of the space.}
\end{prop}

\begin{proof}
Part~\ref{k_hierarchy_closure1-1} follows from Theorem~\ref{hierarchy_theorem}\ref{hierarchy_theorem_even} and Proposition~\ref{k+_hierarchy_closure}, once we notice that \( \alpha \) is limit if and only if \( 1+\frac{\a}{2} \) is limit. Part~\ref{k_hierarchy_closure1-2} follows instead from Theorem~\ref{hierarchy_theorem}\ref{hierarchy_theorem_odd} and Remark~\ref{rmk:easyclosure}.
The additional part concerning the case \( \a = 0 \) is obvious.
\end{proof}

We will show in Proposition~\ref{k_hierarchy_closure2} that the closure properties stated in Proposition~\ref{k_hierarchy_closure1} are optimal. In order to do this, we first need to consider properness of the \( \k \)-Borel hierarchy, and determine when the inclusions in~\eqref{eq:inclusionswhenincreasing} are strict for \( \g = \k \).
The following is the analogue of Corollary~\ref{cor:proper_for_k+} for the \( \k \)-Borel hierarchy.

\begin{proposition} \label{prop:proper_for_k} 
Suppose that \( \k \) is \emph{singular},
and that \( \Siii{0}{1}{\k}(X) \subseteq \Siii{0}{2}{\cof(\k)^+}(X) \).
Let \( \a < \k^+ \) be such that \( 1+ \a < \ord_{\k}(X) \).  
\begin{enumerate-(1)}
\item\label{prop:proper_for_k-1} 
If \( \a \) is even, then \( \Deee{0}{1+\a}{\k}(X) \subsetneq \Siii{0}{1+\a}{\k}(X) \) (equivalently, \( \Deee{0}{1+\a}{\k}(X) \subsetneq \Piii{0}{1+\a}{\k}(X) \)).
\item\label{prop:proper_for_k-2} 
  \( \bigcup_{\b < \a} (\Siii{0}{1+\b}{\k}(X) \cup \Piii{0}{1+\b}{\k}(X)) \subsetneq \Deee{0}{1+\a}{\k}(X). \)
  \end{enumerate-(1)}
  Moreover, part~\ref{prop:proper_for_k-2} holds also for \( 1+ \a = \ord_{\k}(X) \), if \( \ord_{\k}(X) \) is odd.
\end{proposition}

\begin{proof}
Part~\ref{prop:proper_for_k-1} follows from
Theorem~\ref{hierarchy_theorem}\ref{hierarchy_theorem_even} and Corollary~\ref{cor:proper_for_k+}.

The proof of part~\ref{prop:proper_for_k-2} follows the one of Proposition~\ref{prop:proper_for_k+}.
For \( \a = 0 \), it is enough to observe that \( \Deee{0}{1}{\k}(X) \neq \emptyset \), as witnessed by \( X \).
If $\alpha=1$, so that $\ord_{\k}(X)>2$, the result follows from Lemma~\ref{lem:door_spaces}.
Assume now \( \a \geq 2 \). 

If \( \a > 1 \) is odd, let \( \a' \) be its predecessor, so that \( \a' \) is an even ordinal not smaller than \( 2 \). By Proposition~\ref{k_hierarchy_closure1}\ref{k_hierarchy_closure1-1}, the class \( \Siii{0}{1+\a'}{\k}(X) \) is closed under unions of size \( \k \). Apply Proposition~\ref{prop:CB_for_hierarchy_collapse_successor_case} to the ordinals \( 1 = \beta < 1+\a' < \k^+ \) for \(\g=\l=\k\) and \( \B \) any basis for \( X \) satisfying \( |\B| \leq \k \), and let \( P \subseteq X \) be the closed set obtained this way. Then,  \( \ord_\k (B \cap P) > 1+ \a' \) for every \( B \in \B \) such that \( B \cap P \neq \emptyset \), and since \( \ord_\k(X) > 1+ \a'  \), we have \( P \neq \emptyset\) and \( \ord_\k(P) > 1+ \a' \), which in particular implies that \( P \) contains at least two distinct points \( x_0 \) and \( x_1 \). 
Let \( B_0,B_1 \in \B \) be disjoint open sets such that \( x_0 \in B_0 \) and \( x_1 \in B_1 \). Since \( \a' \) is even, by Proposition~\ref{prop:conditions_collapse} and Theorem~\ref{hierarchy_theorem}\ref{hierarchy_theorem_even} there are \( A_0 \in \Siii{0}{1+\a'}{\k}(B_0 \cap P) \setminus \Piii{0}{1+\a'}{\k}(B_0 \cap P) \) and \( A_1 \in \Piii{0}{1+\a'}{\k}(B_1 \cap P) \setminus \Siii{0}{1+\a'}{\k}(B_1 \cap P) \). Let \( A = A_0 \cup A_1 \). 
Since \( \Siii{0}{1+\a'}{\k} \) and \( \Piii{0}{1+\a'}{\k} \) are hereditary, and since \( \Deee{0}{1+\a}{\k}(X) \) is at least an \( \omega \)-algebra by Remark~\ref{rmk:closureunderfinite operations},
 one easily gets \( A \in \Deee{0}{1+\a}{\k}(X) \). 
Moreover, \( A \in \bigcup_{\b<\a} (\Siii{0}{1+\b}{\k}(X) \cup \Piii{0}{1+\b}{\k}(B)) = \Siii{0}{1+\a'}{\k}(X) \cup \Piii{0}{1+\a'}{\k}(X) \) is forbidden by the choice of \( A_0 \) and \( A_1\), hence we are done.

Let us move to the even case.
Consider first the case when \( \a \) is a successor ordinal, so that \( \a = \b+2 \) for some even ordinal \( \b \geq 2 \).
Fix any increasing sequence of ordinals \( \seq{\k_i}{i < \cof(\k)} \) cofinal in \( \k \).

\begin{claim}\label{claim:newtakeonevenordinals}
There is \( x \in X \) and a partition \( \seq{B_i}{i < \cof(\k)} \) of \( X \setminus \{ x \} \) such that \( B_i \in \Deee{0}{3}{\k}(X) \) and \( \ord_\k(B_i) > 1+\a \) for every \( i < \cof(\k) \).
\end{claim}

\begin{proof}[Proof of the Claim]
Let \( f \colon X \to \pre{\k}{2}\) be defined as in the proof of Proposition~\ref{prop:embed_T_0_spaces_into_Cantor}, so that \( f \) is injective and \( \Siii{0}{3}{\k} \)-measurable (because \( \weight(\pre{\k}{2}) = 2^{< \k} = \k \) and \( \Siii{0}{3}{\k}(X) \) is closed under unions of size \( \k \) by Proposition~\ref{k_hierarchy_closure1}). Let 
\[ 
S = \{ s \in \pre{<\k}{2} \mid \ord_{\k}( f^{-1}(\clopen{s})) \leq 1+ \a \} , 
\]
and let \( O = \bigcup \{  \clopen{s} \mid s \in S \} \). 
 
If \( C = f[X] \setminus O \) were discrete, then it would have size at most \( \k \), so that \( |f^{-1}(C)| \leq \k  \) too. This would mean that \( f^{-1}(C) \in \Siii{0}{2}{\k^+}(X) = \Siii{0}{3}{\k}(X) \), and also \( \ord_\k(f^{-1}(C)) \leq 3 \). In particular, \( \{ f^{-1}(\clopen{s}) \mid s \in S \} \cup \{ f^{-1}(C) \} \) would be a \( \k \)-sized covering of \( X \) consisting of \( \Siii{0}{3}{\k} \)-subspaces with order at most \( 1+ \a \), and hence \( \ord_\k(X) \leq 1+ \a \) because \( \Siii{0}{3}{\k}(X) \subseteq \Siii{0}{1+\a}{\k}(X) \) and \( \Siii{0}{1+\a}{\k}(X) \) is closed under finite intersections and unions of size \( \k \) (Proposition~\ref{k_hierarchy_closure1}). This contradicts the fact that \( 1 + \a < \ord_\k(X) \), and so we conclude that \( C \) is not discrete.

Let \( y \in C \) be a non-isolated point in \( C \), and set \( x = f^{-1}(y) \). 
Recursively construct a sequence of ordinals \( \seq{\b_i}{i < \cof(\k)} \) cofinal in \( \k \) such that 
$U_i \cap C \neq \emptyset$ for \( U_i =  \clopen{y\restriction \sup_{j < i} \b_j}\setminus \clopen{y\restriction \b_i} \).
Finally, let \( B_i = f^{-1}(U_i) \). 
Then \( B_i \in \Deee{0}{3}{\k}(X) \) because \( U_i \) is clopen, and \( \seq{B_i}{i < \cof(\k)} \) is a partition of \( X \setminus \{ x \} \) because \( \seq{U_i}{i < \cof(\k)} \) is a partition of \( \pre{\k}{2} \setminus \{ y \}  \). Moreover, since the clopen set \( U_i \) intersects \( C \), there is $s\in \pre{\k}{2}\setminus S$ such that $\clopen{s}\subseteq U_i$. Therefore \( f^{-1}(\clopen{s}) \subseteq B_i\), and hence \( \ord_\k(B_i) > 1+\a \) by \( s \notin S \) and Proposition~\ref{prop:abstractpropertiesofcollapse}\ref{prop:abstractpropertiesofcollapse-1}.
\end{proof}

\begin{claim} \label{claim:prop:Delta_alpha_eq_Delta_alphaplus1_iff_collapse}
If \( \ord_\k(Y) > 1 + \a \), then for every cardinal \( \l < \k \) there is \( C \in \Siii{0}{1+\b+1}{\k}(Y) \setminus \left(\Piii{0}{1+\beta}{\kappa} (Y)\right)_{\sigma_{\lambda}} \).
\end{claim}

\begin{proof}[Proof of the Claim]
Suppose not, so that \( \Siii{0}{1+\b+1}{\k}(Y) = \left( \Piii{0}{1+\b}{\k}(Y) \right)_{\sigma_\l} \). Since \( \b+1 \) is odd, by Theorem~\ref{hierarchy_theorem} we have \( \Siii{0}{1+\b+1}{\k}(Y) = \Piii{0}{1+\b+1}{\k}(Y) \) and \( \Siii{0}{1+\a}{\k}(Y) = \left(  \Piii{0}{1+\b+1}{\k}(Y) \right)_{\sigma_{\cof(\k)}}\). Therefore \( \Siii{0}{1+\a}{\k}(Y) \subseteq \left( \Piii{0}{1+\b}{\k}(Y) \right)_{\sigma_\mu}\) for \( \mu = \l \cdot \cof(\k) < \k \), hence \( \Siii{0}{1+\a}{\k}(Y) = \Siii{0}{1+\b+1}{\k}(Y) = \Deee{0}{1+\b+1}{\k}(Y)  \) is a \( \k \)-algebra (in fact, a \( \k^+ \)-algebra) because \( \Siii{0}{1+\a}{\k}(Y) \) is closed under unions of size \( \k \) by Proposition~\ref{k_hierarchy_closure1}\ref{k_hierarchy_closure1-1}. It follows that \( \Bor{\k}(Y) \subseteq \Siii{0}{1+\b+1}{\k}(Y) \), contradicting \( \ord_\k(Y) > 1+ \a \).
\end{proof}

Let \( \seq{B_i}{i < \cof(\k)} \) be as in Claim~\ref{claim:newtakeonevenordinals}.
For each \( i < \cof(\k) \), apply Claim~\ref{claim:prop:Delta_alpha_eq_Delta_alphaplus1_iff_collapse} with \( Y = B_i \) and \( \l = \k_i \) in order to find \( A_i \in \Siii{0}{1+\b+1}{\k}(B_i ) \setminus \left(\Piii{0}{1+\beta}{\kappa} (B_i )\right)_{\sigma_{\k_i}} \). Let \( A = \bigcup_{i < \cof(\k)} A_i \): we claim that \( A \in \Deee{0}{1+\a}{\k}(X) \setminus \Siii{0}{1+\b+1}{\k}(X) \), which is enough because \( \b+1 \) is odd and hence \( \Siii{0}{1+\b+1}{\k}(X) = \Piii{0}{1+\b+1}{\k}(X) \) by Theorem~\ref{hierarchy_theorem}\ref{hierarchy_theorem_odd}. 
On the one hand, using the fact that \( \Siii{0}{1+\b+1}{\k}(X) \) is hereditary and the closure properties of \( \Siii{0}{1+\a}{\k}(X) \) provided in Proposition~\ref{k_hierarchy_closure1}\ref{k_hierarchy_closure1-1}, it is easy to see that both \( A \) and
\[
X \setminus A = \{ x \} \cup \bigcup\nolimits_{i < \cof(\k)} \left( B_i \setminus A_i \right)
\]
belong to \( \Siii{0}{1+\a}{\k}(X) \) because \( 1+ \a \geq 3 \).
On the other hand, if \( A \in \Siii{0}{1+\b+1}{\k}(X) \), then there would be \( \l < \k \) such that \( A \in \left( \Piii{0}{1+\b}{\k}(X) \right)_{\sigma_{\l}}\). Let \( i < \cof(\k) \) be such that \( \l \leq \k_i \). 
Then \( A_i = A \cap B_i  \in \left( \Piii{0}{1+\b}{\k}(B_i ) \right)_{\sigma_{\k_i}} \), against the choice of \( A_i \).

Finally, let \( \a \) be limit. Since in this case \(  1+ \a = 1 + 2 \cdot \a \), we have \( \ord_{\k^+}(X) > 1+\a = \a \) by \( \ord_\k(X) > 1+\a \) and Corollary~\ref{cor:ordersingular}. Moreover, since \( \frac{\a}{2} = \a \) we have \( \Deee{0}{1+\a}{\k}(X)  = \Deee{0}{\a}{\k^+}(X) \) and 
\[ 
\bigcup_{\b < \a} (\Siii{0}{1+\b}{\k}(X) \cup \Piii{0}{1+\b}{\k}(X)) = \bigcup_{1 \leq \b < \a} (\Siii{0}{\b}{\k^+}(X) \cup \Piii{0}{\b}{\k^+}(X))
\]
by the fact that the \( \k \)-Borel hierarchy is increasing and Theorem~\ref{hierarchy_theorem}\ref{hierarchy_theorem_even}.
Therefore the result follows from Corollary~\ref{cor:proper_for_k+}\ref{cor:proper_for_k+-2}.
        \end{proof}

At the current stage, it is not yet clear whether it can happen that \( 2 < \ord_\k(X)  < \k^+\) if \( \o < \cof(\k) < \k \). Nevertheless, we record the following result that complements Proposition~\ref{prop:proper_for_k}\ref{prop:proper_for_k-2} for the case $\a =\ord_{\k}(X)$ and is the
analogue of Corollary~\ref{cor:proper_for_k+_final_level} for the $\k$-Borel hierarchy.

\begin{proposition}\label{prop:proper_for_k_alpha=ord} 
Suppose that \( \k \) is \emph{singular},
and that \( \Siii{0}{1}{\k}(X) \subseteq \Siii{0}{2}{\cof(\k)^+}(X) \).
The following are equivalent:
\begin{enumerate-(1)}
    \item\label{prop:proper_for_k_alpha=ord-1}  there is $A\in \Deee{0}{\ord_{\kappa}(X)}{\k}(X)\setminus \bigcup_{1 \leq \b < \ord_{\kappa}(X)} \left( \Siii{0}{\b}{\k}(X) \cup \Piii{0}{\b}{\k}(X) \right)$,
    \item\label{prop:proper_for_k_alpha=ord-2} $X$ does not have exactly one non-isolated point.
\end{enumerate-(1)}
In particular, \ref{prop:proper_for_k_alpha=ord-1} holds whenever $\ord_{\kappa^+}(X)\neq 2$.
\end{proposition}

\begin{proof}
If $X$ has exactly one non-isolated point, then $\ord_{\k}(X)=2$ and $\Deee{0}{2}{\k}(X)= \Siii{0}{1}{\k}(X) \cup \Piii{0}{1}{\k}(X)$ by Lemma~\ref{lem:door_spaces}. This shows
\ref{prop:proper_for_k_alpha=ord-1}~\( \Rightarrow \)~\ref{prop:proper_for_k_alpha=ord-2}.
 
We now show that \ref{prop:proper_for_k_alpha=ord-2}~\( \Rightarrow \)~\ref{prop:proper_for_k_alpha=ord-1}, using ideas and methods developed in the proofs of Propositions~\ref{prop:proper_for_k+} and~\ref{prop:proper_for_k}.
If \( \ord_\k(X) = 1 \), then~\ref{prop:proper_for_k_alpha=ord-1} reduces to \( \Deee{0}{1}{\k}(X) \neq \emptyset \), which is true because e.g.\ \( X \in \Deee{0}{1}{\k}(X) \). 
If \( \ord_k(X) = 2 \), then \( X \) is not discrete and thus \( X \) has at least two non-isolated points by~\ref{prop:proper_for_k_alpha=ord-2}: therefore~\ref{prop:proper_for_k_alpha=ord-1} holds by Lemma~\ref{lem:door_spaces}.

If \( \ord_\k(X) = 1+\a \) with $\a$ odd, then the result follows from Proposition~\ref{prop:proper_for_k}.

If instead \( \ord_\k(X) = 1+\a \) for \( \a \) an even successor ordinal, then we use the fact that the predecessor \( \a' \) of \( \a \) is odd, and hence the class \( \Siii{0}{1+\a'}{\k} \) is selfdual on \( X \). Thus, if \( \Deee{0}{\ord_\k(X)}{\k}(X) =  \Deee{0}{1+\a'+1}{\k}(X) \subseteq \Siii{0}{1+\a'}{\k}(X) \cup \Piii{0}{1+\a'}{\k}(X) \), then we would have \( \Bor{\k}(X) \subseteq \Siii{0}{1+\a'}{\k}(X) \cup \Piii{0}{1+\a'}{\k}(X) \subseteq \Siii{0}{1+\a'}{\k}(X) \), and hence \( \ord_\k(X) \leq 1+\a' < 1+\a \), a contradiction.

Finally, suppose that \( \ord_\k(X) = 1+\a \) for \( \a \) limit. Then \( \ord_{\k^+}(X) = \ord_\k(X) = 1+\a = \a \) by Corollary~\ref{cor:ordersingular}. Moreover, \( \Deee{0}{\ord_\k(X)}{\k}(X) = \Deee{0}{\ord_{\k^+}(X)}{\k^+}(X) \) and
\[ 
\bigcup_{1 \leq \b < \ord_\k(X)} (\Siii{0}{\b}{\k}(X) \cup \Piii{0}{\b}{\k}(X)) = \bigcup_{1 \leq \b < \ord_{\k^+}(X)} (\Siii{0}{\b}{\k^+}(X) \cup \Piii{0}{\b}{\k^+}(X))
\]
by the fact that the \( \k \)-Borel hierarchy is increasing and Theorem~\ref{hierarchy_theorem}\ref{hierarchy_theorem_even}.
Therefore the result follows from Corollary~\ref{cor:proper_for_k+_final_level}.
\end{proof}

We are now ready to prove a counterpart of Proposition~\ref{prop:conditions_collapse} for the special case of a singular cardinal \( \g = \k \) (which requires the additional assumption \( \Siii{0}{1}{\k}(X) \subseteq \Siii{0}{2}{\cof(\k)^+}(X) \)). The major difference is that condition~\ref{prop:conditions_collapse_singular-8} can be included in the list only when \( \a \) is even: for example, by Theorem~\ref{hierarchy_theorem}\ref{hierarchy_theorem_odd}, for every odd \( \a \) we have \ \( \Siii{0}{1+\a}{\k}(\pre{\k}{2}) = \Piii{0}{1+\a}{\k}(\pre{\k}{2}) \), yet the \( \k \)-Borel hierarchy on \( \pre{\k}{2}\) does not collapse by Corollary~\ref{collapse_k}.

\begin{proposition}\label{prop:conditions_collapse_singular}
Suppose that \( \k \) is \emph{singular}, and that \( \Siii{0}{1}{\k}(X) \subseteq \Siii{0}{2}{\cof(\k)^+}(X) \). 
Then for every \( \a < \k^+ \), the following are equivalent:
\begin{enumerate}[label={\upshape (\arabic*)}, leftmargin=2pc]
\item \label{prop:conditions_collapse_singular-1}
$\ord_\k(X) \leq 1+\alpha$;
\item \label{prop:conditions_collapse_singular-2}
for every \( \beta \geq \alpha\), \( \Siii{0}{1+\b}{\k}(X) = \Piii{0}{1+\b}{\k}(X) = \Deee{0}{1+\b}{\k}(X) = \Bor{\k}(X)\);
\item \label{prop:conditions_collapse_singular-3}
for some \( \beta > \alpha \), one of \( \Siii{0}{1+\a}{\k}(X) \), \( \Piii{0}{1+\a}{\k}(X) \), or \( \Deee{0}{1+\a}{\k}(X) \) coincides with one of \( \Siii{0}{1+\b}{\k}(X) \) or \( \Piii{0}{1+\b}{\k}(X) \);
\item \label{prop:conditions_collapse_singular-4}
the class $\Siii{0}{1+\alpha}{\k}(X)$ is closed under intersections of size $\k$ (equivalently, $\Piii{0}{1+\alpha}{\k}(X)$ is closed under unions of size $\k$);
\item \label{prop:conditions_collapse_singular-5}
the class $\Siii{0}{1+\alpha}{\k}(X)$ is closed under intersections shorter than $\k$ (equivalently, $\Piii{0}{1+\alpha}{\k}(X)$ is closed under unions shorter than $\k$);
\item \label{prop:conditions_collapse_singular-9}
for some \( \beta > \alpha \), one of \( \Siii{0}{1+\a}{\k}(X) \), \( \Piii{0}{1+\a}{\k}(X) \), or \( \Deee{0}{1+\a}{\k}(X) \) coincides with \( \Deee{0}{1+\b}{\k}(X) \).
\end{enumerate}
Moreover, if \( \a \geq 1 \) the above conditions are also equivalent to
\begin{enumerate}[resume*]
\item \label{prop:conditions_collapse_singular-6}
the class $\Deee{0}{1+\alpha}{\k}(X)$ is closed under intersections of size $\k$, and hence it is a \( \k^+ \)-algebra;
\item \label{prop:conditions_collapse_singular-7}
the class $\Deee{0}{1+\alpha}{\k}(X)$ is closed under intersections shorter than $\k$, and hence it is a \( \k \)-algebra.
\end{enumerate}
If \( \a \) is even, conditions~\ref{prop:conditions_collapse_singular-1}-\ref{prop:conditions_collapse_singular-9} are also equivalent to the following one:
\begin{enumerate}[resume*]
\item \label{prop:conditions_collapse_singular-8}
the boldface pointclass $\Siii{0}{1+\alpha}{\k}$ is selfdual on \( X \), i.e.\ $\Siii{0}{1+\alpha}{\k}(X) = \Piii{0}{1+\alpha}{\k}(X)$.
\end{enumerate}
In particular, if \( \a\) is even and \( \a\geq 2 \), then all conditions~\ref{prop:conditions_collapse_singular-1}-\ref{prop:conditions_collapse_singular-8} are equivalent.
\end{proposition}

Notice that~\ref{prop:conditions_collapse_singular-6} and~\ref{prop:conditions_collapse_singular-7} are equivalent to the other conditions even when \( \a = 0 \), if we additionally require \( X \subseteq \pre{\k}{2}\).

\begin{proof}
The equivalence among conditions \ref{prop:conditions_collapse_singular-1}, \ref{prop:conditions_collapse_singular-2}, \ref{prop:conditions_collapse_singular-3} and, under the assumption \( \a \geq 1 \), \ref{prop:conditions_collapse_singular-6} and \ref{prop:conditions_collapse_singular-7} follows from Proposition~\ref{prop:conditions_collapse}, as the $\k$-Borel hierarchy is increasing (because we assumed \( \Siii{0}{1}{\k}(X) \subseteq \Siii{0}{2}{\cof(\k)^+}(X) \)), and hence also $\Siii{0}{1}{\k}(X)\cup \Piii{0}{1}{\k}(X)\subseteq \Siii{0}{1+\a}{\k}(X)$ if $\a\geq 1$. For condition~\ref{prop:conditions_collapse_singular-6}, we also use the fact that every $\k$-algebra is a $\k^+$-algebra, by singularity of $\k$.

\ref{prop:conditions_collapse_singular-2}~\( \Rightarrow \)~\ref{prop:conditions_collapse_singular-4}.
Follows from the fact that \( \Bor{\k}(X) = \Bor{\k^+}(X) \) because \( \k \) is singular.

\ref{prop:conditions_collapse_singular-4}~\( \Rightarrow \)~\ref{prop:conditions_collapse_singular-5}.
Obvious.

\ref{prop:conditions_collapse_singular-5}~\( \Rightarrow \)~\ref{prop:conditions_collapse_singular-9}.
Under~\ref{prop:conditions_collapse_singular-5}, we have 
\[ 
\Siii{0}{1+\a}{\k}(X) \subseteq \Deee{0}{1+\a+1}{\k}(X) \subseteq \Piii{0}{1+\a+1}{\k}(X) \subseteq \Siii{0}{1+\a}{\k}(X) .
\]

\ref{prop:conditions_collapse_singular-9}~\( \Rightarrow \)~\ref{prop:conditions_collapse_singular-1}. 
If \( \a = 0 \) the results is trivial, as in all cases we would get \( \Deee{0}{1}{\k}(X) = \Deee{0}{1+\b}{\k}(X) \), and hence \( X \), being Hausdorff, would be discrete because \( \Piii{0}{1}{\k}(X) \subseteq \Deee{0}{1+\b}{\k}(X) = \Deee{0}{1}{\k}(X) \). Thus we can suppose that \( \a \geq 1 \).
We prove the contrapositive, so assume \( \ord_\k(X) > 1 + \a \) and fix an arbitrary \( \b > \a \). We distinguish two cases. If \( 1+ \b < \ord_k(X) \), then \( \Deee{0}{1+\b}{\k}(X) \setminus (\Siii{0}{1+\a}{\k}(X) \cup \Piii{0}{1+\a}{\k}(X) ) \neq \emptyset\) by Proposition~\ref{prop:proper_for_k}. Assume now that \( \ord_\k(X) \leq 1+ \b \). Since \( 2 \leq 1+\a < \ord_\k(X) \), then Proposition~\ref{prop:proper_for_k_alpha=ord}\ref{prop:proper_for_k_alpha=ord-1} applies and \( \Deee{0}{\ord_\k(X)}{\k}(X) \setminus (\Siii{0}{1+\a}{\k}(X) \cup \Piii{0}{1+\a}{\k}(X) ) \neq \emptyset \). Since \( \Deee{0}{\ord_\k(X)}{\k}(X) \subseteq \Deee{0}{1+\b}{\k}(X) \) by case assumption, we are done.

Finally, let \( \a \) be even. The implication
\ref{prop:conditions_collapse_singular-2}~\( \Rightarrow \)~\ref{prop:conditions_collapse_singular-8} is obvious, so to conclude the proof it is enough to show that~\ref{prop:conditions_collapse_singular-8}~\( \Rightarrow \)~\ref{prop:conditions_collapse_singular-1}. Since \( \a \) is even, \( \Siii{0}{1+\a}{\k}(X) \) is closed under unions of size \( \k \) by Proposition~\ref{k_hierarchy_closure1}\ref{k_hierarchy_closure1-1}, and hence it is a \( \k^+ \)-algebra by~\ref{prop:conditions_collapse_singular-8}. Since \( \Siii{0}{1}{\k}(X) \subseteq \Siii{0}{1+\a}{\k}(X) \), we get in particular that \( \Bor{\k}(X) \subseteq \Siii{0}{1+\a}{\k}(X)\), and hence \( \ord_\k(X) \leq 1+ \a \).
\end{proof}

We can now come back to the closure properties of the pointclasses in the \( \k \)-Borel hierarchy (Proposition~\ref{k_hierarchy_closure1}), and prove their optimality.

\begin{prop} \label{k_hierarchy_closure2}
Suppose that $\k$ is \emph{singular}, and that \( \Siii{0}{1}{\k}(X) \subseteq \Siii{0}{2}{\cof(\k)^+}(X) \). 
 Let \( \alpha \) be such that \( 1 + \a < \ord_{\k}(X)\), and let once again \( \widehat{\a} = \cof(\a) \) if \( \a \) is limit, and \( \widehat{\a} = \cof(\kappa) \) otherwise.
\begin{enumerate-(1)}  \item \label{k_hierarchy_closure2-1}
If \(\a\) is even, then \(\Siii{0}{1+\a}{\k}(X)\) is not closed under complements or intersections of size \( \widehat{\a} \), \(\Piii{0}{1+\a}{\k}(X) \) is not closed under complements or unions of size \( \widehat{\a} \). If furthermore \( \a > 0 \), then \(\Deee{0}{1+\a}{\k}(X)\) is not closed under unions or intersections of size \( \widehat{\a} \), and the same is true for \( \a = 0 \) if \( X \subseteq \pre{\k}{2}\).
\item \label{k_hierarchy_closure2-2}
If \(\a\) is odd, then \(\Siii{0}{1+\a}{\k}(X)= \Piii{0}{1+\a}{\k}(X)=\Deee{0}{1+\a}{\k}(X)\) is not closed under unions or intersections of size \( \cof (\kappa)\).
\end{enumerate-(1)}
\end{prop}

As already observed, \( \Deee{0}{1}{\k}(X) \) might instead be closed under arbitrary unions and intersections, depending on the space \( X \).

\begin{proof}
We begin with the case when \( \a \) is even. Suppose first that \( \a = 0 \). If \( \Siii{0}{1}{\k}(X) \) were closed under intersections of size \( \cof(\k) \) (equivalently: if \( \Piii{0}{1}{\k}(X) \) were closed under unions of size \( \cof(\k) \)), then every singleton \( \{ x \} \) would also be open, as our assumptions entail that \( \Piii{0}{1}{\k}(X) \subseteq \Piii{0}{2}{\cof(\k)^+}(X) \) and singletons are closed. But this means that \( X \) is discrete, contradicting the assumption \( \ord_\k(X) > 1 \).
As for \( \Deee{0}{1}{\k}(X)\), if it were closed under unions (or equivalently, intersections) of size \( \cof(\k) \) and \( X \subseteq \pre{\k}{2}\), then by Fact~\ref{fact:open=unionclopen} we would have \( \Siii{0}{1}{\k}(X) \subseteq \Deee{0}{1}{\k}(X)\), and hence \( X \) would again be discrete, against \( \ord_\k(X) > 1 \).
 
Suppose now that \( \a \geq 2 \) is even. Arguing as in the proof of Proposition~\ref{prop:optimalityregular}, the fact that \( \Piii{0}{1+\a}{\k}(X) \) is closed under intersections of size \( \k \) by Proposition~\ref{k_hierarchy_closure1}\ref{k_hierarchy_closure1-1} yields that it is enough to show that \( \Deee{0}{1+\a}{\k}(X) \) is not closed under intersections of size \( \widehat{\a} \). Suppose towards a contradiction that this is not true. Since \( \bigcup_{\b < \a} \Siii{0}{1+\b}{\k}(X) \subseteq\Deee{0}{1+\a}{\k}(X) \), applying Theorem~\ref{hierarchy_theorem}\ref{hierarchy_theorem_even} we get that 
\[
\Piii{0}{1+\a}{\k}(X) = \left( \bigcup_{\b < \a} \Siii{0}{1+\a}{\k}(X) \right)_{\delta_{\widehat{\a}}} \subseteq \Deee{0}{1+\a}{\k}(X).
\]
This easily leads to \( \Siii{0}{1+\a}{\k}(X) = \Piii{0}{1+\a}{\k}(X) \), which contradicts Proposition~\ref{prop:conditions_collapse_singular} since we assumed \( \ord_{\k}(X) > 1+\a \).

Suppose now that \( \a \) is odd. It is enough to consider the case of unions. Since \( \a + 1 \) is an even successor ordinal, by Theorem~\ref{hierarchy_theorem}\ref{hierarchy_theorem_even} we get that if \( \Piii{0}{1+\a}{\k}(X) \) were closed under unions of size \( \cof(\k) \) then \( \Siii{0}{1+\a+1}{\k}(X) = \left( \Piii{0}{1+\a}{\k}(X) \right)_{\sigma_{\cof(\k)}} = \Piii{0}{1+\a}{\k}(X) \), which by \( \ord_\k(X) > 1+ \a \) contradicts again Proposition~\ref{prop:conditions_collapse_singular}.
\end{proof}

From Proposition~\ref{k_hierarchy_closure2} and Theorem~\ref{hierarchy_theorem}\ref{hierarchy_theorem_even} we can derive the optimality of the closure properties of the pointclasses in the $\k^+$-Borel hierarchy when \( \k \) is singular, thus complementing Proposition \ref{prop:optimalityregular}.

\begin{cor}\label{cor_hier_clos}
Suppose that \( \k \) is \emph{singular}, and that \( \Siii{0}{1}{\k}(X) \subseteq \Siii{0}{2}{\cof(\k)^+}(X) \).
 For \( 1 \leq \alpha < \ord_{\k^+}(X)\), let \( \widehat \a = \cof(\k) \) if \( \a \) is a successor ordinal, and \( \widehat \a = \cof(\a) \) if \( \a \) is limit. Then:
\begin{enumerate-(1)}
\item \label{cor_hier_clos-1}
\( \Siii{0}{\a}{\k^+}(X) \) is not closed under complements and  intersections of size \( \widehat{\a} \);
\item \label{cor_hier_clos-2}
\( \Piii{0}{\a}{\k^+}(X) \) is not closed under complements and unions of size \( \widehat{\a} \);
\item \label{cor_hier_clos-3}
if \( \a > 1 \), then \( \Deee{0}{\a}{\k^+}(X) \) is not closed under unions or intersections of size \( \widehat{\a} \), and the same is true for \( \a = 1 \) if \( X \subseteq \pre{\k}{2}\).
\end{enumerate-(1)}
\end{cor}

We conclude this section by briefly discussing the existence of universal and complete sets for the various pointclasses in the \( \k \)-Borel hierarchy. The following results are the counterpart of Proposition~\ref{universal_sigma_pi}.

\begin{proposition}\label{universal_sigma_pi_for_k-Borel} 
Suppose that \( \k \) is singular, and that \( \Siii{0}{1}{\k}(X) \subseteq \Siii{0}{2}{\cof(\k)^+}(X) \). Let \( \a < \k^+ \).
\begin{enumerate-(1)}
\item \label{universal_sigma_pi_for_k-Borel-1}
If $\alpha$ is even, then there are \(\pre{\k}{2}\)-universal sets for both \(\Siii{0}{1+\a}{\k}(X)\) and \(\Piii{0}{1+\a}{\k}(X)\). 
Therefore there exist subsets of \( \pre{\k}{2}\) that are \(\k\)-complete for \( \Siii{0}{1+\a}{\k} \) and \( \Piii{0}{1+\a}{\k}\).
In contrast, there is no \( \pre{\k}{2} \)-universal set for \( \Deee{0}{1+\a}{\k}(\pre{\k}{2}) \) or \( \Bor{\k}(\pre{\k}{2}) \).
\item \label{universal_sigma_pi_for_k-Borel-2}
If $\alpha$ is odd, then there are neither \( \pre{\k}{2}\)-universal sets for \( \Siii{0}{1+\a}{\k}(\pre{\k}{2}) \) nor subsets of \( X \) that are \(\k\)-complete for \( \Siii{0}{1+\a}{\k} \), and the same for \( \Piii{0}{1+\a}{\k} \) and \( \Deee{0}{1+\a}{\k} \).
\end{enumerate-(1)}
Moreover, \( \pre{\k}{2}\) can be replaced by \( X \) in all the above statements as soon as \( \pre{\k}{2}\) embeds into \( X \).
\end{proposition}

\begin{proof}
Part~\ref{universal_sigma_pi_for_k-Borel-1} follows from Theorem~\ref{hierarchy_theorem}\ref{hierarchy_theorem_even} and Proposition~\ref{universal_sigma_pi}, so let us move to part~\ref{universal_sigma_pi_for_k-Borel-2}. 

Assume that \( \a \) is odd, so that \( \Siii{0}{1+\a}{\k}(\pre{\k}{2}) = \Piii{0}{1+\a}{\k}(\pre{\k}{2}) = \Deee{0}{1+\a}{\k}(\pre{\k}{2}) \) by Theorem~\ref{hierarchy_theorem}\ref{hierarchy_theorem_odd}. 
Then none of those classes can have a \( \pre{\k}{2} \)-universal set by Lemma~\ref{universal_selfdual}. 

Next, let us consider complete sets.
Let \( \b \) be such that \( \a = \b +1 \). Then \( \Siii{0}{1+\a}{\k}(X) = \left( \Piii{0}{1+\b}{\k}(X) \right)_{\sigma_{<\k}} = \bigcup_{\l < \k}\left( \Piii{0}{1+\b}{\k}(X) \right)_{\sigma_{\l}} \).

\begin{claim} \label{claim:universal_sigma_pi_for_k-Borel}
For every \( \l < \k \), the class \( \left( \Piii{0}{1+\b}{\k}(X) \right)_{\sigma_{\l}} \) has a \( \pre{\k}{2}\)-universal set.
\end{claim}

Indeed, \( \b \) is even and hence, by part~\ref{universal_sigma_pi_for_k-Borel-1}, \(\Piii{0}{1+\b}{\k}(X)\) has a \( \pre{\k}{2}\)-universal set. Applying Lemma~\ref{lem:universalinborelhierarchy}\ref{lem:universalinborelhierarchy-2}, we easily get the desired result.

Towards a contradiction, assume that \(  A \subseteq X \) is \(\k\)-complete for \( \Siii{0}{1+\a}{\k} \). Let \( \l < \k \) be such that \( A \in \left( \Piii{0}{1+\b}{\k}(X) \right)_{\sigma_{\l}} \). Since the latter is closed under continuous preimages, we get that \( \Siii{0}{1+\a}{\k}(\pre{\k}{2}) = \left( \Piii{0}{1+\b}{\k}(\pre{\k}{2}) \right)_{\sigma_{\l}} \). Since \( \Siii{0}{1+\a}{\k} \) is a boldface pointclass that is selfdual on \( \pre{\k}{2}\), while \( \left( \Piii{0}{1+\b}{\k}(\pre{\k}{2}) \right)_{\sigma_{\l}} \) has a \( \pre{\k}{2}\)-universal set by Claim~\ref{claim:universal_sigma_pi_for_k-Borel}, this contradicts Lemma~\ref{universal_selfdual}.

Finally, the additional part can be inferred arguing as in the second part of the proof of Proposition~\ref{universal_sigma_pi}.
\end{proof}

The proof of Proposition~\ref{universal_sigma_pi_for_k-Borel} also gives the interesting information that if \( \pre{\k}{2}\) embeds into \( X \), then for every odd ordinal \( \a = \b+1 < \k^+\) and every cardinal \( \l < \k \),
\[
\left( \Piii{0}{1+\b}{\k}(X)\right) _{\sigma_\l}\subsetneq \Siii{0}{1+\alpha}{\k}(X).
\]
In other words: the class \( \Siii{0}{1+\a}{\k}(X) \) cannot be obtained by considering only unions of a \emph{fixed} cardinality below \( \k \). This is in stark contrast with what happens when \( \a \) is even: by Theorem~\ref{hierarchy_theorem}\ref{hierarchy_theorem_even}, in that case it is enough to consider unions of a fixed size \( \widehat{\a} < \k \).

\section{\texorpdfstring{\( \kappa \)-Thin $\kappa^+$-Borel sets in ${<}\kappa$-closed forcing extensions}{kappa-Thin kappa+-Borel sets in <kappa-closed forcing extensions}} \label{sec: psp section}
The initial motivation for the results in this section is Theorem \ref{thm:non-collpase_general_space}, which, in full alignment with classical results, provides a sufficient criterion for the non-collapse of the $\kappa^+$-Borel hierarchy: the existence of a \( \kappa \)-perfect subset. The following question is left open: suppose $X$ is a space of weight at most $\kappa$ such that its $\kappa^+$-Borel hierarchy does not collapse. Is there necessarily a $\kappa^+$-Borel embedding of $\pre{\kappa}{2}$ into $X$? On $\omega$ there are consistent examples of spaces violating this, notably in a model of Miller \cite[Theorem~22]{miller_length_1979} in which all uncountable subspaces of the reals have order $\omega_1$, while one can simultaneously have $2^\omega > \omega_1$.

We will show that when $\kappa$ is a regular cardinal, the question can be answered in the negative as well; however, we will use arguments of a different nature than for $\kappa = \omega$. Additionally, we will provide counterexamples from a particularly well-behaved class of spaces, namely the closed subspaces of $\pre{\kappa}{\kappa}$; on $\omega$, the Perfect Set Property forbids the existence of such examples. From this point onward, the regularity of $\kappa$ will always be tacitly assumed. 
     Furthermore, we will restrict our attention to subspaces of $\pre{\kappa}{\kappa}$, for two reasons.
First, providing examples among subspaces of $\pre{\kappa}{\kappa}$ is more interesting than simply constructing a topological space with no additional structure. 
Also, this allows us to deal with matters of definability and provide results that hold in general for a wide class of subspaces of $\pre{\kappa}{\kappa}$.
Nevertheless, notice that these results can be further extended to other \( T_0 \) topological spaces of weight at most \( \kappa \) anyway, thanks to Proposition~\ref{fact: wlog subsets cantor space}.

There are several equivalent versions of the \(\kappa\)-Perfect Set Property for subsets of \( \pre{\kappa}{2}\) (see, for example, \cite[p.~3]{SzirakiVaananenMR3661728}). In this section, we adopt the following standard formulation.

\begin{definition}\label{def:perfect}
We say that a set $P \subseteq \pre{\kappa}{\kappa}$ is \markdef{\( \kappa \)-perfect} if it is closed and homeomorphic to the generalized Cantor space $\pre{\kappa}{2}$, and call a set \markdef{\( \kappa \)-thin} if it has no \( \kappa \)-perfect subset.
\end{definition}

Accordingly, we will say that a set \( A \subseteq \pre{\k}{\k} \) has the \markdef{\( \k \)-Perfect Set Property} if either \( |A| \leq \k \), or \( A \) contains a \( \k \)-perfect set. Moreover, we say that a boldface pointclass \( \boldsymbol{\Gamma}\) has the \( \k \)-Perfect Set Property if there are no \( \k \)-thin sets in \( \boldsymbol{\Gamma}(\pre{\k}{\k}) \) of size greater than \( \k \), that is, if every \( A \in \boldsymbol{\Gamma}(\pre{\k}{\k}) \) has the \( \k \)-Perfect Set Property. This is the strongest form of the property that can be formulated in purely topological terms.

\subsection{\texorpdfstring{Robust codes for \( \k^+ \)-Borel sets}{Robust codes for kappa+-Borel sets}}

In this subsection we collect some generally known facts about coding Borel sets and absoluteness, and we formulate a variation on a result of L\"ucke \cite[Lemma~7.6]{luckeSigma11Definability2012} specialized to the case of a $\kappa^+$-Borel set.

Recall that a tree is a set of sequences closed under initial segments. 
 Given a set $A$, \( \mu \in \Card \), and a tree $T\subseteq \pre{<\mu}{A}$, we define the body of $T$ as
\[
[T]^\mu_A =   \{  x \in \pre{\mu}{A} \mid \forall \alpha < \mu \, ( x \restriction \alpha \in T) \}.
\]
When \( \mu \) and/or \( A \) are clear from the context, we suppress them from the notation and write \( [T] \) instead of \( [T]^\mu_A \).
In most cases, we will have either \( \mu=\kappa \) of \( \mu = \omega \). If \( T \) is a tree, then for every \( t \in T \) we denote by $\succ_T(t)$ the set of immediate successors of $t$ in $T$; when clear from the context, we may suppress the subscript $T$.

A tree $S \subseteq \pre{<\o}{A}$ is called well-founded if \( [S] = \emptyset \). 
Every well-founded tree \( S \) naturally carries a rank function, defined as $\rank_S(s)= 0 $ if $s \in S$ is a leaf (i.e.\ \( \succ_S(s) = \emptyset \)), and $\rank_S(s) = \sup\{\rank_S(t) + 1 \mid t \in \succ_S(s)\}$ otherwise. Denote by $S^\alpha$ (respectively: $S^{<\alpha}$, or $S^{> \alpha}$) the set of nodes $s \in S$ such that $\rank_S(s) = \alpha$ (respectively: $\rank_S(s) <\alpha$, or $\rank_S(s) > \alpha$). In particular, \( S^0 \) is the set of leaves of the tree \( S \).

 In classical descriptive set theory, Borel codes provide a way of interpreting a Borel set in different transitive models of set theory. The same is true in the generalized setting. 
Let \( X \subseteq \pre{\k}{\k} \), and let \( S \subseteq \pre{<\o}{A}\) be a well-founded tree with \( |A| \leq \k \). Given a function $f \colon S^0 \to \pre{<\kappa}{\kappa}$, we may assign to each $s \in S$ its \markdef{interpretation} $\mathcal{I}^f_{S,X}(s)$ as a $\kappa^+$-Borel subset of $X$ via
\[
\mathcal{I}^f_{S,X}(s) = 
\begin{cases}
    \clopen{f(s)} \cap X & \text{ if } s \in T^0 \\
    \displaystyle \bigcap_{t \in \succ_S(s)} X \minus \mathcal{I}^f_{S,X}(t) & \text{ otherwise.}
\end{cases}
\] 
The space \( X \) will always be clear from the context, and it will systematically be omitted from the notation; when reasonable, also the reference to the tree $S$ will be dropped, resulting in the lighter notations \( \mathcal{I}^f_S(s) \) or \( \mathcal{I}^f(s) \). Note that if $\rank_{S}(s) = \beta > 0$, then $\mathcal{I}^f_S(s) \in \Piii{0}{\b}{\k^+}(X)$, while if \( \rank_S(s) = 0 \), then $\mathcal{I}^f_S(s) \in \Deee{0}{1}{\k^+}(X)$. 
Conversely, if \( B \in \Piii{0}{\b}{\k^+}(X) \) for some \( 1 \leq \b < \k^+ \), then there is \( \langle S,f \rangle \) such that \( \rank_S(\emptyset) = \b \) and \( B = \mathcal{I}^f_S(\emptyset) \).

We call the pair $\langle S, f\rangle$ a \markdef{$\kappa^+$-Borel code} for the set $B = \mathcal{I}^f_S(\emptyset)$; in view of the above discussion, when \( \rank_S(\emptyset) = \b \) we also say that \( \langle S,f \rangle \) is a \markdef{\( \Piii{0}{\b}{\k^+} \) code} for \( B \). We would like to emphasize that although \( \k^+\)-Borel sets can generally be coded in many different ways, and these different codings are equivalent for most practical purposes, it is convenient for this and the next section that we work with this specific coding.

If \( \langle S,f \rangle \) is a \( \k^+ \)-Borel code for \( B = \mathcal{I}^f_S(\emptyset) \in \Bor{\k^+}(X)\) and \( V' \) is a transitive model of \( \mathsf{ZFC} \) containing \( \langle S,f \rangle \), then \( B \) can be re-interpreted in \( V' \) giving rise to the \( \k^+ \)-Borel set \( B^{V'} = (\mathcal{I}^f_S(\emptyset))^{V'} \). (The superscript \( V' \) will be dropped from the notation if clear from the context.) In particular, if \( B \in \Piii{0}{\b}{\k^+}(X) \) and we choose a \( \k^+ \)-Borel code \( \langle S,f \rangle \) for \( B \) such that \( \rank_S(\emptyset) = \b \), then \( B^{V'} \) is \( \Piii{0}{\b}{\k^+}(X) \) in \( V' \) as well. 
Moreover, by \( \Siii{1}{1}{\k} \)-absoluteness (Lemma~\ref{lemma:analyticabsoluteness}), the choice of the \( \k^+ \)-Borel code \( \langle S,f \rangle \) for \( B \) is irrelevant whenever \( V' \) is a generic extension by a \( {<}\k\)-closed forcing, so that in such a situation one can refer to \( B^{V'} \) without having to specify which \( \k^+ \)-Borel code is used to re-interpret \( B \) in \( V' \).

In this section we are mostly concerned with \( \k^+ \)-Borel sets that are \( \k \)-thin. A convenient way to control such property is to view \( \k^+\)-Borel sets as Suslin\(_\k\) sets, i.e.\ as projections of the bodies \( [T] \) of trees of the form \( T \subseteq \pre{< \k}{(\k \times W)}\), as in this case a close connection can be made between the \( \kappa \)-Perfect Set Property and fresh elements in ${<}\kappa$-closed forcing extensions. 
This connection is provided by Theorem~\ref{th: imperfect equiv}, which is due to Philipp L\"ucke \cite{luckeSigma11Definability2012} and proved in the restricted context of $\Siii{1}{1}{\k}$ (i.e.\ \( \k\)-analytic) sets. However, the proof given by L\"ucke easily extends to all Suslin$_\kappa$ sets, yielding the slight generalization stated here.
The proof ultimately relies on some classical arguments, for which we refer the reader to \cite[Lemma 3.4]{kunen_set_1983}.

To introduce some notation, suppose that $T \subseteq \pre{<\kappa}{(\kappa \times W)}$ is a tree, where $W$ is any set. Let $\pi[T]$ denote the projection of the body of $T$ onto its first coordinate, that is:
\[
    \pi[T] = \{x \in \pre{\kappa}{\kappa} \mid \exists y \in \pre{\kappa}{W} \, ( (x,y) \in [T] )\}.
\]
Likewise, let $\pi(s,w) = s$ for every pair $(s, w) \in T$; clearly, \( \pi(T) \subseteq \pre{<\k}{\k} \) is still a tree.

\begin{definition}
Given sets  $W$ and $X \subseteq \pre{\kappa}{\kappa}$, we say that $X$ is \markdef{$W$-Suslin$_\kappa$}
if there exists a tree $T \subseteq \pre{<\kappa}{(\kappa \times W)}$ such that $X = \pi[T]$.
\end{definition} 

We often do not care about the specifics of the second coordinate in the tree; in that case, we simply write Suslin\(_\kappa \) for sets that are $W$-Suslin\(_\kappa \) for some $W$.%
\footnote{Suslin$_\kappa$ sets are not to be confused with the established notion of a $\kappa$-Suslin set, i.e.\ the projection of a tree on $\pre{<\omega}{(\omega \times \kappa)}$.}    
The most prominent members of this class are the \( \kappa \)-analytic (or $\Siii{1}{1}{\k}$) sets, which are obtained setting $W = \kappa$. When we refer to a Suslin$_\kappa$ set, it is to be understood as referring to a fixed code representing that set.

Suslin\(_\k \) sets \( X = \pi[T] \) can again be naturally re-interpreted in any transitive model $V'$ of \( \mathsf{ZFC} \) containing $T$ as as $(\pi[T])^{V'}$; however, we should be careful to remind ourselves that this interpretation is far from robust, and without any absoluteness results it depends heavily on the chosen tree $T$.
Indeed, in the generalized setting we lack the rich array of absoluteness results we are accustomed to when studying the real line. For example, there is no hope for even $\Siii{1}{1}{\k}$-absoluteness to hold between any two transitive models of \( \mathsf{ZFC} \), since for a tree $T \subseteq \pre{<\kappa}{\kappa}$ the statement \enquote{$T$ has a branch of length $\kappa$} will in general not be absolute. On the bright side, the next lemma shows that at least for ${<}\kappa$-closed forcing extensions, $\Siii{1}{1}{\k}$-absoluteness and more generally, Suslin$_\kappa$-absoluteness, holds.

\begin{lemma} \label{lemma:analyticabsoluteness}
Let $\bP$ be a ${<}\kappa$-closed forcing notion, $G$ a $\bP$-generic filter, $W \in V$ and $T \subseteq \pre{<\kappa}{W}$ a tree. If $V[G] \models [T] \neq \emptyset$, then $V \models [T] \neq \emptyset$.
\end{lemma}

\begin{proof}
    See \cite[Lemma~2.7]{friedman_regularity_2016}.
\end{proof}

Recall that a map \( \varphi \) between two trees \( T \) and \( T' \) is \markdef{order preserving} if \( s \subseteq t \Rightarrow \varphi(s) \subseteq \varphi(t) \) for every \( s,t \in T \). Similarly, \( \varphi \) is \markdef{strict order preserving} if \( s \subsetneq t \Rightarrow \varphi(s) \subsetneq \varphi(t) \) for every \( s,t \in T \), and it \markdef{preserves incompatibility} if \( s \perp t \Rightarrow \varphi(s) \perp \varphi(t) \) for all \( s,t \in T \). If \( \varphi \) is strict order preserving and also preserves incompatibility, then it is called \markdef{order embedding}; equivalently, \( \varphi \) is an order embedding if it is injective and \( s \subsetneq t \IFF \varphi(s) \subsetneq \varphi(t) \), for all \( s,t \in T \). Similar terminology will also be used for \emph{partial} maps between \( T \) and \( T' \).

It $T, T'\subseteq \pre{<\mu}{A}$, to every strict order preserving function $\varphi \colon T \to T'$ we can associate a function $f_\varphi \colon [T]\to [T']$ by setting $f_\varphi(x)=\bigcup_{\alpha < \mu} \varphi(x \restriction \a)$, for every $x\in [T]$.
Notice that for every strict order preserving function $\varphi \colon T \to T'$ there is a \emph{continuous}%
\footnote{By continuous, we mean that if \( s \in T \) has limit length \( \delta \), then \( \varphi'(s) = \bigcup_{\a < \delta} \varphi(s \restriction \a) \).}
strict order preserving function $\varphi' \colon T \to T'$ such that $f_\varphi=f_{\varphi'}$; indeed, it is enough to define $\varphi'(s)=\bigcup_{t\subsetneq s} \varphi(t)$ for all $s\in T$ of limit length, and $\varphi'(s)=\varphi(s)$ for all other $s\in T$.
Also, if \( \varphi \) is an order embedding, then \( f_\varphi \) is a topological embedding, i.e.\ a homeomorphism onto its image.

Following L\"ucke, we introduce the following.

\begin{definition}
Let $T \subseteq \pre{<\kappa}{(\kappa \times W)}$ be a tree. We say $\varphi \colon \pre{<\kappa}{2} \to T$ is an \markdef{$\exists$-perfect order embedding} if
\begin{itemizenew}
\item 
\( \varphi \) is a strict order preserving map; 
\item 
\( \pi \circ \varphi \) preserves incompatibility. 
\end{itemizenew}
\end{definition}

\begin{lemma} \label{lem:char_exists-embeddings}
Let $T \subseteq \pre{<\kappa}{(\kappa \times W)}$ be a tree, and let $\varphi \colon \pre{<\kappa}{2} \to T$ be an order preserving function. The following are equivalent: 
\begin{enumerate-(1)}
\item \label{lem:char_exists-embeddings-1}
$\varphi$ is $\exists$-perfect order embedding;
\item \label{lem:char_exists-embeddings-2}
$\pi\circ \varphi$ is an order embedding of $\pre{<\k}{2}$ into $\pi(T)$;
\item \label{lem:char_exists-embeddings-3}
both $\varphi$ and $\pi \restriction \varphi(\pre{<\kappa}{2})$ are order embeddings.
\end{enumerate-(1)} 
\end{lemma}

\begin{proof} 
\ref{lem:char_exists-embeddings-1} \( \Rightarrow \) \ref{lem:char_exists-embeddings-2}. 
Obvious, since \( \pi \) is strict order preserving.

\ref{lem:char_exists-embeddings-2} \( \Rightarrow \) \ref{lem:char_exists-embeddings-3}.
Since $\pi\circ \varphi$ is, in particular, injective, so are $\varphi$ and $\pi\restriction \varphi(\pre{<\kappa}{2})$. 
Thus, since $\varphi$ is order preserving by assumption, it also preserves the strict order.
Now, $\pi$ is always order preserving,  so in order for $\pi\circ \varphi$ to preserve incompatibility, $\varphi$ needs to preserve incompatibility as well. 
Thus $\varphi$ is an order embedding. It follows that $\pi\restriction \varphi(\pre{<\kappa}{2})$ is an order embedding as well (the argument is similar).

\ref{lem:char_exists-embeddings-3} \( \Rightarrow \) \ref{lem:char_exists-embeddings-1}. 
If $\varphi$ and $\pi \restriction \varphi(\pre{<\kappa}{2})$ are both order embeddings, so is \( \pi \circ \varphi \). In particular, \( \varphi \) is strict order preserving, and \( \pi \circ \varphi \) preserves incompatibility. 
\end{proof}

\begin{lemma} \label{lem:order_embeddings_closed_images}
    Let $T,T'\subseteq \pre{<\k}{A}$ for some $A$. Let $\varphi \colon T\to T'$ be an order embedding and suppose that $T$ is ${<} \kappa$-splitting, i.e.\ that \( |\succ_T(s)| < \k \) for all \( s \in T \). Then $f_\varphi([T])$ is closed in \( \pre{\k}{A} \) (with respect to the bounded topology).
\end{lemma}

\begin{proof}
We need to prove that \( \pre{\k}{A} \setminus f_\varphi([T]) \) is open, so pick any
\( x \in \pre{\k}{A} \setminus f_\varphi([T])\).
Let \( S = \{ s \in T \mid \varphi(s) \subseteq x \} \). Then \( S \) is a chain because \( \varphi \) preserves incompatibility, therefore we can set \( \bar{s} = \bigcup S \). Since \( \varphi \) is strict order preserving, we need to have \( \leng{\bar{s}} < \k\), as otherwise \( f_\varphi(\bar s) = x \), against our choice for \( x \). We distinguish two cases.

If \( \bar s \notin T \), then for all \( y \in [T] \) there is \( \a < \leng{\bar s} \) such that \( y \restriction \a \perp \bar s \restriction \a \), therefore \( \clopen{x \restriction \beta} \cap f_\varphi([T]) = \emptyset \) for \( \beta = \sup \{ \leng{\varphi(s)} \mid s \in S \} \).

Suppose now that \( \bar s \in T \). For each \( t \in \succ_T(\bar{s}) \) there is \( \b_t < \k \) such that \( \varphi(t) \restriction \b_t \perp x \restriction \b_t \), by the definition of $\bar s$. Since \( |\succ_T(\bar s) |<\k\), by regularity of \( \k \) we have \( \b = \sup \{ \b_t \mid t \in \succ_T(\bar s) \} < \k \). Then \( \clopen{x \restriction \b} \cap f_\varphi([T]) \neq \emptyset \).
\end{proof}

\begin{corollary} \label{cor: exists psp implies strong psp}
Let $T \subseteq \pre{<\kappa}{(\kappa \times W)}$ be a tree. If there is an $\exists$-perfect order embedding  $\varphi \colon \pre{<\kappa}{2} \to T$, then $\pi[T]$ contains a \( \kappa \)-perfect subset.
\end{corollary}

\begin{proof}
By Lemma~\ref{lem:char_exists-embeddings}, the map \( \pi \circ \varphi \colon \pre{<\k}{2} \to \pi(T) \subseteq \pre{<\k}{\k}\) is an order embedding. Since \( \pre{<\k}{2}\) is obviously \( {<} \k \)-splitting, we get that \( P = f_{\pi \circ \varphi}(\pre{\k}{2}) \) is closed in \( \pre{\k}{\k}\) by Lemma~\ref{lem:order_embeddings_closed_images}. Moreover, \( f_{\pi \circ \varphi}\) witnesses that \( P \) is homeomorphic to \( \pre{\k}{2}\). Therefore it is enough to prove that the \( \k \)-perfect set \( P \) is contained in \( \pi[T] \). But this easily follows from \( f_\varphi(\pre{\k}{2}) \subseteq [T]\) and the fact that \( f_{\pi \circ \varphi} = f_\pi \circ f_{\varphi} \).
\end{proof}

As mentioned, L\"ucke observed that there is a tight connection between the possibility of adding new elements to a set of the form \( \pi[T] \) via a \( {<} \k \)-closed forcing notion, and the fact that there is an \( \exists \)-perfect order embedding into \( T \) (in which case \( \pi(T) \) is not \( \k \)-thin by Corollary~\ref{cor: exists psp implies strong psp}). 

\begin{theorem}[{{\cite[Lemma~7.6]{luckeSigma11Definability2012}}}] \label{th: imperfect equiv}
    Let $W$ be a set and $T \subseteq \pre{<\kappa}{(\kappa \times W)}$ a tree. Then the following are equivalent:
    \begin{enumerate-(1)}         \item Every ${<}\kappa$-closed forcing which adds a new element of \( \pre{\k}{2}\) adds a new element to $\pi[T]$. 
        \item There exists a ${<}\kappa$-closed forcing which adds a new element to $\pi[T]$. 
        \item There is an $\exists$-perfect order embedding into $T$.  
    \end{enumerate-(1)}
\end{theorem}

When restricting the attention to \( \k^+ \)-Borel sets, there is a tension between two ways of coding them. On the one hand, every \( \k^+\)-Borel set \( B \) is \( \k \)-analytic, and thus it is of the form \( B = \pi[T]\) for some tree \( T \subseteq \pre{<\k}{(\k \times \k)}\); this point of view allows us to exploit Theorem~\ref{th: imperfect equiv} in order to determine whether \( B \) contains a \( \k \)-perfect subset. However,
this way of coding \( \k^+ \)-Borel sets is not very stable. Without a large enough fragment of absoluteness, we cannot even ensure that 
 for two trees $T, T' \subseteq \pre{<\k}{(\k \times \k)}$ with $\pi[T] = \pi[T']$, the same equality continue to hold in ${<}\kappa$-closed forcing extensions (because it is a $\Piii{1}{2}{\k}$ statement).
On the other hand, $\Siii{1}{1}{\k}$-absoluteness (Lemma~\ref{lemma:analyticabsoluteness}) guarantees that the statement \enquote{$\langle S_1, f_1 \rangle$ and $\langle S_2, f_2\rangle$ code the same $\kappa^+$-Borel set} is absolute for ${<}\kappa$-closed forcing extensions,
so it is worth stating Theorem~\ref{th: imperfect equiv} in a form that fits the setting of \( \k^+\)-Borel codes (Corollary~\ref{cor: thin borel set eq}).
      To this aim,
one can exploit the following construction, which is borrowed from
 L\"ucke and Schlicht's proof of the following result (see~\cite[Lemma~1.11]{lucke_continuous_2015}). 
For the reader's convenience, we recount the proof.

\begin{lemma} \label{lem: borel set has canonical tree}
For every (code for a) $\kappa^+$-Borel set $B \subseteq \pre{\kappa}{\kappa}$ there exists a tree $T \subseteq \pre{<\kappa}{(\kappa \times \kappa)}$ such that \( B = \pi[T] \), and moreover
\[
B^{V[G]} = (\pi[T])^{V[G]}
\]
holds in all ${<}\kappa$-closed forcing extensions \( V[G] \).
\end{lemma}

\begin{proof}
Let \( \langle S,f \rangle \) be the chosen code for \( B \),
and fix an enumeration $e \colon S \to \kappa$ such that $s \subsetneq t \implies e(s) < e(t)$. Let us call a pair $(y,z) \in \pre{S}{2} \times \pre{S}{S}$ \markdef{correct} if the following conditions hold:
\begin{enumerate-(1)}
\item \label{borel set has canonical tree-1}
 $y(\emptyset) = 1$, and for each $s \in S^{>0} $ we have $y(s) = 0 \IFF \exists t \in \succ_S(s) \, (y(t) = 1)$.
\item \label{borel set has canonical tree-2}
 for each $s \in s \in S^{>0} $ with $y(s) = 0$, we have $z(s) \in \succ_S(s)$ and $y(z(s)) = 1$; otherwise, $z(s) = \emptyset$.%
\footnote{It bears mentioning that this diverges slightly from L\" ucke and Schlicht's proof, which even implies that the projection $\pi \restriction [T]$ can be chosen to be injective. However, this is not relevant for our arguments.}
\end{enumerate-(1)}

Consider now the set
\begin{equation}
C = \left\{ (x,y,z) \in \pre{\kappa}{\kappa} \times \pre{S}{2} \times \pre{S}{S}\ \Big|\
\begin{gathered}
(y,z) \text{ is correct } \wedge  \\
\forall s \in S^0 \left( y(s) = 1 \IFF x \in \clopen{f(s)}  \right) 
\end{gathered}
\right\}. \label{eq: def C}
\end{equation}

If we identify $\pre{S}{2}$ and $\pre{S}{S}$ with $\pre{\kappa}{2}$ and $\pre{\kappa}{\kappa}$, respectively, via the bijection $e$, one can construe $C$ as a closed subset of $\pre{\kappa}{(\kappa \times (2 \times \kappa))}$, and hence also as a closed subset of \( \pre{\k}{(\k \times \k)} \). 
Indeed, condition~\ref{borel set has canonical tree-2} is clearly closed. As for~\ref{borel set has canonical tree-1}, as stated it seems more complicated because of the implication \( y(s) = 0 \Rightarrow \exists t \in \succ_S(s) \, (y(t) = 1) \): however, it is not hard to see that it is closed \emph{relatively to the closed set determined by~\ref{borel set has canonical tree-2}} because \( z \) can be used to extract a witness of the existential quantification appearing in the problematic implication. 
     
The first coordinate in a correct pair \( (y,z) \) models the interpretation function \( \mathcal{I}^f_S \) associated to the $\kappa^+$-Borel code \( \langle S,f \rangle \), hence $B = \pi[T]$ where \( T \subseteq \pre{<\k}{(\k \times \k)} \) is a tree such that \( C = [T] \). 
Moreover, the equality defining $C$ in~\eqref{eq: def C} is a \( \Piii{1}{1}{\k} \) statement, therefore \( C^{V[G]} = [T]^{V[G]}\) is still the set of triples such that \( (y,z) \) is correct and \( \forall s \in S^0 \left( y(s) = 1 \IFF x \in \clopen{f(s)}  \right) \) (in the sense of \( V[G] \)). But this entails \( B^{V[G]} = (\pi[T])^{V[G]} \), as desired. 
  \end{proof}

The codes for \( \k^+ \)-Borel sets can also be used to canonically re-interpret \( \k^+ \)-Borel measurable functions in \( {<}\k \)-closed forcing extentions. Indeed, suppose that \( T \subseteq \pre{<\k}{\k} \) is a tree, and that \( f \colon [T] \to \pre{\k}{\k}\) is \( \k^+ \)-Borel measurable. Suppose that \( \bP \) is a \( {<}\k\)-closed forcing notion, and that \( G \) is \( \bP \)-generic over \( V \). For each \( t \in \pre{<\k}{\k}\) fix a \( \k^+ \)-Borel code \( \langle S_t,f_t \rangle \) for \( B_t= f^{-1}(\clopen{t}) \), so that the latter can be re-interpreted in \( V[G] \) as \( B_t^{V[G]} = (\mathcal{I}^{f_t}_{S_t}(\emptyset))^{V[G]}\). (By \( \Siii{1}{1}{\k}\)-absoluteness, the actual choice of such codes is irrelevant.) Working in \( V[G] \), we can then canonically define a function \( f^{V[G]} \colon ([T])^{V[G]} \to  (\pre{\k}{\k})^{V[G]} \) by setting
for every \( x \in ([T])^{V[G]}\)
\[
f^{V[G]}(x) = \bigcup\left \{t \in \pre{<\kappa}{\kappa}  \bigm | x \in \left(B_t\right)^{V[G]} \right \}.
\]
Since the statement (the quantifier $\exists!$ means \enquote{there exists exactly one})
\[
 \forall x \in [T]\ \forall i < \kappa\ \exists! t \in \pre{i}{\kappa} \, (x \in B_t)
\]
is $\Piii{1}{1}{\k}$, hence absolute, $f^{V[G]}$ is a well-defined function. 
 We also note that injectivity of $f$ is $\Piii{1}{1}{\k}$.

From previously established facts, we formulate the following observation (cf.\ also \cite[Proposition~7.8]{luckeSigma11Definability2012}, \cite[Proposition~2]{FriedmanSilverIsConsistentMR3257956}, and \cite{SzirakiVaananenMR3661728}).

\begin{corollary} \label{cor: thin borel set eq}
For every (code for a) \( \k^+ \)-Borel set $B \subseteq \pre{\kappa}{\kappa}$, the following are equivalent:
\begin{enumerate-(1)} \item \label{enum: every new2}
Every ${<}\kappa$-closed forcing which adds a new element of \( \pre{\k}{2}\) adds a new element to $B$. 
\item \label{enum: some new2}
There exists a ${<}\kappa$-closed forcing which adds a new element to $B$. 
\item \label{enum: strong psp2}
$B$ contains a \( \kappa \)-perfect subset. 
\item \label{enum: weak psp2}
There is $\kappa^+$-Borel injection of $\ \pre{\kappa}{2}$ into $B$. 
\end{enumerate-(1)}
\end{corollary}

\begin{proof}
The implications \ref{enum: every new2}~$\Rightarrow$~\ref{enum: some new2} and \ref{enum: strong psp2}~$\Rightarrow$~\ref{enum: weak psp2} are trivial.
For the implication \ref{enum: some new2}~$\Rightarrow$~\ref{enum: strong psp2}, apply Theorem~\ref{th: imperfect equiv} to the tree $T$ constructed in Lemma~\ref{lem: borel set has canonical tree} to find an $\exists$-perfect embedding into $T$. By Corollary~\ref{cor: exists psp implies strong psp}, this implies \ref{enum: strong psp2}.

It remains to show \ref{enum: weak psp2}~$\Rightarrow$~\ref{enum: every new2}. Let \( f \colon \pre{\k}{2} \to B \) be a \( \k^+ \)-Borel injection.
Suppose that \( V[G] \) is a forcing extension obtained via a \( {<}\k\)-closed forcing, and that there is \( \bar x \in (\pre{\k}{2})^{V[G]} \setminus (\pre{\k}{2})^V \). Since the formula $\forall x \in \pre{\kappa}{2} \, ( f(x) \in B)$ is \( \Piii{1}{1}{\k} \), by \( \Siii{1}{1}{\k} \)-absoluteness we have \( f^{V[G]}(\bar x) \in B^{V[G]} \). Likewise, \( f^{V[G]} \) is injective because \( f \) is, therefore \( f^{V[G]}(\bar x) \notin \ran(f) \). Let \( \bar y \in (\pre{\k}{\k})^V \) be any element of \( B \setminus \ran(f) \). 
The statement \enquote{$\forall x \in \pre{\kappa}{2} \, ( f(x) \neq \bar y$)} is $\Piii{1}{1}{\k}$, hence we have \( f^{V[G]}(\bar x) \neq \bar y \). Summing up, \( f^{V[G]}(\bar x) \in B^{V[G]} \setminus B^V\), as desired.
\end{proof}
 
From Corollary~\ref{cor: thin borel set eq}, the construction of a closed set $X \subseteq \pre{\kappa}{\kappa}$ with $|X| > \kappa$ and $\ord_{\k^+}(X) = 2$ can be achieved by a straightforward generalization of standard arguments involving an iteration of Solovay's almost disjoint forcing \cite{jensenApplicationsAlmostDisjoint1970}. A generalization of almost disjoint forcing for uncountable $\kappa$ has appeared in \cite{luckeSigma11Definability2012}, essentially providing an answer to Problem~\ref{q: large and ord 2} among other related results surrounding $\Siii{1}{1}{\k}$- and $\Deee{1}{1}{\k}$-definability.

\begin{corollary}\label{cor:closed_uncountable_order_2}
Consistently, there exists a closed subset $X \subseteq \pre{\kappa}{\kappa}$ with $|X| > \kappa$ and $\ord_{\k^+}(X) = 2$.
\end{corollary}

We do not provide an explicit proof of Corollary~\ref{cor:closed_uncountable_order_2} at this point, since the next section will develop a more general approach subsuming the use of almost disjoint forcing.

\subsection{\texorpdfstring{A \( \kappa \)-thin closed set of order $\kappa^+$}{A kappa-thin closed set of order kappa+}} \label{subsec: thin closed set}

The rest of this section is devoted to the existence of \( \kappa \)-thin definable sets. Note that by the results of Schlicht \cite{schlicht_perfect_2017}, it is consistent relative to an inaccessible above $\kappa$ that every set $X \subseteq \pre{\kappa}{\kappa}$ definable from a $\kappa$-sequence of ordinals has the \( \kappa \)-Perfect Set Property. 

Recall that a tree $T \subseteq \pre{<\kappa}{\kappa}$ is called \markdef{Kurepa}, if $|[T]|> \kappa$ and $|\mathcal{L}_\alpha(T)| \leq |\alpha|$ for stationary many $\alpha < \kappa$, where $\mathcal{L}_\alpha(T) = T \cap \pre{\alpha}{\kappa}$ denotes the $\alpha$-th level of the tree;
it is called \markdef{Jech-Kunen}, if $\kappa < |[T]| < 2^\kappa$.

From the analysis of Kurepa trees on $\omega_1$ it is known that the bodies of Kurepa and Jech-Kunen trees are typical examples of \( \kappa\)-thin closed sets \cite[Chapter~VIII,~Lemma~3.4]{kunen_set_1983}. The same is true in the generalized setting. For Jech-Kunen trees this is just a simple matter of cardinality; for Kurepa trees, it follows by an argument in \cite[Section 7]{luckeSigma11Definability2012}.

By a result of Hamkins, it is possible to turn the ground model generalized Baire space into a \( \kappa \)-thin closed set.
Notice that \( \k \)-thin sets constructed in this way differ from the ones discussed above, as in the situation described in Theorem~\ref{th: hamkins} the closed set $\pre{\kappa}{\kappa} \cap V$ is not the body of a Kurepa or a Jech-Kunen tree. 

\begin{theorem}[\cite{hamkins_gap_2001}] \label{th: hamkins}
Let $\bP$ be a forcing notion such that $|\bP| < \kappa$. Then in any forcing extension \( V[G] \) by $\bP$, the set $\pre{\kappa}{\kappa} \cap V$ is a \( \kappa \)-thin closed set.
\end{theorem}

\begin{proof}
Work in \( V[G] \). 
By \cite[Key Lemma]{hamkins_gap_2001}, the tree $T = (\pre{<\k}{\k})^V = \pre{<\kappa}{\kappa} \cap V$ does not gain any new branches in the extension of \( V[G] \) by $\bQ$, where $\bQ$ is such that $V[G] \models {}$``\( \bQ \) is a ${<}\kappa$-closed forcing notion''. By Theorem \ref{th: imperfect equiv}, this means that \( \pre{\kappa}{\kappa} \cap V =[T]\) does not contain any \( \kappa \)-perfect subset.
\end{proof}

We conclude the section by giving a general method for constructing \( \kappa \)-thin closed sets $X \subseteq \pre{\kappa}{\kappa}$ on which the $\kappa^+$-Borel hierarchy does not collapse. This construction follows \cite[Theorem~14.3]{miller_descriptive_1995}.

The following auxiliary lemma is a straightforward generalization of \cite[Lemma~2]{kunen_borel_1983}.

\begin{lemma} \label{lem: kp is pi alpha}
Let $X \subseteq \pre{\kappa}{\kappa}$, $\bP$ be a forcing notion such that $|\bP| \leq \kappa$, and $\name{J}$ be a $\bP$-name for a set in $\Piii{0}{\alpha}{\k^+}(X)$, for some \( 1 \leq \a < \k^+ \). Then for any $p \in \bP$
\[
K_p(\name{J}) = \{x \in X \mid p \Vdash \check{x} \in \name{J}\}
\]
is $\Piii{0}{\alpha}{\k^+}(X)$ in the ground model.
\end{lemma}

\begin{proof}
We proceed by induction on $\alpha$. For $\alpha = 1$, suppose that $\seq{x_i}{i < \kappa}$ is a sequence of points in $K_p(\name{J})$ (for some $p \in \bP$) which converges to $x \in X$. By assumption, $\name{J}$ is a name for a closed set, so $p \Vdash \check{x} \in \name{J}$, hence \( x \in K_p(\name{J})\). This shows that $K_p(\name{J})$ is closed.

In the inductive step \( \a > 1 \), assume the statement is true for all $1 \leq \beta < \alpha$, and that $\name{J} = \bigcap_{i < \kappa} \name{J}_i$, where $\name{J}_{i}$ is a name for a set in $\Siii{0}{\beta_i}{\k^+}(X)$, for $1 \leq \beta_i < \alpha$. 
For any $p \in \bP$ and $x \in X$, we know that
\[
p \Vdash \check{x} \in \name{J} \IFF \forall i < \kappa\ \forall q \leq p \left( q \not \Vdash \check{x} \not \in \name{J}_i \right).
\]
By inductive hypothesis, the set 
\[
\{x \in X \mid q \not \Vdash \check{x} \notin \name{J}_i\} = X \minus K_q(X \minus \name{J}_i)
\]
belongs to $\Siii{0}{\beta_i}{\k^+}(X)$, and since we assumed $|\bP| \leq \kappa$, the set $K_p(\name{J})$ can be written as an intersection of at most $\kappa$-many sets from $\bigcup_{1 \leq \beta < \alpha} \Siii{0}{\beta}{\k^+}(X)$, hence it is $\Piii{0}{\alpha}{\k^+}(X)$.
\end{proof}

\begin{lemma} \label{lem: small forcing preserves order}
Let $X \subseteq \pre{\kappa}{\kappa}$ be such that $\ord_{\k^+}(X)\geq \alpha$ (for some \( 1 \leq \a < \k^+ \)), and let $\bP$ be a forcing notion such that $|\bP| \leq \kappa$. Then $\bP \Vdash \ord_{\k^+}(\check{X}) \geq  \alpha$. 

If moreover $\ord_{\k^+}(X) = \alpha$, then $\bP \Vdash \ord_{\k^+}(\check{X}) =  \alpha$.
\end{lemma}

\begin{proof}
By Proposition~\ref{prop:proper_for_k+}, in the ground model there is a set \( A \in \Piii{0}{\a}{\k^+}(X) \setminus \bigcup_{1 \leq \b < \a} \Piii{0}{\b}{\k^+}(X) \) because we assumed \( \ord_{\k^+}(X) \geq \a \). Towards a contradiction, suppose that there are \( 1 \leq \b < \a \), a name $\name{J}$ for a $\Piii{0}{\beta}{\k^+}$ subset of \( X \), and \( p \in \bP \) such that $p \Vdash \check{A} = \name{J}$. By absoluteness, we know that
\[
\{x \in X \mid p \Vdash \check{x} \in \check{A}\} = A.
\]
But the set on the left equals \( K_p(\name{J}) \) by choice of \( p \), and it belongs to $\Piii{0}{\beta}{\k^+}(X)$ by Lemma~\ref{lem: kp is pi alpha}, contradicting the choice of \( A \).

Assume now that \( \ord_{\k^+}(X) = \a \), and let $\name{J}$ be any name for a $\kappa^+$-Borel subset of $X$. Then for any $\bP$-generic filter $G$, we have that 
\[
V[G] \models \name{J} = \bigcup_{p \in G} \check{K}_p(\name{J}).
\]
Each \( K_p(\name{J}) \) is \( \k^+ \)-Borel in the ground model because of Lemma~\ref{lem: kp is pi alpha}, and hence it belongs to $\Siii{0}{\alpha}{\k^+}(X)$ because we assumed $\ord_{\k^+}(X) = \alpha$.
Therefore the interpretation of $\name{J}$ in \( V[G]\) is a $\Siii{0}{\alpha}{\k^+}$ subset of $X$. Since this is true for all $\name{J}$ and $G$ as above, we have $\bP \Vdash \ord_{\k^+}(\check{X}) \leq \alpha$, which together with the first part of the lemma yields \( \bP \Vdash \ord_{\k^+}(\check{X}) = \a \).
\end{proof}

The most interesting forcing Lemma \ref{lem: small forcing preserves order} applies to is $\kappa$-Cohen forcing, which is the only forcing notion of size at most $\kappa$ that is simultaneously ${<}\kappa$-closed.

Finally, combining the results obtained so far we obtain our examples of \( \k \)-thin closed sets with non-collapsing \( \k^+ \)-Borel hierarchy.

\begin{corollary} \label{cor: closed set with full order}
Consistently, there exists a \( \kappa \)-thin closed subset $X \subseteq \pre{\kappa}{\kappa}$ such that the $\kappa^+$-Borel hierarchy does not collapse on $X$.
\end{corollary}

\begin{proof}
Let \( V[G] \) be a forcing extension by a forcing of size smaller than \( \k \). By Theorem~\ref{th: hamkins}, the closed set \( X = \pre{\k}{\k} \cap V \) is \( \k \)-thin in \( V[G] \). Moreover, in the ground model we have that \( \ord_{\k^+}(\pre{\k}{\k}) \geq \a \) for every \( 1 \leq \a < \k^+ \), hence for any such \( \a \) we have
\[
V[G] \models  \ord_{\k^+}(\pre{\k}{\k} \cap V) \geq \a 
\] 
by Lemma~\ref{lem: small forcing preserves order}, that is, the \( \k^+ \)-Borel hierarchy on \( X = \pre{\k}{\k} \cap V \) does not collapse.
\end{proof}

\section{\texorpdfstring{Variants of Miller's $\alpha$-forcing}{Variants of Miller's alpha-forcing}} \label{sec: alpha-forcing}

Throughout this section, fix a set $X \subseteq \pre{\kappa}{\kappa}$. All $\kappa^+$-Borel sets appearing are to be understood as relatively $\kappa^+$-Borel subsets of $X$. Our goal will be to modify $\ord_{\k^+}(X)$ in a suitable forcing extension to prove Theorem~\ref{th: set order to n} and Corollary~\ref{cor:final_7}. This will be done by providing a partial generalization of Miller's $\alpha$-forcing to the case of an uncountable $\kappa = \kappa^{<\kappa}$ and closely following his ideas of using rank functions on iterations of the forcing to obtain lower bounds for the value of $\ord_{\k^+}(X)$ in the generic extension. 
In particular, the forcing $\aforc_\alpha(A,B,X)$ introduced in Definition~\ref{def: alpha forcing} is modeled on the forcing used by Miller in \cite[Section~1]{miller_length_1979}. In Section~\ref{subsec:preservation_cardinals}, we show that it has the appropriate $\kappa$-closure properties to preserve cardinals, in analogy with the original forcing, which is c.c.c. \cite[Lemma~7]{miller_length_1979}. The use of $\alpha$-forcing, rank functions, the bookkeeping argument that codes all relevant Borel subsets along an iteration, and the crucial Theorem~\ref{th: switcheroo} are adaptations of the ideas used by Miller in the proof of \cite[Theorem~34]{miller_length_1979}.

For each $\alpha < \omega$,
let $T_\alpha = \pre{\leq \alpha}{\kappa} = \pre{< \a}{\k} \cup \pre{\a}{\k}$ be the canonical well-founded tree with \( \rank_{T_\a}(\emptyset) = \a \). The tree \( T_\a \) shall serve as a template for the \( \k^+ \)-Borel code of a generic $\Piii{0}{\alpha}{\k^+}$ set:
  indeed, it is easy to see that a set $B \subseteq X$ is $\Piii{0}{\a}{\k^+}$ if and only if there is some $f \colon T^0_\alpha \to \pre{<\kappa}{\kappa}$ \footnote{Recall that for a well-founded tree $T$, $T^0$ denotes the set of leaves of $T$, while \( T^{>0} = T \setminus T^0\).} 
such that $\langle T_\alpha, f\rangle$ codes $B$, i.e. $B = \mathcal{I}^f_{T_\alpha}(\emptyset) \cap X$.

In this section, we restrict our efforts only to finite ordinals, as for limit ordinals \( \a \) it is necessary to first find an appropriate definition for $T_\a$. Although this is not an issue when \( \cof(\a) = \k \), the advent of limit ordinals of cofinality less than $\kappa$ introduces nontrivial obstacles that are not present in the classical case. 
This explains why Miller was able to develop analogues of the results of this section in full generality for all ordinals $\alpha < \omega_1$ (see e.g.~\cite[Lemma~35~and~Lemma~36]{miller_length_1979}), while here we stop already at level \( \o \), which has cofinality strictly smaller than \( \k \). A method for working with ordinals $\alpha \geq \omega$ and nodes whose rank is a limit ordinal with small cofinality is in preparation and will appear in future work by the second author.

We are now equipped to define (a generalization of) Miller's $\alpha$-forcing. 

\begin{definition} \label{def: alpha forcing}
For $A, B$ disjoint subsets of $X$ and a natural number \( 1 < \a < \o \),
let $\aforc_\alpha(A, B, X)$ be the partially ordered set consisting of all pairs $p = \langle f_p , R_p \rangle$ such that: 
 \begin{enumerate}[label = {\upshape (\alph*)},leftmargin=2pc]
\item\label{def: alpha forcing-a} 
$f_p \colon T_\alpha^0 \to \pre{<\kappa}{\kappa}$ is a partial function with $|f_p| < \kappa$;
\item\label{def: alpha forcing-b} 
$R_p \subseteq T_\alpha^{>0} \times X$ and $|R_p| < \kappa$;
\item\label{def: alpha forcing-c} 
If $\langle t, x\rangle \in R_p$, then for all $t' \in \succ_{T_\a}(t)$ we have
\begin{itemizenew}
\item 
$\langle t', x\rangle \notin R_p$ if $t' \in T_\alpha^{>0}$;
\item 
$x \notin \clopen{f_p(t')}$ if $t' \in T_\alpha^0 \cap \dom(f_p)$.
\end{itemizenew}
\end{enumerate}
We also have two constraints involving the parameters $A$ and $B$:
\begin{enumerate}[resume*]
\item\label{def: alpha forcing-e} 
$\{x \in X \mid\exists t \in \succ_{T_\alpha}(\emptyset) \, (\langle t,x \rangle \in R_p)\} \cap A = \emptyset$;
\item\label{def: alpha forcing-d} 
$\{x \in X \mid\langle \emptyset, x \rangle \in R_p\} \cap B = \emptyset$.
\end{enumerate}
The ordering is given by $q \leq p$ if and only if $f_p \subseteq f_q$ and $R_p \subseteq R_q$.
\end{definition}

Since the space $X$ is fixed, we will often omit it from the notation and write \( \aforc_\a(A,B) \) instead of \( \aforc_\a(A,B,X) \).
Our main focus are going to be the forcings $\aforc_\alpha(\emptyset, \emptyset)$ and $\aforc_\alpha(A, X \minus A)$. These two kinds of forcings are fundamentally of a different flavor: while the first one adds a generic $\Piii{0}{\alpha}{\k^+}$ subset to the space $X$ (see Theorem \ref{th: switcheroo} for a formalization of the word \enquote{generic} in this instance), the latter will instead add a $\Piii{0}{\a}{\k^+}$ code for the set $A\subseteq X$ (see Corollary \ref{cor: collapse A}). 

Until further notice, from now on we fix a finite ordinal \( 1 < \a < \o \) and two disjoint sets \( A,B \subseteq X \).
If $G$ is $\aforc_\alpha(A,B)$-generic, then by a density argument $f_G = \bigcup_{p \in G} f_p$ is a total function from $T_\alpha^0$ to 
$\pre{<\kappa}{\kappa}$, thus one can interpret every $t \in T_\alpha$ as a $\kappa^+$-Borel subset $G_t = \mathcal{I}^{f_G}(t)$ of $X$. The semantic meaning of a pair $\langle t, x\rangle \in R_p$ is to ensure that $x$ is an element of $G_t$. The next lemmas show that this is sound.

\begin{lemma} \label{lem: D dense}
For $t \in T^{>0}_\alpha$ and $x\in X$, let \( D_{t,x} \) be defined by
\[
D_{t, x} =
\{p \in \aforc_\alpha(A,B) \mid \langle t , x \rangle \in R_p \vee \exists t' \in \succ_{T_\a}(t) \cap  \dom(f_p) \, ( x \in \clopen{f_p(t')})\}
\]
if \( \rank_{T_\a}(t)=1 \), and
\[
D_{t,x} = \{p \in \aforc_\alpha(A,B) \mid \langle t , x \rangle \in R_p \vee \exists t' \in \succ_{T_\a}(t) \, (\langle t' , x \rangle  \in R_p)\}
\]
if \( \rank_{T_\a}(t)>1 \).
Then \( D_{t,x} \) is dense in $\aforc_\alpha(A,B)$.
\end{lemma}

\begin{proof}
Given $p\in \aforc_\alpha(A,B)$, we want to find $p' \leq p$ such that $p' \in D_{t, x}$.  
If $p \in D_{t, x}$ we are done, so assume $p \notin D_{t, x}$ (so that, in particular, $\langle t, x\rangle \notin R_p$).
Also, if $x\in A$ and $t=\emptyset$, then $p' = \langle f_p, R_p\cup \{ \langle \emptyset, x\rangle \} \rangle $
is a condition in $\aforc_\alpha(A,B)$ satisfying $p' \leq p$ and $p' \in D_{t, x}$. Thus, we can assume that $x\notin A$ or $t \neq \emptyset$.
    
Since $\kappa$ is regular, $|f_p|< \kappa$, $|R_p| < \kappa$, and $t \in T^{>0}_\a$ has $\kappa$-many successors, we can find a $t' \in \succ(t)$ such that $(\{t'\} \cup \succ_{T_\a}(t')) \cap (\dom(f_p) \cup \dom(R_p)) = \emptyset$ and an $s\in \pre{<\kappa}{\kappa}$ such that $x\in \clopen{s}$ and $\clopen{s}\cap \{y\mid  \langle t, y \rangle\in R_p\}=\emptyset$. But this means that 
\[
p' = 
\begin{cases}
\langle f_p \cup \{ \langle t', s\rangle \}, R_p \rangle
& \text{ if } \rank_{T_\a}(t)=1,\\
\langle f_p, R_p \cup \{ \langle t', x\rangle \} \rangle
& \text{ if }\rank_{T_\a}(t)>1
\end{cases}
\]
is a condition in $\aforc_\alpha(A,B)$ satisfying $p' \leq p$ and $p' \in D_{t, x}$.
\end{proof}

\begin{lemma} \label{lem: R interprets correctly}
Let $G$ be $\aforc_\alpha(A,B)$-generic, and set $f_G = \bigcup_{p \in G} f_p $ and $R_G = \bigcup_{p \in G} R_p$. Then for each $t \in T_\alpha^{>0}$ and $x \in X$, we have that 
\[ 
V[G] \models x \in G_t \IFF \langle t, x \rangle \in R_G .
\]
\end{lemma}

\begin{proof}
We proceed by (finite) induction on \( \rank_{T_\a}(t) \).
First notice that, by definition, for every leaf $t' \in T_\alpha^0$ we have \( x \in G_{t'} \IFF x \in \clopen{f_G(t')} \). 

We begin with the basic case $\rank_{T_\a}(t) = 1$, in which case all \( t' \in \succ_{T_\a}(t) \) are in $T^0_\alpha$. If $x \in G_t$, then by the definition of the interpretation function \( \mathcal{I}^{f_G} \) we have that $x \notin G_{t'} = \clopen{f_G(t')}$ for every $t' \in \succ_{T_\a}(t)$. Since $D_{t, x}$ is dense by Lemma~\ref{lem: D dense}, then $G$ must meet $D_{t, x}$, and we can conclude $\langle t, x\rangle \in R_G$.
On the other hand, if $\langle t,x\rangle \in R_G$, then for each $p \in G$ and $t' \in \succ_{T_\a}(t) \cap \dom(f_p)$ we have $x \notin \clopen{f_p(t')}$, and therefore $x \in G_t = \bigcap_{t' \in \succ_{T_\a}(t)} X \minus \clopen{f_G(t')}$.

Suppose now that \( \rank_{T_\a}(t) > 1 \), and that the statement of the lemma holds for all $t' \in \succ_{T_\a}(t)$. Since the set \( D_{t,x} \)
 from Lemma~\ref{lem: D dense} is dense, we have 
\[
\langle t, x \rangle \in R_G \IFF \forall t' \in \succ_{T_\a}(t) \, ( \langle t', x \rangle \notin R_G).
\]
By inductive hypothesis, this means that
\begin{align*}
\langle t, x \rangle \in R_G & \IFF \forall t' \in \succ_{T_\a}(t) \, ( \langle t', x \rangle \notin R_G ) \\
& \IFF \forall t' \in \succ_{T_\a}(t) \, (x \notin G_{t'} ) \\
& \IFF x \in G_t  . \qedhere
\end{align*}
\end{proof}

By Lemma~\ref{lem: R interprets correctly} and conditions~\ref{def: alpha forcing-e} and~\ref{def: alpha forcing-d} in Definition~\ref{def: alpha forcing} (which take care of the inclusions \( A \subseteq G_\emptyset \) and \( G_\emptyset \subseteq X \setminus B \), respectively) we get:

\begin{corollary} \label{cor: collapse A}
Let $G$ be $\aforc_\alpha(A,B)$-generic. Then in the forcing extension \( V[G] \) the set \( G_\emptyset \) separates \( A \) from \( B \), i.e.\ $A \subseteq G_\emptyset \subseteq X \minus B$.
\end{corollary}

In other words, forcing with $\aforc_\alpha(A, X \minus A)$ adds a $\Piii{0}{\a}{\k^+}$-code \( \langle T_\a, f_G \rangle \) for the set $A$ as a subset of $X$, independently of its complexity in the ground model \( V \).

\subsection{Preservation of cardinals}\label{subsec:preservation_cardinals}

In the classical case, we are able to iterate any c.c.c.\ forcing without collapsing cardinals. In the generalized setting, preservation theorems are of a much more limited character, often requiring additional regularity assumptions. Fortunately, our forcing is structurally simple enough that preservation of cardinals is not a difficult issue. Recall Definition~\ref{def:forcinproperties}.

\begin{lemma}\label{lem: <kappa closure}
The forcing $\aforc_\alpha(A,B)$ is ${<}\kappa$-closed and well-met.       \end{lemma}

\begin{proof}
Let $\seq{p_i}{i < \delta}$ for $\delta < \kappa$ be a decreasing sequence. It is easy to see that $q = \langle f_q, R_q \rangle$ with $f_q = \bigcup_{i<\delta} f_{p_i}$ and $R_q = \bigcup_{i < \delta} R_{p_i}$ is the greatest lower bound of the sequence. Likewise, for two compatible conditions $p,q$ the condition $r = \langle f_p \cup f_q, R_p \cup R_q \rangle$ is their greatest lower bound.
\end{proof}

To each pair of sets \( Z \) and \( W \) we can associate the forcing \( \mathrm{Fn}(Z,W,{<}\k) \) consisting of all partial functions from \( Z \) to \( W \) of size smaller than \( \k \), ordered by inclusion. 

\begin{lemma} \label{lem: kappa linked}
The forcing $\aforc_\alpha(A,B)$ is $\kappa$-linked.
\end{lemma}

\begin{proof}
We prove this by reducing our forcing to a $\kappa$-Cohen forcing.
For each $x \in X$ and $p \in \aforc_\alpha(A,B)$, let $h_{x,p} = \{t \in T_\alpha^{>0} \mid \langle t, x \rangle \in R_p\}$, and notice that \( |h_{x,p}| < \kappa\). Let $X_p=\{x\in X\mid h_{x,p}\neq\emptyset\}$. Let us also fix a bijection $\psi \colon \powset_{<\k}( T_\alpha^{>0})\to \kappa$, where \( \powset_{< \k}( T_\alpha^{>0}) \) denotes the collection of all subsets of $T_\alpha^{>0}$ of size smaller than \( \k \), and define $g_p \colon X_p \to \kappa $ by setting $g_p(x)=\psi(h_{x,p})$ for any $p \in \aforc_\alpha(A,B)$.
Then, the map $\phi(p) = g_p$ assigns to each $p \in \aforc_\alpha(A,B,X)$ a condition in 
\( \mathrm{Fn}(X, \kappa, {<}\kappa) \).
    
Given $p = \langle f_p, R_p\rangle$ and $ q = \langle f_q, R_q \rangle$, let 
\[
p \parallel^* q \IFF \langle \emptyset, R_p\rangle \parallel \langle \emptyset, R_q\rangle.
\]
In other words, the relation \( \parallel^* \) only checks if two conditions are compatible at inner nodes of the tree and forgets information about leaves. Note further that $p \parallel^* q$ and $f_p = f_q$ implies $p \parallel q$. It is now easy to see that $\phi(p) \parallel \phi(q)$ implies $p \parallel^* q$.%
\footnote{The reverse implication is not true, as the reduction is very coarse and forgets a lot of the information contained in $p$ and $q$.} 
    
The set $\F=\mathrm{Fn}(T_\alpha^0, \pre{<\kappa}{\kappa}, {<}\kappa)$ can be construed at the collection of all \( {<}\kappa \)-sized partial assignments of clopen sets to leaves of $T_\alpha$,  and clearly \( |\F| \leq \k \).
Therefore, if we can prove that $\mathrm{Fn}(X, \kappa, {<}\kappa)$ is $\kappa$-linked as witnessed by the equality $\mathrm{Fn}(X, \kappa, {<}\kappa) = \bigcup_{i<\kappa} \bP_i$, then \( \aforc_\alpha(A,B) \) is \( \k \)-linked as well because it can be written as
\[
\aforc_\alpha(A,B) = \bigcup_{i < \kappa, f \in \F} \{p \in \aforc_\alpha(A,B) \mid f_p = f \wedge \phi(p) \in \bP_i\}.
\]

\begin{claim}
If $|Z| \leq 2^\kappa$ and $|W| \leq \kappa$, then $\mathrm{Fn}(Z, W, {<}\kappa)$ is $\kappa$-linked.
\end{claim}

\begin{proof}[Proof of the Claim]
Consider $W$ as a topological space endowed with the discrete topology. In \cite{elser_density_2011}, a generalization of Marczewski's Separability Theorem is proven: if $|I| \leq 2^\kappa$ and $X_i$ has a dense subset of cardinality ${\leq}\kappa$ for each $i \in I$, then the product space $\prod_{i \in I} X_i$, endowed with the ${<}\kappa$-box topology, contains a dense subset of size $\kappa^{<\kappa} = \kappa$.

Let thus $\seq{b_i}{i < \kappa}$ be a dense subset of $\pre{Z}{W}$, seen as the product of \( |Z| \)-many copies of \( W \) and equipped with the ${<}\kappa$-box topology. Since each $f \in \mathrm{Fn}(Z, W, {<}\kappa)$ corresponds to the basic open set $\clopen{f} = \{x \in \pre{Z}{W}\mid f \subseteq x\}$ of \( \pre{Z}{W} \), two conditions $f_1, f_2 \in \mathrm{Fn}(Z, W, {<}\kappa)$ are compatible if and only if $\clopen{f_1} \cap \clopen{f_2} \neq \emptyset$. Therefore we can conclude that \( \mathrm{Fn}(Z, W, {<}\kappa) \) is \( \k \)-linked because it can be written as $\mathrm{Fn}(Z, W, {<}\kappa) = \bigcup_{i < \kappa} D_i$, where for every \( i < \k \) we let 
\[
D_i = \{f \in \mathrm{Fn}(Z, W, {<}\kappa) \mid b_i \in \clopen{f}\}. \qedhere
\]
\end{proof}

This concludes the proof of Lemma~\ref{lem: kappa linked}.
 \end{proof}

Combining Lemmas~\ref{lem: <kappa closure} and~\ref{lem: kappa linked} with Fact~\ref{fact: iteration cc}, we finally get:

\begin{corollary}
Any ${<}\kappa$-supported iteration of $\aforc_\alpha(A,B)$-forcings is ${<}\kappa$-closed and satisfies the $\kappa^+$-c.c.. In particular, no cardinal is collapsed by such an iteration. 
\end{corollary}

\subsection{\texorpdfstring{A model for $\ord_{\k^+}(X) = n$}{A model for ordkappa+(X) = n}}\label{subsec:forcing_results}

The goal of this section is to show that, consistently, there can be even closed sets \( X \subseteq \pre{\k}{\k}\) such that \( \ord_{\k^+}(X) = n \), for any given natural number \( n \geq 2 \).
This will be achieved
 using an iteration of $\alpha$-forcings. 

For \( n > 2 \), the idea is to first force with $\aforc_{n-1}(\emptyset, \emptyset, X)$ to add a \enquote{fresh} $\Piii{0}{n-1}{\k^+}(X)$ subset of $X$, and then successively, using bookkeeping, force with $\aforc_n(\name{A}, X \minus \name{A},X)$ to add a $\Piii{0}{n}{\k^+}$-code for every $\kappa^+$-Borel subset of $X$; by closure under complements of \( \Bor{\k^+}(X) \), this entails \( \Bor{\k^+}(X) = \Deee{0}{n}{\k^+}(X) \), so that $\ord_{\k^+}(X) \leq n$. To prove that $\ord_{\k^+}(X) \geq n$, we will show that the first generic set we added will never become a $\Siii{0}{n-1}{\k^+}$ set in the iteration (Theorem \ref{th: switcheroo}); this is the heart and main difficulty of the construction.
Notice that by Corollary~\ref{cor:proper_for_k+_final_level}, we will automatically get that \( \Deee{0}{n}{\k^+}(X) \setminus (\Siii{0}{n-1}{\k^+}(X) \cup \Piii{0}{n-1}{\k^+}(X)) \neq \emptyset \).

For $n = 2$ the construction is much easier and we may skip the first step of forcing with $\aforc_{n-1}(\emptyset, \emptyset,X)$. The reader interested solely in the result for $n = 2$ may skip ahead up to the proof of Theorem \ref{th: set order to n}. The same result can also be achieved by using Solovay's almost disjoint forcing in place of $\aforc_2$, see \cite{luckeSigma11Definability2012}. 

Recall that we already fixed an arbitrary set \( X \subseteq \pre{\k}{\k}\).
Throughout the rest of this section, let us also fix a ${<}\kappa$-supported forcing iteration $\overline{\aforc} = \seq{\bP_\zeta, \name{\bQ}_\zeta}{\zeta < \zeta^*}$ such that 
\begin{itemizenew}
    \item $\bP_\zeta \Vdash \name{\bQ}_\zeta = \aforc_{\alpha_\zeta}(\name{A}_{\zeta}, \name{B}_{\zeta}, \check{X})$.
    \item $A_0 = B_0 = \emptyset$.
    \item $1 < \alpha_0 < \omega$ and $\alpha_\zeta = \alpha_0 + 1$ for $0 < \zeta < \zeta^*$.
\end{itemizenew}
As usual, $\bP_\zeta$ denotes the partial iteration of all $\name{\bQ}_{\zeta'}$ for $\zeta' < \zeta$. In particular, the forcing $\bP_0$ in the iteration is the trivial forcing $\bP_0 = \{\mathbbm{1}\}$.

A condition $p\in\overline{\aforc}$ is thus a tuple $p=\seq{\name{p}(\zeta)}{\zeta<\zeta^*}$ such that $p \restriction \zeta \Vdash \name{p}(\zeta) \in \aforc_{\alpha_\zeta}(\name{A}_{\zeta}, \name{B}_{\zeta}, \check{X})$ and, furthermore, $|\supp(p)|<\kappa$ for 
\[
\supp(p)=\{\zeta\in \zeta^*\mid p(\zeta)\neq\check{\mathbbm{1}}\}.
\]
For the sake of readability, in what follows we will often (but not always) write \( T(\g) \) instead of \( T_{\a_\g}\). 

While in general a coordinate $p(\zeta)$ can be just a name for a condition in the forcing $\aforc_{\alpha_\zeta}(\name{A}_{\zeta}, \name{B}_{\zeta}, \check{X})$, we may without loss of generality work with conditions of a nicer form.

\begin{definition}
A condition $p \in \overline{\aforc}$ is \markdef{good} if $p(\zeta) = \langle \check{f}_{p(\zeta)}, \check{R}_{p(\zeta)}\rangle$ for each $\zeta <\zeta^*$, i.e.\ all entries of the condition are already decided in the ground model.
\end{definition}

Every good condition can be identified with an element of the product forcing $\prod_{\zeta<\zeta^*} \aforc_{\alpha_\zeta}(\emptyset, \emptyset, X)$.
However,  the converse is not true, as for example if $\name{B}_1$ and $x\in X$ are such that $\aforc_{\alpha_0}(A_0, B_0, X) \not \Vdash \check{x}\in \name{B}_1$ and $\aforc_{\alpha_0}(A_0, B_0, X)\not \Vdash \check{x}\notin \name{B}_1$, then $p$ defined by 
\begin{equation}\label{eq:not_valid_condition}
p(\zeta)=
\begin{cases}
(\emptyset, \{\langle\emptyset,\check{x} \rangle\}) \qquad & \text{ if } \zeta=1, \\
\mathbbm{1} \qquad & \text{ otherwise} 
\end{cases}
\end{equation}
does not satisfy $p(0)\Vdash p(1)\in \aforc_{\alpha_1}(\name{A}_1, \name{B}_1, \check{X})$, and thus can not be a condition in $\overline{\aforc}$.

One of the key points in several proofs of this section will be to check that a sequence $p=\seq{p(\zeta)}{\zeta<\zeta^*} $ we defined explicitly is actually a condition in our forcing.
Notice that requiring $p \restriction \zeta \Vdash \name{p}(\zeta) \in \aforc_{\alpha_\zeta}(\name{A}_{\zeta}, \name{B}_{\zeta}, \check{X})$ amounts to checking that $p \restriction \zeta \Vdash \check{x} \notin \name{B}_\zeta$ for any $\langle \emptyset,x \rangle\in R_{p(\zeta)}$, and $p \restriction \zeta \Vdash \check{x} \notin \name{A}_\zeta$ for any $\langle t,x \rangle\in R_{p(\zeta)}$ with $t \in\succ_{T(\g)}(\emptyset)$.

A standard argument shows that the set of good conditions is dense in  $\overline{\aforc}$.

\begin{lemma}
    For every $p \in \overline{\aforc}$ there is $q \in \overline{\aforc}$ such that $q\leq p$ and $q$ is good.
\end{lemma}

\begin{proof}
We prove by induction on \( \g \) that good conditions are dense in $\bP_\zeta$, for each $\zeta \leq \zeta^*$. For $\zeta = 0$ there is nothing to prove. In the successor step of the induction, we assume that good conditions are dense in $\bP_\zeta$, 
and consider an arbitrary $p \in \bP_{\zeta+1}$. We know that there are $\bP_\zeta$-names $\name{f}$ and $\name{R}$ such that \begin{multline*}
p \restriction \zeta \Vdash \name{p}(\zeta) = \langle \name{f}, \name{R} \rangle \in \aforc_{\alpha_\zeta}(\name{A},\name{B}, X) \ \wedge\  
\name{f} \colon T^0_{\alpha_\zeta} \to \pre{<\kappa}{\kappa} \text{ is a partial function} \wedge {} \\
|\name{f}| < \kappa \ \wedge\  \name{R} \subseteq T^0_{\alpha_\zeta} \times \check{X} \ \wedge\  |\name{R}| < \kappa.
\end{multline*}
Since the iteration is ${<}\kappa$-closed, both $\name{f}$ and $\name{R}$ can be decided as ground model objects $f$ and $R$, respectively, by a condition $q \leq p \restriction \zeta$. 
By inductive hypothesis, we can assume that $q$ is good, hence $q \conc \langle \check{f}, \check{R} \rangle \leq p$ is good as well.

Suppose now that $\zeta \leq \zeta^*$ is a limit. If $\cof(\zeta) \geq \kappa$, every condition in $\bP_\zeta$ is already in $\bP_\beta$ for some $\beta < \zeta$, hence we are done by the inductive hypothesis. 
If instead $\cof(\zeta) < \kappa$, let $\seq{\zeta_i}{i < \cof(\zeta)}$ be a cofinal sequence in $\zeta$ and $p \in \bP_\zeta$. We inductively define a sequence $\seq{q_i}{i < \cof(\zeta)}$, where
for each \( i < j < \cof(\g) \)
\begin{itemizenew}
\item 
$q_i \in \bP_{\zeta_i}$ is good and $q_i \leq p \restriction \zeta_i$;
\item 
$q_j \restriction \zeta_i \leq q_i$.
\end{itemizenew}
(In limit steps of this construction, we make use of the ${<}\kappa$-closure of $\bP_\zeta$.) 
Let $q = \bigcup_{i < \cof(\zeta)} q_i$ be the pointwise union of these conditions, it is easy to see that $q \leq p$ and $q$ is good, as desired.
\end{proof}

Therefore, we shall tacitly assume that all conditions appearing from now on are in the dense subset of good conditions.

\begin{definition}
For a set $H \subseteq X$ define the \markdef{rank function}
\[
\crank{p}{H} =  \sup \left\{\rank_{T(\zeta)}(t) \bigm | \zeta < \zeta^* \wedge t \in T(\g) \wedge  \langle t, x \rangle \in R_{p(\zeta)} \text{ for some } x \in X \minus H \right\} .
\]
\end{definition}

 Notice that for every $H \subseteq X$ and $p\in \overline{\aforc}$, we have $\crank{p}{H}\leq \alpha_0 + 1$.
Also, $\crank{p}{H}=0$ if and only if every $x$ appearing in $\bigcup_{\zeta < \zeta^*} R_{p(\zeta)}$ is in $H$.

\begin{definition}
Given $K \subseteq \zeta^*$ and $H \subseteq X$, let 
\[
\overline{\aforc}_{H,K,{\leq} \beta}=\{q\in \overline{\aforc}\mid \supp(q) \subseteq K, \crank{q}{H} \leq \beta \}.
\]
The set $\overline{\aforc}_{H,K,{<} \beta} \subseteq \overline{\aforc}$ is defined analogously.
\end{definition}

We also let $\overline{\aforc}_{H,K}=\overline{\aforc}_{H,K,{\leq} 0}$ be the set of conditions \( p \) with \( \supp(p) \subseteq K \) and $\crank{p}{H}=0$.
 
 For every $p\in \overline{\aforc}$, $H \subseteq X$, $K \subseteq \zeta^*$, and $\beta\leq \alpha_0$, define the restriction of $p$ to $(H,K,{\leq}\beta)$ as the tuple $p\restriction_{H,K,{\leq}\beta}=\seq{p\restriction_{H,K,{\leq}\beta}(\zeta)}{\zeta<\zeta^*}$, where
\[
p\restriction_{H,K,{\leq}\beta}(\zeta) = 
\begin{cases}
\langle f_{p(\zeta)}, R_{p(\zeta)}\restriction_{H,{\leq}\beta} \rangle & \text{ if } \zeta \in K \\
\mathbbm{1} & \text{ if } \zeta \notin K,
\end{cases}
\]
and for each $\zeta \in K$ 
\[
R_{p(\zeta)}\restriction_{H,{\leq}\beta} = \{ \langle t, x\rangle \in R_{p(\zeta)} \mid x \in H \vee \rank_{T(\g)}(t) \leq \beta\}.
\]
Notice that in general $p\restriction_{H,K,{\leq}\beta}$ is not a condition in $\overline{\aforc}$ (even if $p$ is), since by weakening an initial segment of the condition we can no longer guarantee that $\left(p \restriction_{H,K,{\leq} \beta}  \right) \restriction \zeta$ decides any statement of the form $\check{x} \notin \name{A}_\zeta$. However, if $p\restriction_{H,K,{\leq}\beta}$ is a condition, then it is also in $\overline{\aforc}_{H,K,{\leq} \beta}$.

\begin{definition}
We say that $p\in \overline{\aforc}$ is \markdef{$(H,K)$-restrictable} if for any $\beta < \alpha_0$ we have that $p\restriction_{H,K,{\leq}\beta}$ is a well-defined condition in $\overline{\aforc}$.
\end{definition}

We will just say that $p$ is restrictable when $H$ and $K$ are clear from the context.
Recall that we have restricted our attention to good conditions. Since the restriction of any good restrictable condition remains good, this process may be applied freely to every condition under consideration.

\begin{definition} \label{def: appropriate}
A pair of sets $H \subseteq X$ and $K \subseteq \zeta^*$ is said to be \markdef{appropriate} if the set
\[
 \{p\in \overline{\aforc}\mid p \text{ is $(H,K)$-restrictable}\}
\]
is dense in $ \overline{\aforc}$.
Furthermore, given a $\overline{\aforc}$-name $\name{\tau}$  for an element of \( \pre{\k}{\k} \), we say that the pair $(H,K)$ is \markdef{appropriate for $\name{\tau}$} if, additionally,  for each $i < \kappa$ the set
\begin{equation}\label{eq:predense}
\{p \in \overline{\aforc}_{H,K} \mid p \text{ decides } \name{\tau} \restriction i\}
\end{equation} 
is predense%
\footnote{Recall that a set $D$ is predense if every condition in the forcing is compatible with a condition from $D$.} in \( \overline{\aforc} \).
\end{definition}

\begin{lemma} \label{lem: h,k exist}
For every $\overline{\aforc}$-name $\name{\tau}$ for an element of \( \pre{\k}{\k}\) and all sets $H' \subseteq X$ and $K' \subseteq \zeta^*$ of size at most $\kappa$, there are $H \supseteq H'$ and $K \supseteq K'$ of size at most $\kappa$ such that the pair \( (H,K) \) is appropriate for $\name{\tau}$.
\end{lemma}

\begin{proof}
We recursively define an increasing sequence $\seq{H_i, K_i}{i\leq\kappa}$ of sets of size at most $\kappa$ as follows.

If $i = 0$, for each $j < \kappa$ we choose a maximal antichain $C_j$ that decides $\name{\tau} \restriction j$. Keeping in mind that $\overline{\aforc}$ has the $\kappa^+$-c.c., we can find $H_0 \supseteq H'$ and $K_0 \supseteq K'$ of size at most $\kappa$ such that $\crank{p}{H_0} = 0$ and $\supp(p) \subseteq K_0$ for each $p \in \bigcup_{j < \kappa} C_j$.

For the successor step, suppose that \( H_i \) and \( K_i \) have already been defined.
For each $Y \in [H_i]^{<\kappa}$ and $Z \in [K_i]^{<\kappa}$, let
\[
D_{Y, Z} = \{p \in \overline{\aforc} \mid \forall y \in Y \, \forall \zeta \in Z \, (p \restriction \zeta \text{ decides } \check{y} \in \name{A}_\zeta \text{ and } \check{y} \in \name{B}_\zeta ) \}.
\]
Notice that $D_{Y, Z}$ is dense, since \( \overline{\aforc}\) is ${<}\kappa$-closed. Choose now maximal antichains $C_{Y, Z} \subseteq D_{Y, Z}$, for \( Y \) and \( Z \) as above. 
Collect all $x$ and $\zeta$ appearing in $\bigcup_{\zeta < \zeta^*} R_{p(\zeta)}$ and $\supp(p)$, respectively, where \( p \) varies in the set
\[
\bigcup \left\{ C_{Y, Z} \mid Y \in [H_i]^{<\kappa} \wedge Z \in [K_i]^{<\kappa} \right \},
\]
and include them in $H_{i+1}, K_{i+1}$.

For \( i \) limit,  we just set $H_i = \bigcup_{j < i} H_j$ and $K_i = \bigcup_{j < i} K_j$.

We claim that $H=H_\kappa$ and $K=K_\kappa$ are as desired. First, notice that $|H|, |K| \leq \kappa$, and that we took care of the predensity of the sets from~\eqref{eq:predense}
   in the first step of the construction. It remains to see that the pair $(H,K)$ is appropriate.
	
Let $p$ be a condition in $\overline{\aforc}$, and let $\seq{(\zeta_i,t_i,x_i)}{i < \delta}$ with $\delta < \kappa$ be an enumeration of all triples $(\zeta, t, x)$ with $\langle t, x\rangle \in R_{p(\zeta)}$, $\zeta \in \supp(p) \cap K$, $x \in H$, and $t \in \{\emptyset\} \cup \succ_{T(\g)}(\emptyset)$. By construction of $H$ and $K$, we can find a condition $\bar{p} \in \overline{\aforc}_{H,K}$ with $\bar{p} \parallel p$ such that \( \bar{p} \restriction \zeta_i \) decides the statements \enquote{$\check{x}_i \in \name{A}_{\zeta_i}$} and \enquote{$\check{x}_i \in \name{B}_{\zeta_i}$}, for all $i < \delta$. Let $q = p \cup \bar{p}$ be the pointwise union of \( p \) and \( \bar{p} \).
 It is easy to see that $q$ is a condition such that $q \leq p, \bar{p}$. We claim that $q$ is \( (H,K) \)-restrictable, as wanted.

Fix $\beta < \alpha_0$, and let $q'=q\restriction_{H,K,{\leq}\beta}$. Note that since $\bar{p} \in \overline{\aforc}_{H,K}$, we have $q' \leq \bar{p}$.
We prove by induction that $q' \restriction \zeta$ is a condition in $\bP_\zeta$, for each $\zeta \leq \zeta^*$. For $\zeta$ limit and $\zeta=0$, this is trivial. 
Assume now that the statement holds up to some $\zeta$, and let us show that it is true for $\zeta + 1$ as well. If $\zeta \notin K$, we know that $q'(\zeta) = \mathbbm{1}$ and there is nothing to do. 
If $\zeta \in K$, instead, we need to check that $q' \restriction \zeta \Vdash \check{x}\notin \name{B}_\zeta$ for any \( x \in X \) such that $\langle \emptyset,x \rangle\in R_{q'(\zeta)}$ and $q' \restriction \zeta \Vdash \check{x}\notin \name{A}_\zeta$ for any \( x \in X \) such that  $\langle t,x \rangle\in R_{q'(\zeta)}$ for some $t\in\succ_{T(\g)}(\emptyset)$.
	
Consider $\langle \emptyset, x \rangle \in R_{q'(\zeta)}$. If $\langle \emptyset, x \rangle \in R_{\bar{p}(\zeta)}$, then $\bar{p} \restriction \zeta \Vdash \check{x} \notin \name{B}_{{\zeta}}$, so $q' \leq \bar{p}$ forces this as well. Otherwise, $\langle \emptyset, x \rangle \in R_{p(\zeta)}$. This implies $x \in H$, since 
\begin{equation} \label{eq: star}
\rank_{T(\g)}(\emptyset) = \alpha_\zeta = \alpha_0 + 1 > \beta.  \end{equation}
 Hence $(\zeta, \emptyset, x)$ is one of the triples we took care of at the beginning of the proof, and $\bar{p} \restriction \zeta$ decides \enquote{$\check{x} \in \name{B}_{\zeta}$}. But since $\bar{p}$ is compatible with $p$, it must be the case that $\bar{p} \restriction \zeta \Vdash \check{x} \notin \name{B}_{\zeta}$, and so $q' \restriction \zeta \Vdash \check{x} \notin \name{B}_{\zeta}$ as wanted.
Any pair $\langle t, x\rangle \in R_{q'(\zeta)}$ with $t \in \succ_{T(\g)}(\emptyset)$ can be treated analogously, with~\eqref{eq: star} now reading \enquote{$\rank_{T(\g)}(t) = \alpha_\zeta - 1 = \alpha_0 > \beta$}. 

Thus we can conclude that $q' \restriction \zeta \Vdash \name{q}'(\zeta) \in \name{\bQ}_\zeta$ and we are done. 
\end{proof}

\begin{lemma} \label{lem: projection}
Let $0 < \beta < \alpha_0$, $H \subseteq X$, $K \subseteq \zeta^*$, and $p \in \overline{\aforc}$ be $(H,K)$-restrictable. Then for every $r\in \overline{\aforc}_{H,K,{<}\beta}$, we have that if \( {p\restriction_{H,K,\leq \beta}} \parallel r \), then \( p \parallel r\).
\end{lemma}

\begin{proof}
Let $q = p\restriction_{H,K,{\leq}\beta}$. Fix $r \in \overline{\aforc}_{H,K, <\beta}$, and assume, towards a contradiction, that $q \parallel r$, but $p \perp r$. As in the previous lemma, let $\bar{r} = p \cup r$ be the pointwise union of \( p \) and \( r \), and notice that $\bar{r}$ is not a valid condition because we assumed \( p \perp r \).
Let $\zeta < \zeta^*$ be such that $\bar{r} \restriction \zeta \in \bP_\zeta$ but $\bar{r} \restriction \zeta \not \Vdash \name{\bar{r}}(\zeta) \in \name{\bQ}_\zeta$, so that in particular $\zeta \in K$. This state of affairs can only happen for one of the following reasons, each of which we are going to exclude:
\begin{itemizenew}
\item 
$f_{p(\zeta)} \perp f_{r(\zeta)}$. This cannot happen, since $f_{p(\zeta)} = f_{q(\zeta)}$ and $q \parallel r$.
\item 
There are \( x \in X\), \( t \in T(\g)\), and \( t' \in \succ_{T(\g)}(t) \cap \dom(f_{r(\zeta)})\) such that \( \rank_{T(\g)}(t) = 1 \) (so that \( \rank_{T(\g)}(t') = 0 \)), $\langle t, x \rangle \in R_{p(\zeta)}$, and 
 $x \in \clopen{f_{r(\zeta)}(t')}$. This is not possible: since $\rank_{T(\g)}(t) = 1$ and $\beta > 0$, we must have $\langle t, x \rangle \in R_{q(\zeta)}$.
\item 
There are \( x \in X\), \( t \in T(\g)\), and \( t' \in \succ_{T(\g)}(t) \cap \dom(f_{p(\zeta)}) \) such that \( \rank_{T(\g)}(t) = 1 \) (so that \( \rank_{T(\g)}(t') = 0 \)),
 $\langle t, x \rangle \in R_{r(\zeta)}$, and 
 $x \in \clopen{f_{p(\zeta)}(t')}$. This is not possible, since $f_{q(\zeta)} = f_{p(\zeta)}$.
\item 
There are \( x \in X\), $t \in T(\g)$ and \( t' \in \succ_{T(\g)}(t) \) such that $\langle t, x \rangle \in R_{r(\zeta)}$ and $\langle t', x \rangle \in R_{p(\zeta)}$. Since $\rank_{T(\g)}(t') < \rank_{T(\g)}(t)$, we know that $\langle t, x \rangle \in R_{q(\zeta)}$, which is a contradiction.
\item 
There are \( x \in X \), $t \in T(\g)$, and \( t' \in \succ_{T(\g)}(t) \) such that $\langle t, x \rangle \in R_{p(\zeta)}$ and $\langle t', x \rangle \in R_{r(\zeta)}$. If $x \notin H$, it follows that $\beta > \rank_{T(\g)}(t')$, and since $\rank_{T(\g)}(t) = \rank_{T(\g)}(t') + 1$ (because of the specific definition of \( T(\g) \)), we have $\langle t, x \rangle \in R_{q(\zeta)}$. If instead $x \in H$, we also immediately get $\langle t, x \rangle \in R_{q(\zeta)}$. This contradicts $q \parallel r$. \qedhere
\end{itemizenew}
\end{proof}

\begin{theorem} \label{th: rank argument}
Suppose that $\alpha \leq \alpha_0$, $t \in T_\alpha^{>0}$, $\name{f}$ is a $\overline{\aforc}$-name for an assignment $T_\alpha^0 \to \pre{<\kappa}{\kappa}$, $x \in X$, and $(H,K)$ is an appropriate pair for (a coding of) $\name{f}$. Then, for every $p \in \overline{\aforc}$ such that $p \Vdash \check{x} \notin \mathcal{I}^{\name{f}}(t)$ there exists a $q \in \overline{\aforc}_{H,K, {<\rank_{T_\a}(t)}}$ such that \( q \parallel p \) and $q \Vdash \check{x} \notin \mathcal{I}^{\name{f}}(t)$.
\end{theorem}

\begin{proof}
We proceed by induction on $\rank_{T_\a}(t)$. 
Let 
\[
p \Vdash \check{x} \notin \mathcal{I}^{\name{f}}(t) = \bigcap_{t' \in \succ_{T_\a}{t}} X \setminus \mathcal{I}^{\name{f}}(t'),
\]  
and fix $p'\leq p$ and $t' \in \succ_{T_\a}(t)$ such that $p' \Vdash \check{x} \in \mathcal{I}^{\name{f}}(t')$.
    
Suppose first that $\rank_{T_\a}(t)=1$.
Since $(H,K)$ is appropriate for $\name{f}$, we can find $q \in \overline{\aforc}_{H,K} = \overline{\aforc}_{H,K, {<}1}$ such that $q \parallel p'$ and $q$ decides $\name{f}(t')$ as $s \in \pre{<\kappa}{\kappa}$. Since $q$ is compatible with $p'$, this implies that $x \in \clopen{s}$, and thus $q \Vdash \check{x} \in \mathcal{I}^{\name{f}}(t') \subseteq X \setminus \mathcal{I}^{\name{f}}(t)$.
    
Assume now that $\beta = \rank_{T_\a}(t) > 1$. 
Since $(H,K)$ is appropriate, we may choose an $(H,K)$-restrictable $q\leq p'$. We claim that ${q\restriction_{H,K,{\leq}\beta - 1}} \Vdash \check{x} \in \mathcal{I}^{\name{f}}(t')$. 
If this were not the case, we could find an $r \leq q\restriction_{H,K,{\leq}\beta - 1}$ with $r \Vdash \check{x} \notin \mathcal{I}^{\name{f}}(t')$.  The inductive hypothesis now yields a $\bar{r}\in \overline{\aforc}_{H,K,{<\beta - 1}}$ such that $\bar{r} \parallel r$ and $\bar{r} \Vdash \check{x} \notin \mathcal{I}^{\name{f}}(t')$. But $\bar{r} \parallel q\restriction_{H,K,{\leq}\beta - 1}$ implies $\bar{r} \parallel q$  by Lemma \ref{lem: projection}, and hence we reach a contradiction.
\end{proof}

Note that for any name $\name{J}$ for a set in $\Siii{0}{\a}{\k^+}(X)$ there exists an $\name{f} \colon T_\alpha^0 \to \pre{<\kappa}{\kappa}$ as in the above lemma with $\mathbbm{1} \Vdash \name{J} = \check{X} \setminus \mathcal{I}^{\name{f}}(\emptyset)$.

The following argument is the heart of the proof; it is an instance of a type of argument Miller affectionately dubbed \enquote{the old switcheroo}.

\begin{theorem} \label{th: switcheroo}
Let $G$ be $\overline{\aforc}$-generic, and let $G^0$ be the projection of $G$ onto the $0$-th coordinate, so that \( G^0 \) is \( \aforc_{\a_0}(\emptyset,\emptyset,X) \)-generic. 
If $|X| > \kappa$, then 
\[
V[G] \models G^0_\emptyset \in \Piii{0}{\alpha_0}{\k^+}(X) \minus \Siii{0}{\alpha_0}{\k^+}(X).
\]
\end{theorem}

\begin{proof}
The fact that \( V[G] \models G^0_\emptyset \in \Piii{0}{\a_0}{\k^+}(X) \) easily follows from the fact that \( \langle T_{\a_0}, f_{G^0} \rangle \) is a \( \Piii{0}{\a_0}{\k^+} \)-code for \( G^0_\emptyset \).

Towards a contradiction, assume that \( V[G] \models G^0_\emptyset \in \Siii{0}{\a_0}{\k^+}(X) \).
Fix a condition $p \in \overline{\aforc}$ and a name $\name{J}$ for a set in $\Siii{0}{\alpha_0}{\k^+}(X)$ such that 
\[
p \Vdash G^0_\emptyset = \name{J}.
\]
Then there exists a $\overline{\aforc}$-name $\name{f}$ for an assignment $T_\alpha^0 \to \pre{<\kappa}{\kappa}$ such that $p \Vdash \check{X}\setminus \name{J}=\mathcal{I}^{\name{f}}(\emptyset)$.
Lemma \ref{lem: h,k exist}, find an appropriate pair $(H,K)$ for (a coding of) $\name{f}$ such that $\crank{p}{H} = 0$ and $\supp(p) \subseteq K$. For the rest of the proof, fix a $ y \in X \minus H$. 
 
Define $p'(0) = \langle f_{p(0)}, R_{p(0)} \cup \{ \langle\emptyset,  y\rangle \} \rangle$ and $p'(\zeta) = p(\zeta)$ for $\zeta > 0$, and notice that $p'\in\overline{\aforc}$ is a condition since $A_0=B_0=\emptyset$. Since $p' \Vdash \check{y} \in G^0_\emptyset = \name{J}$ (i.e.\ $p' \Vdash \check{y}\notin\mathcal{I}^{\name{f}}(\emptyset)$), by Theorem~\ref{th: rank argument} there exists a $q \parallel p'$ such that $\crank{q}{H} < \alpha_0$, $\supp(q) \subseteq K$, and $q \Vdash \check{y} \in \name{J}$. As $q \parallel p$ and $p \in \overline{\aforc}_{H,K}$, we may without loss of generality assume $q \leq p$ (by passing to the pointwise union $q \cup p$ if necessary). Thus we have $q \Vdash G^0_\emptyset = \name{J}$, too.
    
Since $\crank{q}{H} < \alpha_0$, we know that $\langle \emptyset, y \rangle \notin R_{q(0)}$, and thus there is $t \in \succ_{T(0)}(\emptyset)$ such that if $q'$ is defined by $q'(0) = \langle f_{q(0)}, R_{q(0)} \cup \{ \langle t, y\rangle \} \rangle$ and $q'(\zeta) = q(\zeta)$ for $\zeta > 0$, then \( q' \) is a condition.  Note that $q' \Vdash \check{y} \notin G^0_\emptyset$, since $\langle t, y\rangle\in R_{q'(0)}$. However, we also have $q' \Vdash \check{y} \in \name{J} = G^0_\emptyset$ because $q'\leq p$. This is a contradiction, which concludes our proof. 
\end{proof}

We are now ready to state the main results of this section.

\begin{theorem} \label{th: set order to n}
Let $X \subseteq \pre{\kappa}{\kappa}$ be such that $|X| > \kappa$, and let $1 < n < \omega$. Then there exists a ${<}\kappa$-closed, $\kappa^+$-c.c.\ forcing notion \( \bP \) such that
$\mathbbm{1}_{\bP} \Vdash \ord_{\k^+}(\check{X}) = n$.
\end{theorem}

\begin{proof}
Let $\lambda = 2^\kappa$.
For $n > 2$, consider the ${<}\kappa$-supported forcing iteration $\bP = \langle \bP_\zeta, \name{\bQ}_\zeta \mid \zeta < \lambda \rangle$ with the first forcing $\bQ_0 = \aforc_{n-1}(\emptyset, \emptyset,X)$ and $\bP_\zeta \Vdash \name{\bQ}_\zeta = \aforc_n(\name{B}_{\zeta}, \check{X} \minus \name{B}_{\zeta}, \check{X})$ for $\zeta > 0$. Since each forcing has the $\kappa^+$-c.c.\ and has size $2^\kappa$, the iteration preserves the value of $2^\kappa$. Hence, using a bookkeeping argument, we can ensure that for every \( \bP \)-generic \( G \), the family of names $\name{B}_\zeta$ satisfies
\[
V[G] \models \forall B \in \Bor{\k^+}(X)\, \exists \zeta < \l \, (B = \name{B}_{\zeta}).
\]
By Theorem~\ref{th: switcheroo}, we know that $V[G] \models \ord_{\k^+}(X) \geq n$, and the rest of the iteration ensures $V[G] \models \ord_{\k^+}(X) \leq n$ (see also Corollary \ref{cor: collapse A}).

For $n=2$, consider instead the ${<}\kappa$-supported forcing iteration $\bP = \langle \bP_\zeta, \name{\bQ}_\zeta \mid \zeta < \lambda \rangle$ with $\bP_\zeta \Vdash \name{\bQ}_\zeta = \aforc_2(\name{B}_{\zeta}, \check{X} \minus \name{B}_{\zeta}, \check{X})$ for all $\zeta$. Using a bookkeeping argument, ensure that for every \( \bP \)-generic \( G \)
\[
V[G] \models \forall B \in \Bor{\k^+}(X)\, \exists \zeta < \l \, (B = \name{B}_{\zeta}).
\]
Then we get that $V[G] \models \ord_{\k^+}(X) \leq 2$. On the other hand, every Hausdorff space \( Y \) with $\ord_{\k^+}(Y) = 1$ is discrete, and since every discrete subspace of $\pre{\kappa}{\kappa}$ has size at most $\kappa$, we also get $V[G] \models \ord_{\k^+}(X) \geq 2$ because we assumed \( |X| > \k \) (and cardinals are not collapsed by \( \bP \)).
\end{proof}

We stress the fact that the conclusion of Theorem~\ref{th: set order to n} applies to the ground model space $X$, i.e., applying it to $X = \pre{\kappa}{\kappa}$, we get a forcing extension with $\ord_{\kappa^+}((\pre{\kappa}{\kappa})^V)=\ord_{\kappa^+}(\pre{\kappa}{\kappa} \cap V) = n$.

Theorem~\ref{th: set order to n} can of course be applied, in particular, to \( \k^+ \)-Borel sets \( X \subseteq \pre{ \k}{\k}\). 
In this case, once we fix a \( \k^+ \)-Borel code for it, the \( \k^+ \)-Borel set \( X \) can also be naturally re-interpreted in \( V[G] \) as a set \( X^{V[G]}\): it is thus natural to ask whether the \( \k^+ \)-Borel hierarchy is collapsed also on the latter space. The next corollary shows that under suitable assumptions on \( X \), no new element is added to the \( \k^+ \)-Borel set at hand, so that \( \ord_{\k^+}(X^{V[G]}) = n \) as well, for the chosen \( 1 < n < \o \).

\begin{corollary}\label{cor:final_7}
Let $X \in \Bor{\k^+}( \pre{\kappa}{\kappa} )$ be a \( \kappa \)-thin set such that $|X| > \kappa$, and let $1 < n < \omega$. Then there is a ${<}\kappa$-closed, $\kappa^+$-c.c.\ forcing extension $V[G]$ of \( V \) such that $V[G] \models {X^{V[G]} = X}  \wedge {\ord_{\k^+}(X) = n}$.
\end{corollary}

\begin{proof}
Fix a \( \k^+ \)-Borel code for $X$.
By Theorem~\ref{th: set order to n}, there exists a ${<}\kappa$-closed, $\kappa^+$-c.c.\ forcing extension $V[G]$ of $V$ such that $V[G] \models \ord_{\k^+}(X^{V})=n$. Since $X$ is $\kappa$-thin and $\kappa^+$-Borel, Corollary~\ref{cor: thin borel set eq} applies to it, showing that the interpretation of the chosen $\kappa^+$-Borel code for $X$ is unchanged in $V[G]$. Therefore, we also have $V[G] \models X^{V[G]} = X^{V}$.
\end{proof}

Since the complexity of both \( X \) and \( X^{V[G]} \) is determined by the very same \( \k^+ \)-Borel code, the side effect of the equality \( X = X^{V[G]} \) is that \( X \) will maintain its \( \k^+ \)-Borel complexity in the forcing extension \( V[G] \); in particular, if \( X \) were closed in the ground model \( V \), then it will stay closed also in \( V[G] \). 
By \cite[Theorem~2.20]{AgosMottoSchlichtRegular}, this also shows that, consistently, there are $\GG$-Polish spaces (for some totally ordered Abelian group $\GG$ of degree $\kappa$) whose order attains intermediate values strictly between \( 2 \) and \( \k^+ \).

Another natural  concern is whether for a given space \( X \) with \( |X| > \k \),%
\footnote{Obviously, \( \Bor{\k^+}(X) =\powset(X) \) for every Hausdorff space \( X \) with \( |X| \leq \k \).}
the collection \( \Bor{\k^+}(X) \) of \( \k^+ \)-Borel sets coincide with the whole powerset \( \powset(X) \) of \( X \). This is often not the case, for simple cardinality reasons.
Indeed, if $\weight(X) \leq \k$ then $|\Bor{\k^+}(X)|\leq 2^\k$ by~\cite[Lemma 2.3]{AMR19}. Hence, if we further have that $|X|=\l$ for some $\l$ with $2^\l > 2^\k$, then \(\Bor{\k^+}(X) \neq \powset(X)\); in particular, this applies to any space \( X \) with \( \weight(X) \leq \k \) and $|X| = 2^\k$, like the generalized Baire space \( \pre{\k}{\k} \). 
However, the situation might be different for spaces \( X \) with \( \k < |X| < 2^\k \) and $2^{|X|} = 2^{\kappa}$.

\begin{prop}
Consistently, for a regular cardinal $\kappa = \kappa^{<\kappa}$ there is a closed subspace $X\subseteq \pre{\k}{\k}$ of size greater than $\k$ and such that $\Bor{\k^+}(X) = \powset(X)$.    
\end{prop}

\begin{proof}
    Start with a ground model $V$ in which a $\kappa$-thin closed subspace $X \subseteq \pre{\kappa}{\kappa}$ of size ${>}\kappa$ exists (c.f.\ Section~\ref{subsec: thin closed set}). Follow the proof of Theorem~\ref{th: set order to n}, using instead an iteration of length $2^{|X|}$ so that a bookkeeping argument can arrange that \emph{every} subset of $X$ appears as a $\name{B}_\zeta$ along the iteration, i.e.
    \[
    V[G] \models \Bor{\k^+}(X) = \powset(X).
    \]
    Since $X$ was $\kappa$-thin in $V$, we also get that $V[G] \models X^{V[G]} = X^V$, hence $X$ remains closed in the extension.

    Note that even if $2^{|X|} > 2^\kappa$ in the ground model, the length of this iteration will ensure that $2^{|X|} \leq 2^\kappa$ and $|X| < 2^\kappa$ holds in the extension.
\end{proof}

We conclude this section by observing that our results also provide some information related to Question 8.1 from~\cite{MRP25}. 
Indeed,~\cite[Corollary 6.4]{MRP25} shows that when \( \k \) is regular, then the closure under \( \k \)-limits of the class of continuous functions \( \mathcal{M}_1(X,Y) \), under suitable assumptions on \( X \) and \( Y \), coincides with \( \bigcup_{1 \leq n < \omega} \mathcal{M}_n(X,Y) \). Combined with Proposition~\ref{prop:stratificationofBorelfunctions} and Corollary~\ref{cor:final_7}, this gives consistent examples of closed spaces \( X \subseteq \pre{\k}{\k}\) with \( |X| > \k \) such that all \( \k^+ \)-Borel measurable functions on \( X \) can be generated from the continuous functions through \( \k \)-limits.
Thanks to \cite[Theorem~2.20]{AgosMottoSchlichtRegular}, we obtain the following (see also \cite[Definition~2.5]{AgosMottoSchlichtRegular} for the definition of $\GG$-Polish space).

\begin{corollary}
Consistently, for a regular cardinal $\kappa = \kappa^{<\kappa}$ and a totally ordered Abelian group $\GG$ of degree $\kappa$, there is a \( \GG \)-Polish space \( X \) (equivalently: $X\subseteq \pre{\k}{\k}$ closed) of size greater than $\k$ such that for every spherically complete \( \GG \)-metrizable space \( Y \), such as \( Y = \pre{\k}{\k} \) or \( Y = \pre{\k}{2}\), the collection of all \( \k^+ \)-Borel measurable functions from \( X \) to \( Y \) coincides with the closure of continuous functions under \( \k \)-limits.
\end{corollary}

\printbibliography

\end{document}